\documentclass[11pt,reqno]{amsart}
\usepackage{hyperref}
\usepackage{geometry}
\usepackage{enumitem}
\usepackage{amsmath}
\usepackage{amssymb}
\usepackage{amsthm}
\usepackage{amsrefs}
\usepackage{amsfonts}
\usepackage[dvipsnames]{xcolor}
\usepackage{mathtools}
\usepackage{stackengine}
\usepackage{verbatim}
\usepackage{tikz-cd} 
\usepackage{mathrsfs}
\usepackage{bbm}
\usepackage{float}

\usepackage{color}
\definecolor{red}{rgb}{1,0,0}
\def\red{\color{red}}
\definecolor{green}{rgb}{0.0,0.8,0.0}

\definecolor{blue}{rgb}{0,0,1}

\definecolor{black}{rgb}{0,0,0}

\definecolor{grey}{rgb}{0.333,0.333,0.333}

\definecolor{lightgrey}{rgb}{0.666,0.666,0.666}

\definecolor{white}{rgb}{1,1,1}

\numberwithin{equation}{section}

\makeatletter
\@namedef{subjclassname@2020}{%
  \textup{2020} MSC}
\makeatother

\hypersetup{
	colorlinks   = true, 
	urlcolor     = blue, 
	linkcolor    = purple, 
	citecolor   = blue 
}

\def\XXint#1#2#3{{\setbox0=\hbox{$#1{#2#3}{\int}$ }
\vcenter{\hbox{$#2#3$ }}\kern-.6\wd0}}

\makeatletter
\newcommand*{\rom}[1]{\expandafter\@slowromancap\romannumeral #1@}

\newcommand{\ind}{\protect\raisebox{2pt}{$\chi$}}

\newcommand{\SL}{\operatorname{SL}}

\newcommand{\X}{\mathcal{X}}
\newcommand{\e}{\epsilon}
\newcommand{\cD}{\mathcal{D}}
\newcommand{\cF}{\mathcal{F}}
\newcommand{\N}{\mathbb{N}}
\newcommand{\R}{\mathbb{R}}
\newcommand{\Z}{\mathbb{Z}}
\newcommand{\bfe}{\mathbf{e}}

\newcommand{\Prob}{\text{Prob}}
\newcommand{\supp}{\text{supp}}

\newcommand{\id}{\mathbbm{1}}
\newcommand{\zer}{{0}}
\newcommand{\cE}{\mathscr{E}}

\newcommand{\cS}{\mathcal{S}}
\newcommand{\cG}{\mathcal{G}}
\newcommand{\tnu}{\widetilde{\nu}}

\newcommand{\Q}{\mathbb{Q}}

\newcommand{\cP}{\mathcal{P}}
\newcommand{\cB}{\mathcal{B}}
\newcommand{\vol}{\operatorname{vol}}
\newcommand{\cA}{\mathcal{A}}
\newcommand{\T}{\mathbb{T}}

\newcommand{\zA}{\mathcal{Z}(\cA)}
\newcommand{\cZ}{\mathcal{Z}}

\newcommand{\tXn}{\widetilde{\X}_{n+1}}
\newcommand{\tHn}{\widetilde{H}_{n+1}}

\begin{document}
\theoremstyle{plain}
\newtheorem{thm}{Theorem}[section]
\newtheorem{lem}[thm]{Lemma}
\newtheorem{prop}[thm]{Proposition}
\newtheorem{cor}[thm]{Corollary}
\newtheorem{question}{Question}
\newtheorem{con}{Conjecture}
\theoremstyle{definition}
\newtheorem{defn}[thm]{Definition}
\newtheorem{exm}[thm]{Example}
\newtheorem{nexm}[thm]{Non Example}
\newtheorem{prob}[thm]{Problem}

\theoremstyle{remark}
\newtheorem{rem}{Remark}[section]

\def\one{\mathbbm{1}}

\def \toprob {\,\,\buildrel\Prob\over\longrightarrow\,\,}
\def \toas {\,\,\buildrel\text{\rm a.s.}\over\longrightarrow\,\,}
\def \todist {\,\,\buildrel\text{\rm d}\over\longrightarrow\,\,}
\def \tovague {\,\,\buildrel\text{\rm v}\over\longrightarrow\,\,}
\def \toweak {\,\,\buildrel\text{\rm w}\over\longrightarrow\,\,}
\def \eqdist {\buildrel\text{\rm d}\over =}

\title[Fine-scale statistics for $\mathbb{Q}^n$]{Fine-scale statistics for $\mathbb{Q}^n$}

\author{Gaurav Aggarwal}
\address{\textbf{Gaurav Aggarwal} \\
Institut f\"ur Mathematik, Universit\"at Z\"urich, 8057 Z\"urich, Switzerland}
\email{gaurav.aggarwal@math.uzh.ch}

\author{Anish Ghosh}
\address{\textbf{Anish Ghosh} \\
School of Mathematics,
Tata Institute of Fundamental Research, Mumbai, India 400005}
\email{ghosh@math.tifr.res.in}

\author{Jens Marklof}
\address{\textbf{Jens Marklof} \\
School of Mathematics,
University of Bristol, BristolBS8 1UG, U.K.}
\email{j.marklof@bristol.ac.uk}

\date{}

\subjclass[2020]{}
\keywords{}


\begin{abstract}  
We study the distribution of rational points in $\mathbb{R}^n$, with denominators restricted to the interval $[Q-\Delta, Q]$, and $Q,\Delta\to\infty$ such that $\Delta/Q\to 0$. Previous results in the literature, due to Hall and others, were limited to Farey sequences, where the window size $\Delta$ is of the same order as $Q$. We prove the convergence of fine-scale statistics in a range of scaling limits and express the limit laws in terms of natural probability measures on the space of affine lattices. The key technical ingredient of our approach is an equidistribution theorem for slowly expanding horospheres, with some new exotic limit measures. Our techniques furthermore allow us to answer a recent question by Anderson, Boca, Cobeli and Zaharescu concerning the directional statistics of lattice points. 
\end{abstract}

\maketitle


\section{Introduction}

The objective of the present study is to establish new limit theorems for the local statistics of rational points in $\Q^n$. We will be interested in the fine-scale distribution of the multiset
\begin{align}\label{FQ}
    \cF_{Q, \Delta} = \left\{ \frac{p}{q} : p \in \Z^n,\; q\in [Q-\Delta, Q]\cap\N \right\},
\end{align}
where $n\in\N$ and $0\leq \Delta\leq Q$ with $Q$ large. ``Multiset'' means that we record points with their multiplicity. For example, ``0'' would be counted with multiplicity equal to the number of integers in the interval $[Q-\Delta, Q]$. In the case $Q=\Delta$, the set $\cF_{Q,Q}$ yields the Farey points of level $Q$ (if we only consider reduced fractions) whose gap distribution was studied by Hall \cite{Hall1970} in dimension $n=1$; see also \cites{KZ1997,BCZ2001,BZ2005} for other local statistics and \cite{Marklof2013} for results in general dimension. These techniques extend also to the case when $\Delta=c Q$ for any fixed $c\in(0,1]$. We will therefore focus on limits when $\Delta$ is significantly smaller than $Q$. 
In this scaling, it is natural to consider all rationals, i.e., not just reduced fractions with $\gcd(p,q)=1$, since denominators are restricted to a small window and the occurrence of multiplicity is rare. The fine-scale statistics of \eqref{FQ} with the additional assumption $\gcd(p,q)=1$ is more complicated; we will discuss this in a future study. 
The standing assumption for our investigation is therefore that
\begin{equation}\label{scalim}
Q\to\infty, \qquad \Delta\to\infty, \qquad \frac{\Delta}{Q} \to 0.
\end{equation}

 The case of finite $\Delta$ is elementary in the present setting, since the local statistics are given by a superposition of rigid lattices. This changes if we impose $\gcd(p,q)=1$. In this case, in dimension $n=1$ and $\Delta=0$ (with $Q\in\N$), Hooley \cites{Hooley1,Hooley2,Hooley3,Hooley4} proved that the gap statistics is Poissonian when $Q/\varphi(Q)\to\infty$, where $\varphi$ is Euler's totient function.

\subsection{Gap statistics}\label{sec1.1}

We start our discussion in dimension $n=1$ with one of the most popular tests of pseudo-randomness of a deterministic sequence: the gap distribution. 
We restrict $\cF_{Q, \Delta}$ to the interval $\cD_{\e,\alpha}=\alpha+\e\cD$, for a given bounded interval $\cD\subset\R$ of length $|\cD|\neq 0$, and $\epsilon$ either fixed, or tending to zero with $\e Q \to \infty$. We keep $\alpha$ fixed throughout.
Let us label the elements of the multiset $\cF_{Q, \Delta}\cap \cD_{\epsilon,\alpha}$ by
\begin{equation}\label{xis}
\xi_1 \leq \xi_2 \leq \ldots \leq \xi_N . 
\end{equation}
A counting argument shows that
\begin{align}\label{Ntot}
  N= |\cF_{Q, \Delta}\cap \cD_{\epsilon,\alpha} | = \sum_{Q-\Delta \leq q \leq Q}  \big(\epsilon q |\cD| + O(1) \big) = \epsilon Q\Delta |\cD| + O(\epsilon \Delta^2 + \Delta).
 \end{align}
Thus the average gap between consecutive $\xi_j$ is asymptotic to $(Q\Delta)^{-1}$.
Furthermore, assuming \eqref{scalim}, we have for any interval $\cA\subset\cD$,
\begin{align}\label{asymp dim 1}
    \frac{|\cA_{\epsilon,\alpha} \cap\cF_{Q, \Delta}|}{|\cD_{\epsilon,\alpha} \cap\cF_{Q, \Delta}|} \to  \frac{|\cA|}{|\cD|} .
\end{align}
This in turn implies that
\begin{equation}
\frac{|\{j\leq N: \xi_j \in \cA_{\epsilon,\alpha}\}|}{N} \to \frac{|\cA|}{|\cD|} . 
\end{equation}
That is, the $\xi_j$ are uniformly distributed on all scales that are large compared to $Q^{-1}$. 

The gap distribution of $(\xi_j)_{j=1}^N$ is defined as
\begin{equation}
\label{eq:def PN}
P_N(s) =\frac{1}{N-1} \; \big|\big\{ j\leq N-1 :  \xi_{j+1} -\xi_j > (Q\Delta)^{-1} s \big\} \big| ,
\end{equation}
where we measure gaps on the scale of the asymptotic average gap size $(Q\Delta)^{-1}$. This scaling ensures that the expected value of $P_N$ is asymptotically one, i.e.,
\begin{equation}\label{expected}
\int_0^\infty P_N(s) ds =\frac{Q\Delta}{N-1} \sum_{j=1}^{N-1} (\xi_{j+1} -\xi_j) = \frac{Q\Delta}{N-1} (\xi_N -\xi_1) \to 1, 
\end{equation}
since uniform distribution implies that $\xi_N -\xi_1\sim \epsilon|\cD|$. We will prove the following limit theorem for the gap distribution.

 \begin{thm}
    \label{thm:main0}
    Let $\alpha \in \R$, $\cD\subset\R$ an interval of length $|\cD|\neq 0$, and $\sigma\in (0,\infty]$. Then there exist a continuous probability density $p$ on $\R_{\geq 0}$ and piecewise continuous probability densities $p_{0,\cD}$ on $\R_{\geq 0}$ such that, for any sequence of $(Q,\Delta,\epsilon)$ satisfying \eqref{scalim} and $\e \Delta\to \sigma$, we have:

\begin{enumerate}
\renewcommand{\labelenumi}{\rm (\roman{enumi})}
\item 
If $\alpha=p/q\in\Q$ with $\gcd(p,q)=1$ and $\sigma<\infty$, then
\begin{equation}
 P_N(s) \to 
 \int_s^{\infty} p_{0,q\sigma\cD}(t) dt
 \end{equation}

\item
If $\alpha\notin\Q$ or $\sigma=\infty$, then
\begin{equation}
 P_N(s) \to 
 \int_s^{\infty} p(t) dt .
 \end{equation}
 \end{enumerate}
\end{thm}

The two limit distributions in Theorem \ref{thm:main0} are compatible, in that we have
\begin{equation}
\lim_{\sigma\to\infty} \int_s^{\infty} p_{0,\sigma\cD}(t) dt = \int_s^{\infty} p(t) dt.
\end{equation}
We will prove that the limit density $p$ is given by the formula
\begin{equation}\label{p(s)}
p(s) = \tfrac{6}{\pi^2} \times
\begin{cases}
1 & (s\leq 1) \\
 \frac1s + 
2\,\Big(1-\frac1s\Big)^2 \log\Big(1-\frac1s\Big)  -\frac{1}{2}\, \Big(1-\frac2s\Big)^2 \log \Big| 1-\frac2s\Big|  & (s>1) ,
\end{cases}
\end{equation}
which is evidently independent of $\alpha,\sigma,\cD$.
The expected value is $\int_0^\infty s\, p(s) ds=1$, which is consistent with \eqref{expected}, and we have a heavy tail $p(s) \sim  \frac{4}{\pi ^2} s^{-3}$ for $s\to\infty$. A numerical illustration of Theorem \ref{thm:main0} (ii) is given in Figure \ref{fig1} for $\alpha\notin\Q$ and $\sigma<\infty$, and in Figures \ref{fig1B} for $\alpha\in\Q$ and $\sigma=\infty$.

The function $p(s)$ also describes the level spacing distribution for a two-dimensional quantum harmonic oscillator with random frequencies (or, equivalently) the gap distribution of the fractional parts of $n\omega$ mod 1 for random $\omega$ \cites{Greenman1996,PSZ2016}, as well as the distribution of free path lengths in the two-dimensional periodic Lorentz gas in the low-density limit \cites{Dahlqvist1997,BZ2007}. Ref.~\cite{Marklof2021} explains why the same limit density \eqref{p(s)} arises in these different settings.

\begin{figure}
\begin{center}
\includegraphics[width=0.5\textwidth]{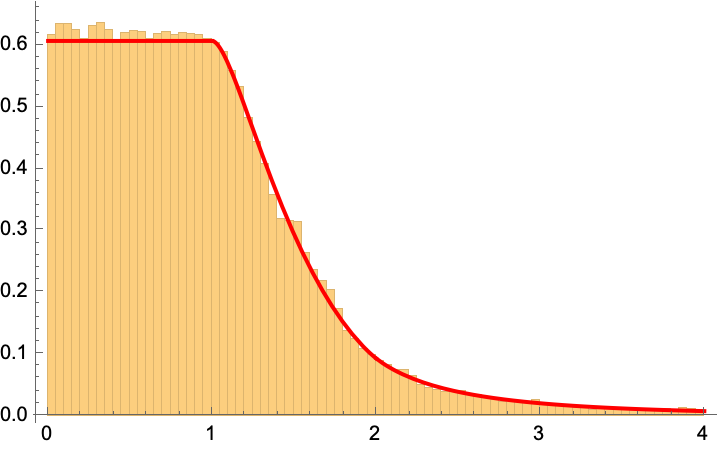}
\end{center}
\caption{Histogram of the gap distribution with $\alpha=\sqrt2-1$, $Q=10^5$, $\sigma=1$, $\cD=[1.2,1.7]$. The continuous curve is the limiting density $p(s)$.} \label{fig1}
\end{figure}

\begin{figure}
\begin{center}
\includegraphics[width=0.5\textwidth]{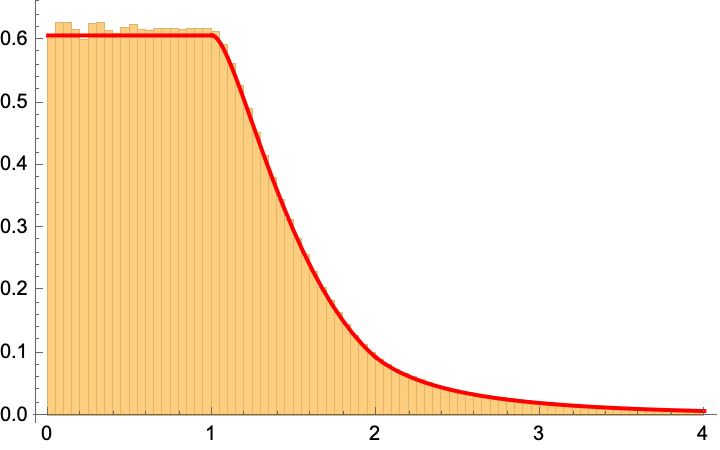}
\end{center}
\caption{Histogram of the gap distribution with $\alpha=0$, $Q=10^5$, $\sigma=10^3$, $\cD=[1.2,1.7]$. The continuous curve is the limiting density $p(s)$.} \label{fig1B}
\end{figure}

Let us turn to $p_{0,\cD}(s)$, where (in view of the reflection symmetry of $\cF_{Q,\Delta}$ about the origin) it is natural to assume $\cD\subset\R_{\geq 0}$. Consider the points $n\omega$ mod 1, $n=1,\ldots,k$, and denote by $s_{k,1}(\omega),\ldots,s_{k,k}(\omega)$ the $k$ gaps between consecutive points on $\R/\Z$; assume here $0<\omega<1$. 
For example, if $k=1$, we have $s_{1,1}(\omega)=1$, and for $k=2$, we have $s_{2,1}(\omega)=\omega$ and $s_{2,2}(\omega)=1-\omega$ (or vice versa). More generally, the classic three gap theorem tells us that, for each given $\omega$ and $k$, the set $\{ s_{k,j}(\omega) : j\leq k\}$ has at most three distinct elements; see \cite{MS2017} and references therein. 
We will show in Section \ref{sec:introproofs} that
\begin{equation}\label{p_alf}
p_{0,\cD}(s) = \frac{1}{s |\cD|} \sum_{k=1}^\infty \sum_{j=1}^k \int_0^1 \chi_{\cD}\big( s_{k,j}(\omega)^{-1} s \big) \max\big( 1 - |k-s_{k,j}(\omega)^{-1} s|, 0\big) d\omega,
\end{equation}
where $\chi_{\cD}$ is the characteristic function of the sets $\cD\subset\R_{\geq 0}$. Since $\cD$ is assumed to be bounded, so is $s_{k,j}(\omega)^{-1}s$ and the sum over $k$ has at most finitely many non-zero terms. This implies the piecewise continuity of $p_{0,\cD}(s)$. We also note that if $b=\sup \cD$, then $s\leq s_{k,j}(\omega) b\leq b$, and hence the support of $p_{0,\cD}$ is contained in $[0,b]$.

Using $s_{k,1}(\omega)+\ldots + s_{k,k}(\omega)=1$, formula \eqref{p_alf} yields for the expected value 
\begin{equation}\label{p_alf_expected}
\int_0^\infty s\, p_{0,\cD}(s)\, ds = \frac{1}{|\cD|}  \int_\cD \min(s , 1) \, ds .
\end{equation}
Thus the expected limiting gap is strictly less than $1$ if $\cD\cap[0,1]$ has positive Lebesgue measure, and is equal to one otherwise (recall we assume here $\cD\subset\R_{\geq 0}$). This in turn means that the limit of the expected gap for finite $N$ \eqref{expected} does in general not coincide with the expected value of the limit distribution, unlike for $p(s)$. The reason for this phenomenon is the occurrence of large gaps that are not statistically significant for convergence in distribution, but which are picked up by the expected value for each finite $N$, due to the higher weighting given to large gaps.

{\em Example 1.} $\cD=[a,b]$ with $0\leq a < b \leq 1$. Only the $k=1$ term contributes, and we have (recall that here $s_{1,1}(\omega)=1$)
\begin{equation}\label{p_alf1}
p_{0,\cD}(s) = \frac{1}{b-a} \, \chi_{[a,b]}(s).
\end{equation}

{\em Example 2.} $\cD=[a,b]$ with $1\leq a < b \leq 2$. Now the $k=1$ and $k=2$ terms contribute, and we have
\begin{equation}\label{p_alf2}
p_{0,\cD}(s) = \frac{2}{b-a} \times
\begin{cases}
\log\frac{b}{a} +\frac1b-\frac1a & (s\leq a) \\
\log\frac{b}{s} +\frac1b-\frac12 & (a< s < b) \\
0 & (s> b).
\end{cases}  
\end{equation}

\begin{figure}
\begin{center}
\includegraphics[width=0.5\textwidth]{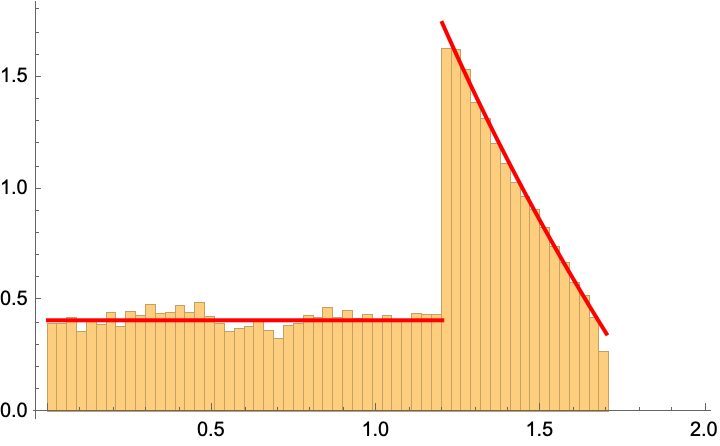}
\end{center}
\caption{Histogram of the gap distribution with $\alpha=0$, $Q=10^5$, $\sigma=1$, $\cD=[1.2,1.7]$. The continuous curve is the limiting density $p_{0,\cD}(s)$.} \label{fig2}
\end{figure}
\begin{figure}
\begin{center}
\includegraphics[width=0.5\textwidth]{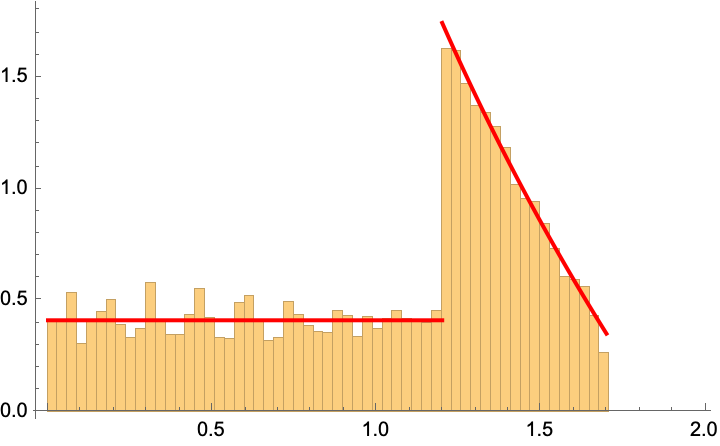}
\end{center}
\caption{Histogram of the gap distribution with $\alpha=3/5$, $Q=10^5$, $\sigma=1$, $\cD=[0.24,0.34]$. The continuous curve is the limiting density $p_{0,5\cD}(s)$.} \label{fig2.1}
\end{figure}

{\em Example 3.} $\cD=[a,b]$ with $0\leq a \leq 1 < b \leq 2$. The distribution is now a combination of the distributions from Example 1 (for the interval $[a,1]$) and Example 2 (for the interval $[1,b]$). We have
\begin{equation}\label{p_alf3}
p_{0,\cD}(s) = \frac{2}{b-a} \times
\begin{cases}
\log b +\frac1b-1 & (s \leq a) \\
\log b +\frac1b-\frac12 & (a< s\leq 1) \\
\log\frac{b}{s} +\frac1b -\frac12 & (1< s \leq b) \\
0 & (s> b).
\end{cases}  
\end{equation}

\begin{figure}
\begin{center}
\includegraphics[width=0.5\textwidth]{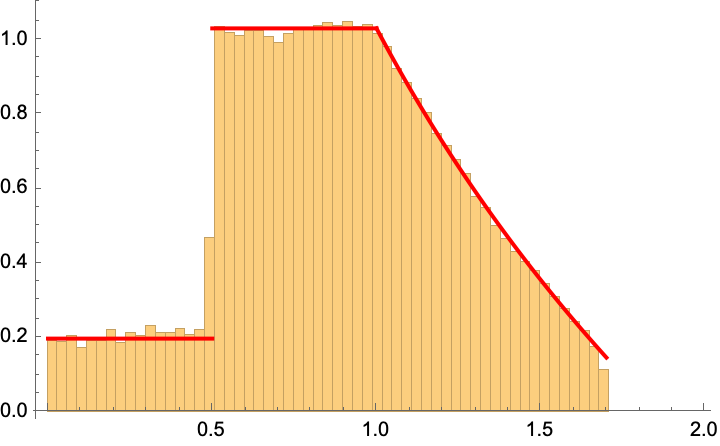}
\end{center}
\caption{Histogram of the gap distribution with $\alpha=0$, $Q=10^5$, $\sigma=1$, $\cD=[0.5,1.7]$. The continuous curve is the limiting density $p_{0,\cD}(s)$.} \label{fig3}
\end{figure}
\begin{figure}
\begin{center}
\includegraphics[width=0.5\textwidth]{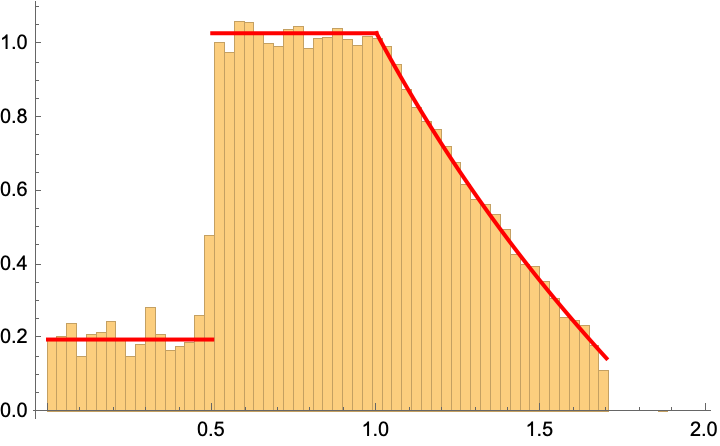}
\end{center}
\caption{Histogram of the gap distribution with $\alpha=3/5$, $Q=10^5$, $\sigma=1$, $\cD=[0.1,0.34]$. The continuous curve is the limiting density $p_{0,5\cD}(s)$.} \label{fig3.1}
\end{figure}

The limit distributions in Theorem~\ref{thm:main0}, for $\alpha=0$ and $\sigma<\infty$ finite, also arise in the gap distribution of angles between two-dimensional lattice points observed by a ``fast-moving'' observer, as studied by Anderson, Boca, Cobeli and Zaharescu \cite{ABCZ}. Theorem~\ref{thm:main0} for $\sigma=\infty$ in fact allows to extend the scaling limits and thus answers a question posed in \cite{ABCZ}. We will return with more details on this point at the end of this introductory section. 

In the present study we will not discuss scaling limits corresponding to $\sigma=0$, except to say that in the case $\alpha=0$ almost all gaps accumulate at zero, so that
\begin{equation}\label{zerosigma}
P_N(s) \to 0
\end{equation}
for all $s> 0$. That is, the gap distribution converges weakly to the Dirac mass $\delta_0$ at zero. We will explain this at the end of Section \ref{sec:introproofs}.

\subsection{Local statistics in higher dimension}
\label{subsec:Local statistics in higher dimension}
Although there are natural generalisations of gap statistics or nearest-neighbour distributions in higher-dimensional settings, it is more convenient to study a different fine-scale statistics, namely the distribution of points in small, randomly placed test sets. We will then extend this to express the fine-scale distributions of $\cF_{Q, \Delta}$ in terms of point processes.

As in dimension one, we restrict the rationals in the multiset $\cF_{Q, \Delta}$ to $\cD_{\epsilon,\alpha}$, the $\epsilon$-dilated and $\alpha$-translated copy of a given set $\cD\subset\R^n$, i.e.,
\begin{equation}
\cD_{\epsilon,\alpha} = \alpha+\epsilon\cD = \{ x \in\R^n : \epsilon^{-1} (x-\alpha)\in\cD \},
\end{equation}
where $\epsilon>0$ (fixed or tending to zero) and $\alpha\in\R^n$ (fixed).
We assume in the following that $\cD$ is bounded with boundary of Lebesgue measure zero and non-empty interior. 

A lattice point counting argument shows that, in the limit \eqref{scalim}, we have
\begin{align}
  |\cF_{Q, \Delta}\cap \cD_{\epsilon,\alpha} | \sim \epsilon^n \vol(\cD) \sum_{Q-\Delta \leq q \leq Q}  q^n  \sim \epsilon^n \vol(\cD)  Q^n \Delta .
 \end{align}
This implies that, for any bounded $\cA\subset\cD$ with boundary of measure zero,
\begin{align}\label{asymp}
    \frac{|\cA_{\epsilon,\alpha} \cap\cF_{Q, \Delta}|}{|\cD_{\epsilon,\alpha} \cap\cF_{Q, \Delta}|} \to  \frac{\vol(\cA)}{\vol(\cD)} .
\end{align}
As in the one-dimensional setting, this means that the points in $\cD_{\epsilon,\alpha} \cap\cF_{Q, \Delta}$ are uniformly distributed in $\cD_{\epsilon,\alpha}$ with respect to the Lebesgue measure.

To test the pseudo-randomness properties of $\cF_{Q, \Delta}$, we will now go beyond the above scales and test the distribution on ``microscopic'' scales comparable to the average separation between the points in $\cF_{Q, \Delta}$. This average separation is, in view of uniform distribution \eqref{asymp} and the fact that $\vol(\cD_{\epsilon,\alpha}) = \epsilon^n \vol(\cD)$, asymptotic to $\eta_{Q,\Delta}=Q^{-1} \Delta^{-1/n}$.
We are interested in the number of points of $\cF_{Q,\Delta}$ in a microscopic test set of size $\eta_{Q,\Delta}$. In this scaling, we can expect only a finite number of points to fall in the scaled test set. We now translate the microscopic test set by a random vector $\xi$, and ask for the probability that the scaled and randomly shifted set contains exactly a given number of points. 
That is, for $k \in \Z_{\geq 0}$ and $\cA,\cD\subset \R^n$ bounded, with boundary of Lebesgue measure zero and non-empty interior, define
\begin{align}
\label{eq: def 1 P}
    E_Q(k,\cA)= \frac{\vol\left(\left\{ \xi \in \cD_{\epsilon,\alpha} : |(\xi+\eta_{Q,\Delta} \cA) \cap \cF_{Q, \Delta}\}|= k\right\} \right)}{\vol(\cD_{\epsilon,\alpha})}.
\end{align}
Here $E_Q(k,\cA)$ depends of course also on $\alpha,\epsilon,\cD,\Delta$.
For the expected value of \eqref{eq: def 1 P} we have 
\begin{equation}
\sum_{k=1}^\infty k E_Q(k,\cA)  = \frac{1}{\vol(\cD_{\epsilon,\alpha})} \int_{\cD_{\epsilon,\alpha}}  |(\xi+ \eta_{Q,\Delta} \cA ) \cap \cF_{Q,\Delta}| \, d\xi\to 1,
\end{equation}
where the convergence to 1 follows from uniform distribution \eqref{asymp}.

The following proves the existence of limit distributions in various scaling limits.

\begin{thm}
    \label{thm:main}
    Let $\alpha \in \R^n$, $\cA$, $\cD\subset\R^n$ bounded with boundary of Lebesgue measure zero, and $\sigma\in (0,\infty]$. Then there exist  probability distributions $E(k,\cA)$ and $E_{0,\cD}(k,\cA)$ such that, for any $k\in\Z_{\geq 0}$, and any sequence of $(Q,\Delta,\epsilon)$ satisfying \eqref{scalim}, and $\e \Delta^{1/n}\to \sigma$, we have:
    
\begin{enumerate}
\renewcommand{\labelenumi}{\rm (\roman{enumi})}
\item 
 If $\alpha=p/q\in\Q^n$ with $\gcd(p,q)=1$ and $\sigma<\infty$, then
\begin{equation}
 E_Q(k,\cA) \to  E_{0,q\sigma\cD}(k,\cA) . 
 \end{equation}

\item
If $\alpha\notin\Q^n$ or $\sigma=\infty$, then
\begin{equation}
 E_Q(k,\cA) \to E(k,\cA) .
\end{equation}
\end{enumerate}
\end{thm}

Following the strategy of \cite{Marklof2013} for Farey fractions, the first step of the proof of Theorem~\ref{thm:main} translates the question into an equidistribution problem for sequences of horospheres in the space of affine lattices. The new feature of our setting is that the horospheres expand at different rates, which depend on the choice of $\sigma$. We establish equidistribution in these scaling limits using measure rigidity techniques. A key ingredient here is a generalization of a theorem of Dani and Margulis~\cite{DM93} on the equidistribution of expanding translates of unipotent orbits, which is of independent interest. Dani and Margulis proved that averages along expanding segments of one-parameter unipotent orbits become equidistributed with respect to Haar measure on the ambient homogeneous space. Our result extends their theorem in two directions: first, from one-parameter to multi-parameter unipotent flows, and second, from individual initial points to varying probability measures. That is, we study the distribution of translates $u_t^{(k)}x$, where $x$ is sampled according to a sequence of measures $\mu_k$ and $t$ ranges over expanding cuboids in $\R^d$. Under suitable assumptions, we prove convergence to Haar measure. The precise statement is given in Section~\ref{sec: Gen Dan Mar}.

Our equidistribution theorem allows us to express the limiting fine-scale statistics in Theorem~\ref{thm:main} in terms of random lattices. In case (ii), we will show that
\begin{equation}\label{Estar}
E(k,\cA) = \mathbb{P}\big( | \Lambda \cap \cZ(\cA)| = k\big),
\end{equation}
where $\cZ(\cA)=\cA\times[-1,0]$ and $\Lambda\subset\R^{n+1}$ is a random affine lattice of co-volume one, distributed according to the Haar probability measure on the space of affine lattices (we will provide precise definitions in Section \ref{sec:Main}). Due to the translation invariance of the limit measure, the interval $[-1,0]$ can be replaced by any unit interval. 

To connect \eqref{Estar} with existing literature, define for $\xi>0$,
\begin{equation}\label{Estar2}
E(k,\cA,\xi) = \mathbb{P}\big( | \Lambda \cap \cZ(\cA,\xi)| = k\big),
\end{equation}
with $\cZ(\cA,\xi)=\cA\times[0,\xi]$. Then, by the previous remark, $E(k,\cA,1)=E(k,\cA)$. Furthermore, the $\SL(n+1,\R)$-invariance of the Haar probability measure on the space of affine lattices implies that, for $r>0$,
\begin{equation}\label{Estar3}
E(k,r\cA,\xi) = E(k,\cA,r^n \xi) .
\end{equation}
The distribution $E(k,\cA,\xi)$ arises in natural lattice point problems \cite[\S 3]{MarStro}.
In particular, if $\cA=\cB_1^n$ (the unit ball in $\R^n$), then $F(\xi)=E(0,\cB_1^n,\xi)$ describes the first hitting time distribution in the $(n+1)$-dimensional periodic Lorentz gas \cite[\S 4]{MarStro}. Precise tail asymptotics for the probability density $\Phi(\xi)=-F'(\xi)$, for small and large $\xi>0$, are given in \cite[Cor.~1.3 \& Thm.~1.13]{MarStroGAFA}.
These asymptotics translate to tail estimates for the limiting void density 
\begin{equation}
\rho_0(r) := -\frac{d}{dr}  E(0,\cB_r^n) 
\end{equation}
in our setting via the relation
\begin{equation}
\rho_0(r) =  -\frac{d}{dr} F(r^n) = n r^{n-1} \Phi(r^n),
\end{equation}
The asymptotics for $\Phi(\xi)$ also imply bounds for more general bounded test sets $\cA$ with non-empty interior. If $\cA$ contains a ball of radius $\e_1$ and is contained in a ball of radius $\e_2$, we have
\begin{equation}
E(0,\cB_{\e_2 r}^n) \leq E(0,r\cA) \leq E(0,\cB_{\e_1 r}^n)
\end{equation}
and therefore, for large $r$, 
\begin{equation}
E(0,r\cA) \asymp r^{-n} .
\end{equation}

In dimension $n=1$, it is a general fact \cite{Marklof2007} that the gap distribution converges if and only if the void distribution ($k=0$) does, and that the limits are related via the well known formula 
\begin{equation}
\label{eq:void to gap 1}
P(s) = - \frac{d}{ds} E(0,[0,s]) .
\end{equation}
Thus, the limiting gap density is given by
\begin{equation}
\label{eq:void to gap 2}
p(s) = - \frac{d}{ds} P(s) =  \frac{d^2}{ds^2} E(0,[0,s]) .
\end{equation}
Since the free flight distribution is the derivative of the first hitting time distribution, this in particular shows that the gap density $p(s)$ is the same as the free path density for the periodic Lorentz gas; see \cite{Marklof2021} for more details.

In the case (i) of Theorem~\ref{thm:main}, the limit distribution $E_{0,q\sigma\cD}(k,\cA)$ satisfies the same formula as \eqref{Estar}, except that $\Lambda=\Lambda_{0,q\sigma \cD}$ is distributed according to a different probability measure $\nu_{0,q\sigma\cD}$ on the space of affine lattice, which now depends on the choice of $\alpha$, $\sigma$ and $\cD$. The limit measure is non-standard and, although invariant under the horospherical action, it is not ergodic. It is furthermore invariant under the substitution $(\sigma,\cD)\mapsto (r\sigma,r^{-1}\cD)$ for any $r>0$.
In dimension $n=1$, with $\cD\subset\R_{\geq 0}$, we have for example
\begin{equation}\label{EkL}
E_{0,\cD}(k,\cA) = \mathbb{P}\big( | \Lambda_{0,\cD} \cap \cZ(\cA)| = k\big)
\end{equation}
with
\begin{equation}\label{LambdaAD}
\Lambda_{0,\cD} = a(t) u(\omega)(\Z^2+ z) ,\qquad a(t) = \begin{pmatrix}
    t & 0 \\ 0 & t^{-1}
\end{pmatrix}, \quad u(\omega) = \begin{pmatrix}
    1 & \omega \\ 0 &1
\end{pmatrix},
\end{equation}
where $t$ is a random variable uniformly distributed in $\cD$, $\omega$ is uniformly distributed in the unit interval $[0,1]$, and $z$ is uniformly distributed in the unit square $[0,1]^2$.

\subsection{Point processes}

Theorem \ref{thm:main} generalises to the convergence of point processes as follows. We can associate with the multiset $\cF_{Q,\Delta}$ the locally finite measure
\begin{equation}
X_Q = \sum_{q\in [Q-\Delta, Q]} \sum_{p \in \Z^n} \delta_{\eta_{Q,\Delta}^{-1}(q^{-1} p - \xi)} ,
\end{equation}
which can be viewed as a point process if $\xi$ is uniformly distributed in $\cD_{\epsilon,\alpha}$ as assumed above. Here $\delta_x$ denotes the Dirac mass at the point $x\in\R^n$. We will prove the following.

\begin{thm}
    \label{thm:main2}
    Let $\alpha \in \R^n$, $\cA$, $\cD\subset\R^n$ bounded with boundary of Lebesgue measure zero, and $\sigma\in (0,\infty]$. Then there exist point processes $X $ and $X_{0,\cD} $ such that for any sequence of $(Q,\Delta,\epsilon)$ satisfying \eqref{scalim}, and $\e \Delta^{1/n}\to \sigma$, we have:
\begin{enumerate}
\renewcommand{\labelenumi}{\rm (\roman{enumi})}
\item 
 If $\alpha=p/q\in\Q^n$ with $\gcd(p,q)=1$ and $\sigma<\infty$, then
\begin{equation}
 X_Q \todist  X_{0,q\sigma\cD}  .
 \end{equation}

\item 
If $\alpha\notin\Q^n$ or $\sigma=\infty$, then
\begin{equation}
 X_Q \todist X  .
\end{equation}
\end{enumerate}
\end{thm}

Convergence is defined in the vague topology on the space of locally finite Borel measures of $\R^n$.
The limiting processes are given by
\begin{equation}
\label{eq:defXpointprocess}
X  = \sum_{\substack{(x,y)\in \Lambda\\ -1\leq y\leq 0}} \delta_x , \qquad
X_{0,\cD}   = \sum_{\substack{(x,y)\in \Lambda_{0,\cD} \\ -1\leq y\leq 0}} \delta_x ,
\end{equation}
with the same random lattices $\Lambda$ and $\Lambda_{0,\cD}$ as above.

The proof of Theorem \ref{thm:main2} follows from the same equidistribution result, as above, for slowly expanding horospheres. There are two routes, either by applying the continuous mapping theorem (using the map from the space of affine lattices to the space of locally finite measures, cf.~\cite{MV2017}) or by proving convergence in finite dimensional distribution, i.e., by extending Theorem \ref{thm:main} to the joint distribution in several test sets $\cA_1,\ldots,\cA_r$, for any $r\in\N$. We will follow the latter route. 

\subsection{Directional statistics from a distant observer}\label{intro:ABCZ}

In a recent paper~\cite{ABCZ}, Anderson et al.\ studied the distribution of angles formed by rays from an observer to lattice points in an expanding square, where the distance between the observer and the square is significantly larger than the sides of the square. Their precise setting (using the notation of \cite{ABCZ}) is as follows. Given $\alpha>1$, $t>0$, assume the observer is located at
\begin{equation}
P_{t,J} = (-tJ^\alpha,0).
\end{equation}
Define the half-squares 
\begin{equation}
R = [-1,1] \times [0,1], 
\qquad 
R_J = J \cdot R = [-J,J] \times [0,J],
\end{equation}
and let
\begin{equation}
N=N_J = |R_J \cap \Z^2|.
\end{equation}
For each $P \in R_J \cap \Z^2$, consider the angle $\angle P_{t,J} P O$, with $O=(0,0)$, and order these angles as
\begin{equation}
0=\alpha_{J,1} \le \cdots \le \alpha_{J,N}.
\end{equation}
The average gap is
\begin{equation}
\Delta_{J,\mathrm{av}} = \frac{1}{N-1}(\alpha_{J,N} - \alpha_{J,1}),
\end{equation}
and define the gap distribution function
\begin{equation}
G_{t,J}(\xi) = \frac{\left| \left\{ j\leq N-1 : \alpha_{J,j+1} - \alpha_{J,j} > \xi \Delta_{J,\mathrm{av}} \right\} \right|}{N-1} , 
\quad \xi > 0.
\end{equation}

The following theorem yields a limit gap distribution for all $\alpha>1$ and hence extends the main result of \cite{ABCZ}, which was restricted to $\alpha=2$.

\begin{thm}\label{thm:ABCZ}
Let $\alpha>1$, $t>0$. For every $\xi > 0$, 
\begin{equation}
\lim_{J\to\infty} G_{t,J}(\xi) = G_t(\xi),
\end{equation}
where
\begin{equation}
G_t(\xi) =
\begin{cases}
0 & \text{if } \alpha>2 \\[6pt]
\displaystyle \int_{\xi}^\infty p_{0,\,[0,2/t]}(s)\,ds
& \text{if } \alpha=2 \\[10pt]
\displaystyle \int_{\xi}^\infty p(s)\,ds
& \text{if } \alpha<2,
\end{cases}
\end{equation}
and the densities $p_{0,\cD}$ and $p$ are as in Theorem~\ref{thm:main0}.
\end{thm}

This answers a question of Anderson et al.~\cite{ABCZ}, who had asked for a proof of the existence and shape of the limiting gap distribution for $\alpha\neq 2$. The proof of Theorem \ref{thm:ABCZ} is essentially the same as Theorem~\ref{thm:main0}; see Section~\ref{sec:ABCZ}. The tools developed in the present work allow generalisations of Theorem~\ref{thm:ABCZ} to more general lattices, domains, and higher dimensions, which we will discuss in forthcoming work.

\subsection*{Plan of the paper}

Section~\ref{sec: Gen Dan Mar} establishes a generalised version of the Dani–Margulis theorem~\cite[Thm.~2]{DM93} on the limiting distributions of multi-parameter unipotent flows when both the unipotent subgroups and the initial measures vary simultaneously. The main result of this section, Theorem~\ref{thm: gen Dani Margulis}, should be of independent interest. It is used in Sections~\ref{sec:Finite case} and~\ref{sec: sigma finite is enough}.
The core of the paper is devoted to proving the key technical tool: a classification of limit measures of slowly expanding horospheres in the space of affine lattices. Section~\ref{sec:Main} introduces the geometric setting, states the main ergodic-theoretic result of this work (Theorem~\ref{thm: main dynamical theorem}), and provides an outline of its proof. The full argument is then carried out in Sections~\ref{sec: Reduction to X}--\ref{sec: sigma finite is enough}. Section~\ref{AppB} collects technical lemmas on convergence of measures that are used in the proofs of Theorems~\ref{thm:main0},~\ref{thm:main} and~\ref{thm:main2} in Section~\ref{sec:introproofs}. Section \ref{sec:ABCZ} provides the proof of Theorem~\ref{thm:ABCZ}.

\subsection*{Acknowledgements}

This project was supported by the Royal Society Yusuf Hamied International Exchange Award IES\textbackslash R1\textbackslash 241328, which is gratefully acknowledged.  G.A. and A.G. gratefully acknowledge support from the Department of Atomic Energy, Government of India (project12-R\&D-TFR-5.01-0500).  G.A. gratefully acknowledges support from the Swiss National Science Foundation (grant 200020-212617). A.G. gratefully acknowledges support from a J.C. Bose grant and a grant from the Infosys Foundation.  J.M. gratefully acknowledges support from the Engineering and Physical Sciences Research Council (grant EP/W007010/1). Data supporting this study are included within the article.

\section{A generalised Dani-Margulis theorem}
\label{sec: Gen Dan Mar}

This section proves a generalisation of the classical equidistribution theorem of Dani and Margulis~\cite{DM93} for expanding translates of unipotent flows. Their result establishes equidistribution for one-parameter unipotent flows acting on homogeneous spaces, starting from a sequence of initial points $x$. We extend this in two directions: (i) from one-parameter to multi-parameter unipotent actions, and (ii) from sequences of initial points to sequences of probability measures describing the initial conditions (the original setting of points will then be a special case by taking as the measure a Dirac point mass).

Let $G$ be a connected Lie group and $\Gamma$ a discrete subgroup of $G$. For any closed subgroup $U \subset G$, we define the \emph{singular set}
\begin{multline}
\cS(U) := \{ x \in G/\Gamma : \exists\, H < G \text{ proper, } U \subset H, \\ \text{ such that } H \cdot x \text{ admits a finite } H\text{-invariant measure} \}.
\end{multline}
We call its complement the \emph{generic set}
\begin{equation}
\label{eq:def g U}
\cG(U) := (G/\Gamma)\smallsetminus \cS(U).
\end{equation}

In this section we prove the following theorem.

\begin{thm}
    \label{thm: gen Dani Margulis}
    Let $G$ be a connected Lie group and $\Gamma$ be a lattice in $G$. Let $\nu$ denote the $G$-invariant probability measure on $G/\Gamma$. Suppose $\{u^{(i)}_t: t \in \R^k\}$ is a sequence of unipotent subgroups of $G$ converging to a unipotent subgroup $U=\{u_t: t \in \R^k\}$; that is 
    \begin{equation}
    u^{(i)}_t \rightarrow u_t   
    \end{equation}
    for all $t \in \R^k$, as $i \rightarrow \infty$. Suppose $\mu_i$ is a sequence of Borel probability measures on $G/\Gamma$ converging weakly  to a Borel probability measure $\mu$ such that $\mu(\cS(U))=0$. Suppose  
    \begin{equation}
    I_i= [\alpha_1^{(i)}, \beta_1^{(i)}] \times \cdots \times [\alpha_k^{(i)}, \beta_k^{(i)}]
    \end{equation}
    is sequence of subsets of $\R^k$ such that for all $j$, we have $\alpha_j^{(i)} \leq 0 \leq \beta_j^{(i)}$ and 
    \begin{equation}
        |\beta_j^{(i)} - \alpha_j^{(i)}| \rightarrow \infty,
    \end{equation}
   as $i \rightarrow \infty$. Then for any continuous bounded function $f$ on $G/\Gamma$, we have
    \begin{equation}\label{lambdaiF}
\lambda_i(f):=   \frac{1}{m_{\R^k}(I_i)} \int_{I_i} \int_{G/\Gamma} f(u^{(i)}_t x) \, d\mu_i(x) \, dt \rightarrow \int_{G/\Gamma} f\, d\nu. 
    \end{equation}
\end{thm}

\begin{rem}
Theorem~\ref{thm: gen Dani Margulis} extends \cite[Thm.~2]{DM93}, which treats the case of a single-parameter unipotent flow and for $\mu_i$ and $\mu$ equal to Dirac measures.
\end{rem}

The proof of the theorem requires several steps. First of all, we need to show that any subsequential limit of the measures $\lambda_i$ defined in \eqref{lambdaiF}
is a probability measure. To prove this, we will need the following result from \cite{DM93}.

\begin{thm}[{\cite[Thm.~6.1]{DM93}}]
\label{thm: DM93 Thm 6.1}
   Let $G$ be a connected Lie group and $\Gamma$ be a lattice in $G$. Let $F$ be a compact subset of $G/\Gamma$ and let $\e>0$ be given. Then there exists a compact subset $K$ of $G/\Gamma$ such that for any unipotent one-parameter subgroup $\{v_t\}$ of $G$, any $x \in F$, and any $T \geq 0$,
    \begin{equation}
    m_{\R} \left\{ t \in [0,T]: v_tx \in K\right\}  \geq (1-\e) T.
    \end{equation}
\end{thm}

Using Theorem~\ref{thm: DM93 Thm 6.1}, we obtain tightness of the sequence of measures $(\lambda_i)_i$. By a standard argument, any subsequential limit is invariant under the action of the unipotent group $U$. Then, by Ratner's measure classification theorem \cites{Ratner91A}, every such limit measure is an integral of homogeneous measures. To conclude that the limit measure equals $\nu$, it therefore suffices to show that it gives zero measure to the set $\cS(U)$. For this, we will use the following result from \cite{DM93}.

\begin{thm}[{\cite[Thm.~1]{DM93}}]
\label{thm: DM93 Thm 1}
    Let $G$ be a connected Lie group and $\Gamma$ be a discrete subgroup of $G$. Suppose $W$ is a closed subgroup of $G$ which is generated by the unipotent elements contained in it. Let $F$ be a compact subset of $\cG(W)$. Then for any $\e>0$, there exists a neighbourhood $\Omega$ of $\cS(W)$ such that for any unipotent one-parameter subgroup $\{v_t\}$ of $G$, any $x \in F$, and any $T \geq 0$,
    \begin{equation}
    m_{\R} \left\{ t \in [0,T]: v_tx \in \Omega\right\} \leq \e T.
    \end{equation}
\end{thm}

In order to apply the above result, we need to show for every $\e>0$, the existence of a compact subset $F_\e$ of $\cG(U)$ such that $\mu_i(F_\e)>1-\e$ for all $i$. However, this existence cannot be guaranteed. For example, one could take $\mu_i= \delta_{x_i}$, where $x_i$ is a sequence of points in $ \cS(U)$ converging to a point in $\cG(U)$. Hence, to apply Theorem~\ref{thm: DM93 Thm 1}, the first step is to avoid such cases, that is, we show that we can replace $(\mu_i)_i$ by a sequence $(\widetilde{\mu}_i)_i$ such that $\widetilde{\mu}_i(\cS(U))=0$, and the modified sequence of measures $\widetilde\lambda_i$ defined by
\begin{align}
    \label{eq: def measures appendix 2}
\widetilde\lambda_i(f):=   \frac{1}{m_{\R^k}(I_i)} \int_{I_i} \int_{G/\Gamma} f(u^{(i)}_t x) \, d\widetilde\mu_i(x)
\end{align}
still converges to the same limit as the $\lambda_i$, along every subsequence. We will need the following lemma.

\begin{lem}\label{lem: zero measure of sub hom spaces}
    Let $G$ be a connected Lie group and $\Gamma$ a lattice in $G$. 
    Let $U$ be a unipotent subgroup of $G$. 
    Suppose there exists $x \in G/\Gamma$ such that 
    \begin{align}
    \label{eq: 1 : lem: zero measure of sub hom spaces}
         \overline{U \cdot x} = G/\Gamma.
    \end{align}
    Then
    \begin{equation}
        \nu(\cS(U)) = 0,
    \end{equation}
    where $\nu$ denotes the $G$-invariant probability measure on $G/\Gamma$.
\end{lem}

\begin{proof}
Let $\delta>0$ be given. Fix $x \in G/\Gamma$ satisfying \eqref{eq: 1 : lem: zero measure of sub hom spaces}.  Suppose $U=\{u_t: t \in \R^k\}$. Then, using Ratner's theorem \cites{Ratner91A}, we have
 \begin{align}
 \label{eq: 2: lem: zero measure of sub hom spaces}
     \lim_{T \rightarrow \infty} \frac{1}{T^k}\int_{[0,T]^k}  \delta_{u_t x} \, dt = \nu.
 \end{align}
 
 Note that using Theorems~\ref{thm: DM93 Thm 6.1} and~\ref{thm: DM93 Thm 1}, we see that for any compact subset $F$ of $\cG(U)$, there exists a compact subset $F'$ contained in $\cG(U)$ such that for any unipotent subgroup $\{v_t\}$ of $G$, any $y \in F$, and any $T \geq 0$,
 \begin{equation}
 m_{\R}\{t \in [0,T]: v_t y \in F'  \} \geq (1-\delta)T.
 \end{equation}

We apply the above fact recursively, to produce a sequence of compact sets $(F_j)_{j \in \Z_{\geq 0}}$ contained in $\cG(U)$, starting from $F_0=\{x\}$, so that for any unipotent subgroup $\{v_t\}$ of $G$, any $j \geq 0$, any $y \in F_j$, and any $T \geq 0$,
 \begin{equation}
 m_{\R}\{t \in [0,T]: v_t y \in F_{j+1}  \} \geq (1-\delta)T.
 \end{equation}

Then, using \eqref{eq: 2: lem: zero measure of sub hom spaces}, we have
\begin{align}
\label{eq: 3: lem: zero measure of sub hom spaces}
    \nu(\cS(U)) \leq \nu(G/\Gamma \setminus F_{k}) \leq \lim_{T \rightarrow \infty} \frac{1}{T^k} \; m_{\R^k} \{t \in [0,T]^k: u_t x \notin F_k\}.
\end{align} 
Let $u^{(i)}_{t_i}$ denote the element $u_{t}$ where $t \in \R^k$ has all zero entries except the $i$-th entry which equals $t_i$. Then one checks that the set
\begin{equation}
\{t \in [0,T]^k: u_t x \notin F_k\}
\end{equation}
is contained in the sets $\Omega_1 \cup \ldots \cup \Omega_k$, where $\Omega_j$ equals the set of all $t=(t_1, \ldots, t_k) \in [0,T]^k$ such that 
\begin{align}
    u_{t_{j-1}}^{(j-1)} \cdots u_{t_1}^{(1)} x \in F_{j_1}, \qquad
  u_{t_{j}}^{(j)} \cdots u_{t_1}^{(1)} x \notin F_{j}.
\end{align}
By Fubini's theorem and the definition of $F_j$, one notes that
\begin{align}
\label{eq: 4: lem: zero measure of sub hom spaces}
     m_{\R^k}(\Omega_j) \leq T^{k-1} \sup_{x \in F_{j-1}} \left( m_{\R}\{t \in [0,T] : u_t^{(j)}x \notin F_j \}\right) \leq \delta T^k.
\end{align}
Combining \eqref{eq: 3: lem: zero measure of sub hom spaces} and \eqref{eq: 4: lem: zero measure of sub hom spaces} implies that 
\begin{align}
    \nu(\cS(U)) \leq k \delta.
\end{align}
Since $k$ is fixed and $\delta>0$ is arbitrary, the lemma follows.
\end{proof}

Using Lemma~\ref{lem: zero measure of sub hom spaces}, we see that $\widetilde{\mu}_i$ can be taken to be the average of the pushforwards $g_*\mu_i$ over the small ball $B_{\varepsilon_i}(e)$ in $G$, normalized by the Haar measure of this ball, for sufficiently small $\varepsilon_i$. This is because the convolved measure is absolutely continuous with respect to $\nu$, and hence gives zero measure to the set $\cS(U)$ by Lemma~\ref{lem: zero measure of sub hom spaces}. Also, if $\e_i$ converges to zero rapidly, the measures $(\lambda_i)_i$ and~\eqref{eq: def measures appendix 2} have the same limit. Now to apply Theorem~\ref{thm: DM93 Thm 1}, we will use the following lemma and the fact that $\cS(U)$ is $\sigma$-compact.

\begin{lem}
\label{lem: sequence zero measure sigma compact}
    Suppose $\mu_i$ is a sequence of Borel probability measures on a locally compact second countable metric space $X$ converging weakly to a Borel probability measure $\mu$. Suppose $K$ is a $\sigma$-compact subset of $X$ such that 
    \begin{align}
        \label{eq: lem: sequence zero measure sigma compact 1}
         \mu(K)= \mu_i(K) =0,
    \end{align}
     for all $i \in \N$. Then for every $\e>0$, there exists an open subset $O$ of $X$ containing $K$ such that
    \begin{equation}
    \mu(O) \leq \e, \quad  \mu_i(O) \leq \e,
    \end{equation}
    for all $ i \in \N$.
\end{lem}
\begin{proof}
   {\bf Case 1:} Assume that $K$ is compact. Note that since $X$ is a locally compact second countable metric space, we see that $\mu$ and each $\mu_i$ is regular. Therefore using \eqref{eq: lem: sequence zero measure sigma compact 1}, there exists an open set $O_1$ containing $K$ such that 
   \begin{equation}
   \mu(O_1)< \frac{\e}{2}.
   \end{equation}
   Now, using the fact that the space $X$ is regular, we fix an open set $O_2$ such that
   \begin{equation}
   K \subset O_2 \subset \overline{O_2} \subset O_1.
   \end{equation}
    Using the definition of weak-convergence, we know that there exists a $j_0 \in \N$ such that for all $j > j_0$, we have
    \begin{equation}
    \mu_j(O_2)\leq \mu_j(\overline{O_2}) \leq \mu(\overline{O_2})+ \frac{\e}{2} \leq  \mu(O_1)+ \frac{\e}{2} \leq \e. 
    \end{equation}
    Also, using regularity of $\mu_1, \ldots, \mu_{j_0}$, there exist open subsets $O_1', \ldots, O_{j_0}'$ containing $K$ such that
    \begin{equation}
    \mu_j(O_j') \leq \e,
    \end{equation}
    for all $j=1, \ldots, j_0$. The lemma in this case now follows by taking 
    \begin{equation}
    O= O_2 \cap O_1' \cap  \cdots \cap O_{j_0}'. 
    \end{equation}
\medskip
    
   {\bf Case 2:} Assume that $K$ is not compact. Using $\sigma$-compactness of $K$, write $K$ as countable union of compact sets $\{K_j: j \in \N\}$. Using Case 1 for each $K_j$, we find an open set $O_j$ containing $K_j$ such that
    \begin{equation}
    \mu(O_j) \leq \frac{\e}{2^j}, \quad  \mu_i(O_j) \leq \frac{\e}{2^j},
    \end{equation}
    for all $ i \in \N$. The lemma in this case follows by taking $O= \bigcup_j O_j$. Hence, the lemma follows.
\end{proof}

Using Lemma~\ref{lem: sequence zero measure sigma compact} and Theorem~\ref{thm: DM93 Thm 1}, one can show that any limit measure of $(\widetilde\lambda_i)_i$,  and therefore of $(\lambda_i)_i$, must give zero measure to $\cS(U)$. This, combined with earlier discussion, implies that the limit must equal $\nu$, which completes the main steps of proof of Theorem~\ref{thm: gen Dani Margulis}. We now proceed with the formal proof.

\begin{proof}[Proof of Theorem~\ref{thm: gen Dani Margulis}] 
Fix a right invariant metric on the group $G$, and denote it by $d_G(\cdot, \cdot)$. Also fix a Haar measure on $G$, denote it by $m_G$. We need to show that the limit of the sequence of measures $(\lambda_i)_i$ exists and equals $\nu$. By the Banach--Alaoglu theorem, $(\lambda_i)_i$ is relatively compact in the vague topology. Hence, every subsequence of $(\lambda_i)_i$ admits a further subsequence converging vaguely to a measure $\lambda_\infty$ with total mass $\lambda_\infty(G/\Gamma)\leq 1$. To prove convergence, it therefore suffices to show that every such limit measure equals $\nu$. Accordingly, by passing to a subsequence, we may assume that $\lambda_i \to \lambda_\infty$ vaguely. We will show that $\lambda_\infty = \nu$. It will be convenient to denote by $u^{(i,j)}_{t_j}$ the element $u_{t}^{(i)}$ where $t \in \R^k$ has all zero entries except the $j$-th entry which equals $t_j$.

    \medskip

    {\bf Step 1: Control escape of mass.} We first show that $\lambda_\infty$ is a probability measure. 
Let $\delta>0$ be given. Fix a subset $F_0$ of $G/\Gamma$ such that 
    \begin{equation}\mu(F_0)> 1-\delta, \quad \mu_i(F_0)> 1-\delta \quad \text{ for all } i.\end{equation}
    Using Theorem~\ref{thm: DM93 Thm 6.1} iteratively, we find a sequence of compact subsets $(F_j)_{j \in \N}$ such that for every $j \geq 0$, for any unipotent one-parameter subgroup $\{v_t\}$ of $G$, any $x \in F_j$, and any $T \geq 0$, 
    \begin{equation}
    m_{\R} \left\{ t \in [0,T]: v_tx \in F_{j+1}\right\}  \geq (1-\delta) T.
    \end{equation}
   Then one checks that the set $\{(x,t) \in G/\Gamma \times I_i: u_t^{(i)} x \notin F_k\}$ is contained in the sets $\Omega_0({i}) \cup \cdots \cup \Omega_{k}(i)$, where for $j \geq 1$, $\Omega_j(i)$ denotes the set of all $(x,t)$ in $G/\Gamma \times I_i$ such that
    \begin{align}
        u^{(i,j-1)}_{t_{j-1}} \cdots  u^{(i,1)}_{t_{1}} x \in F_{j-1} , \qquad
        u^{(i,j)}_{t_{j}} \cdots  u^{(i,1)}_{t_{1}} x \notin F_{j}, 
    \end{align}
    and $\Omega_0(i)$ denote the set of all $(x,t)$ in $G/\Gamma \times I_i$ such that
   $x \notin F_{0}$. 
 
    Using Fubini's theorem and the definition of $F_j$, we know that $\mu_i \otimes m_{\R^k} (\Omega_{0}(i)) \leq \delta \cdot m_{\R^k}(I_i)$ and for all $1 \leq j \leq k$ and $i \in \N$, we have
        \begin{align}
        \mu_i \otimes m_{\R^k} (\Omega_{j}(i)) 
        &\leq \frac{m_{\R^k}(I_i)}{|\beta_j^{(i)}-\alpha_j^{(i)}|} \bigg( \sup_{x\in F_{j-1}} \left(m_{\R}\{ t \in [0, -\alpha_j^{(i)}]: u_{-t}^{(i,j)} x \notin F_j\} \right) \notag \\ & \qquad\qquad\qquad\qquad + \sup_{x\in F_{j-1}} \left(m_{\R} \{ t \in [0,\beta_j^{(i)}]: u_{t}^{(i,j)} x \notin F_j\} \right) \bigg) \notag \\
        &\leq \delta m_{\R^k}(I_i),
    \end{align}
    for all $1 \leq j \leq k$. Therefore,
    \begin{align}
      \lambda_\infty(G/\Gamma) &\geq \lambda_\infty(F_k) \geq \lim_{i \rightarrow \infty}  \frac{1}{m_{\R^k}(I_i)} \int_{I_i} \int_{G/\Gamma} \ind_{F_k}(u^{(i)}(t) x) \, d\mu_i(x) \, dt \notag \\
      &\geq 1- \lim_{i \rightarrow \infty} \sum_{j=0}^k \frac{1}{m_{\R^k}(I_i)} \int_{I_i} \int_{G/\Gamma} \ind_{\Omega_j(i)}(x, t) \, d\mu_i(x) \, dt  \notag \\
      &\geq 1- (k+1) \delta.
    \end{align}
    Since $\delta>0$ is arbitrary, we see that $\lambda_\infty(G/\Gamma)=1$.     
    
\medskip
    {\bf Step 2: Establish unipotent invariance.} We now show that $\lambda_\infty$ is invariant under the action of $U$. To see this, fix $s \in \R^k$ and a compactly supported smooth function $f$ on $G/\Gamma$. Then 
    \begin{align}
        \left| \lambda_\infty(f \circ u_s) - \lambda_\infty(f) \right| &\leq \left| \lambda_\infty(f \circ u_s) - \lambda_i(f \circ u_s) \right| + \left| \lambda_i(f \circ u_s) - \lambda_i(f \circ u^{(i)}_s)  \right| \nonumber \\
        &+ \left| \lambda_i(f \circ u^{(i)}_s)  - \lambda_i(f) \right| + \left| \lambda_i(f) - \lambda_\infty(f) \right|. \label{eq: yay 1}
    \end{align}
    Since $\lambda_i \rightarrow \lambda_\infty$, we see that the first and fourth terms converge to zero as $i \rightarrow \infty$. For the second term, note that
    \begin{align}
        \left| \lambda_i(f \circ u_s) - \lambda_i(f \circ u^{(i)}_s)  \right| \leq \|f\|_{C^1} d_G(u^{(i)}_s, u_s).
    \end{align}
    Since $u^{(i)}_s \rightarrow u_s$, we see that the second term also converges to zero as $i \rightarrow \infty$.
    For the third term, note that 
    \begin{align}
        \left| \lambda_i(f \circ u^{(i)}(s))  - \lambda_i(f) \right| &\leq  \frac{1}{m_{\R^k}(I_i)} \int_{G/\Gamma}  \left|   \int_{I_i + s} f(u^{(i)}(t) x) \, d\mu_i(x) \, dt -  \int_{I_i} f(u^{(i)}(t) x) \, d\mu_i(x) \right| \notag \\
        & \leq \frac{m_{\R^k}((I_i+s) \setminus I_i) + m_{\R^k}( I_i \setminus (I_i+s) ) }{m_{\R^k}(I_i)} \|f\|_{C^0} \rightarrow 0,
    \end{align}
    as $i \rightarrow \infty.$ Thus, on taking limit $i \rightarrow \infty$ in \eqref{eq: yay 1}, we see that 
    \begin{equation}
     \lambda_\infty(f \circ u_s) = \lambda_\infty(f).
    \end{equation}
    Since $s \in \R^k$ and the compactly supported smooth function $f$ on $G/\Gamma$ were arbitrary, we see that $\lambda_\infty$ is invariant under $U$.

\medskip

    {\bf Step 3: Reduction to measures satisfying $\mu_i(\cS(U))=0$.} In this step we show that without loss of generality, we may assume that $\mu_i$ satisfies $\mu_i(\cS(U))=0$ for all $i$. 
    To prove this, we fix $\e_i>0$, for each $i \in \N$, small enough so that for all $g \in G$ with $d_G(g,e) < \e_i$, we have
    \begin{equation}
    d_G(u^{(i)}_t g, u^{(i)}_t) \leq \frac{1}{2^i},
    \end{equation}
    for all $t \in I_i$. Then, define
    \begin{equation}
    \widetilde{\mu}_i = \frac{1}{m_G(B_{\e_i}(e))} \int_{B_{\e_i}(e)} g_*\mu_i \, dm_G(g),
    \end{equation}
    where $B_\e(e)= \{g \in G: d_G(g,e)< \e \}$. Note that for any smooth function $f $ on $G/\Gamma$, we have
    \begin{align}
        &\left| \frac{1}{m_{\R^k}(I_i)} \int_{I_i} (u^{(i)}_t)_*{\mu}_i(f) \, dt- \frac{1}{m_{\R^k}(I_i)} \int_{I_i} (u^{(i)}_t)_*\widetilde{\mu}_i(f) \, dt \right| \notag \\
        &\leq \frac{1}{m_{\R^k}(I_i)}  \int_{I_i} \int_{G/\Gamma} \left|  f(u^{(i)}_t x ) - \frac{1}{m_G(B_{\e_i}(e))} \int_{B_{\e_i}(e)} f(u^{(i)}_t g x )  \,dm_G(g)  \right| \, d\mu_i(x) dt \notag \\
        &\leq \frac{1}{m_{\R^k}(I_i)}  \int_{I_i} \int_{G/\Gamma} \|f\|_{C^1} d_G(u^{(i)}_t g, u^{(i)}_t ) \, d\mu_i(x) dt \leq  \frac{\|f\|_{C^1}}{2^i} \rightarrow 0.
    \end{align}
Therefore, the limit of \eqref{eq: def measures appendix 2} exists and equals $\lambda_\infty$. Along the same lines, one verifies that $\widetilde{\mu}_i \to \mu$ as $i \to \infty$. We have thus shown that the conclusion of Theorem~\ref{thm: gen Dani Margulis} holds for $(\mu_i)_i$ if it holds for  $(\widetilde{\mu}_i)_i$. It therefore suffices to prove Theorem~\ref{thm: gen Dani Margulis} for $(\widetilde{\mu}_i)_i$.

Since, by definition, $\widetilde{\mu}_i$ is absolutely continuous with respect to $\nu$, Lemma~\ref{lem: zero measure of sub hom spaces} implies that $\widetilde{\mu}_i(\cS(U))=0$ for every $i$. This is the only additional property we require to complete the proof. We will therefore consider in the following measures $\mu_i$ under the previous hypotheses plus $\mu_i(\cS(U))=0$ (and the $\widetilde{\mu}_i$ constructed above would be an example of such a measure).

\medskip

{\bf Step 4: Show that the singular set is null.} We now prove that $\lambda_\infty(\cS(U))=0$.
Fix \(\delta>0\). Since \(\mathcal S(U)\) is \(\sigma\)-compact (see e.g.\ \cite[\S2,3]{DM93}), by Lemma~\ref{lem: sequence zero measure sigma compact} applied with \(K=\mathcal S(U)\) and the fact that $\mu_i$ converges to the probability measure $\mu$, we see that there exists a compact set $F_0$ contained in $\cG(U)$ such that $\mu_i(F_0) \geq 1-\delta$ for all $i$.

Now we use Theorems~\ref{thm: DM93 Thm 6.1} and~\ref{thm: DM93 Thm 1} recursively to get a family of compact subsets $(F_j)_{j \in \N}$ contained in $\cG(U)$ such that for every $j \geq 0$, for any unipotent one-parameter subgroup $\{v_t\}$ of $G$, any $x \in F_j$, and any $T \geq 0$, we have
\begin{equation}
m_{\R}\{t \in [0,T]: v_t x \in F_{j+1}\} \geq (1-\delta)T.
\end{equation}

Then one checks that the set $\{(x,t) \in G/\Gamma \times I_i: u_t^{(i)} x \notin F_k\}$ is contained in the sets $\Omega_0({i}) \cup \cdots \cup \Omega_{k}(i)$, where for $j \geq 1$, $\Omega_j(i)$ denotes the set of all $(x,t)$ in $G/\Gamma \times I_i$ such that
    \begin{align}
        u^{(i,j-1)}_{t_{j-1}} \cdots  u^{(i,1)}_{t_{1}} x \in F_{j-1} ,\qquad 
        u^{(i,j)}_{t_{j}} \cdots  u^{(i,1)}_{t_{1}} x \notin F_{j}, 
    \end{align}
    and $\Omega_0(i)$ denote the set of all $(x,t)$ in $G/\Gamma \times I_i$ such that $x \notin F_{0}$. 
    
    Using Fubini's theorem and the definition of $F_j$, we know that 
    \begin{align}
    \mu_i \otimes m_{\R^k} (\Omega_{0}(i)) \leq \delta \cdot m_{\R^k}(I_i) .
    \end{align}
     Furthermore, for all $1 \leq j \leq k$ and $i \in \N$, we have
        \begin{align}
        &\mu_i \otimes m_{\R^k} (\Omega_{j}(i)) \notag \\
        &\leq \frac{m_{\R^k}(I_i)}{|\beta_j^{(i)}-\alpha_j^{(i)}|} \bigg( \sup_{x\in F_{j-1}} \left(m_{\R}\{ t \in [0, -\alpha_j^{(i)}]: u_{-t}^{(i,j)} x \notin F_j\} \right) \\ & \qquad\qquad\qquad\qquad + \sup_{x\in F_{j-1}} \left(m_{\R}\{ t \in [0,\beta_j^{(i)}]: u_{t}^{(i,j)} x \notin F_j\} \right) \bigg) \notag \\
        &\leq \delta m_{\R^k}(I_i),
    \end{align}
    for all $1 \leq j \leq k$. Therefore,
    \begin{align}
      \lambda_\infty(\cG(U)) &\geq \lambda_\infty(F_k) \geq \lim_{i \rightarrow \infty}  \frac{1}{m_{\R^k}(I_i)} \int_{I_i} \int_{G/\Gamma} \ind_{F_k}(u^{(i)}(t) x) \, d\mu_i(x) \, dt \notag \\
      &\geq 1- \lim_{i \rightarrow \infty} \sum_{j=0}^k \frac{1}{m_{\R^k}(I_i)} \int_{I_i} \int_{G/\Gamma} \ind_{\Omega_j(i)}(x,t) \, d\mu_i(x) \, dt  \notag \\
      &\geq 1- (k+1) \delta.
    \end{align}
    Since $\delta>0$ is arbitrary, we see that $\lambda_\infty(\cG(U)) =1$. In other words, we have $\lambda_\infty(\cS(U))=0$. 

    \medskip

{\bf Step 5: Conclusion.} Using Ratner's theorem \cites{Ratner91A} and ergodic decomposition theorems for finite invariant measures, the previous observations (Step 1, 2 and 4) imply that $\lambda_\infty$ equals $\nu$ and, as noted earlier, this completes the proof. Hence, the theorem follows.
\end{proof}

\section{Limit distributions for slowly expanding horospheres}\label{sec:Main}
For $n \geq 1$, define
\begin{align}
    G_{n+1}= \SL_{n+1}(\R), \quad \Gamma_{n+1} = \SL_{n+1}(\Z),
\end{align}
\begin{align}
    \widetilde{G}_{n+1} = \SL_{n+1}(\R)\ltimes\R^{n+1}, \quad \widetilde{\Gamma}_{n+1}= \SL_{n+1}(\Z)\ltimes\Z^{n+1},
\end{align}
where the multiplication law of $\widetilde{G}_{n+1}$ is given by considering $\widetilde{G}_{n+1}$ as subgroup of $\SL_{n+2}(\R)$ via the homomorphism
\begin{equation}
[A,b] \mapsto \begin{pmatrix}
    A & b \\ 0 & 1
\end{pmatrix}.
\end{equation}
That is, 
\begin{equation}
[A_1,b_1] [A_2,b_2] =[A_1A_2,b_1+A_1 b_2] . 
\end{equation}
Define the homogeneous spaces
\begin{equation}
\X_{n+1}=G_{n+1}/\Gamma_{n+1},\qquad \tXn = \widetilde{G}_{n+1}/\widetilde{\Gamma}_{n+1},
\end{equation}
and denote by $\mu_{\X_{n+1}}$ and $\mu_{\tXn}$ the unique $G_{n+1}$- and $\widetilde{G}_{n+1}$-invariant probability measures on $\X_{n+1}$ and $\tXn$, respectively.

For $s\in \R^n$ and $t>0$, define
\begin{align}
    \label{eq: def u n}
    u(s)= u_{n+1}(s)= \begin{pmatrix}
        \id_n & s \\ 0 & 1
    \end{pmatrix}, \quad a(t)= a_{n+1}(t)= \begin{pmatrix}
        t^{1/n}\id_n & 0 \\ 0 & t^{-1}
    \end{pmatrix}.
\end{align} 
Throughout this paper, we will consider $G_{n+1}$ as subgroup of $\widetilde{G}_{n+1}$ via the homomorphism
\begin{equation}
g \mapsto [g, 0],
\end{equation}
and therefore both $u(s), a(t)$ will be viewed as elements of $G_{n+1}$ or $\widetilde{G}_{n+1}$ as appropriate.

Let $\tHn$ denote the subgroup
\begin{equation}
\tHn= \left\{ \left[ \begin{pmatrix}
    A & 0 \\ {}^{\top}\! w & 1
\end{pmatrix}, \zeta \right]: A \in \SL_n(\R), w \in \R^n, \zeta \in \R^{n+1} \right\}.
\end{equation}
The group $\tHn$ is generated by unipotent elements and therefore, by Ratner's theorem~\cites{Ratner91A}, for every $x \in \tXn$, we have that the closure of the set $\tHn \cdot x$ is a homogeneous space and carries a unique $\tHn$-invariant probability measure. Throughout this paper, we will denote this measure by $ \mu_{\tHn \cdot x}$.

\medskip

Define the measure $\tnu_{Q,x}=\tnu_{Q,\Delta,\epsilon,x,\lambda}$ on $\tXn$ by 
     \begin{align}  \label{eq: def nu}
    \tnu_{Q,x}(f) = \int_{\R^n} f\left( \left[ a(\Delta)   u(-\e \xi), \begin{pmatrix}
            { \zer} \\ -Q\Delta^{-1}
        \end{pmatrix} \right] x \right) \, d\lambda(\xi),
        \end{align}
       with $f:\tXn \to \R$ bounded continuous. 
Similarly, define $\mu^{(\sigma)}_{\lambda,x}$ by
 \begin{align}
\label{eq: thm : main dynamical 2}
 \mu^{(\sigma)}_{\lambda,x}(f) &= 
 \int_{\R^n} \int_{\tXn} f(  M_{\sigma \xi} \, y)\, d\mu_{\tHn \cdot x}(y) \ d{\lambda}(\xi),
 \end{align}
 where $M_\xi\in\SL(n+1,\R)$ is any choice of matrix such that
 \begin{equation}\label{Mxi}
M_\xi \begin{pmatrix} 0 \\ 1 \end{pmatrix} = - \begin{pmatrix} \xi \\ 0 \end{pmatrix} .
\end{equation}
The $\tHn$-invariance of $\mu_{\tHn \cdot x}$ implies that the definition \eqref{eq: thm : main dynamical 2} is independent of the choice $M_\xi$. (To see this, let $M_\xi'$ be a second matrix satisfying \eqref{Mxi}. Then 
 \begin{equation}\label{Mxi2}
M_\xi^{-1} M_\xi' \begin{pmatrix} 0 \\ 1 \end{pmatrix} = \begin{pmatrix} 0 \\ 1 \end{pmatrix} ,
\end{equation}
and thus 
\begin{equation}\label{Mxi3}
M_\xi' = M_\xi \begin{pmatrix}
    A_0 & 0 \\ {}^{\top}\! w_0 & 1
\end{pmatrix}
\end{equation}
for suitable $A_0,w_0$.) A convenient choice of $M_\xi$ is 
 \begin{equation}
M_\xi = 
\begin{pmatrix}
     \id_{n-1}  &0 & -\xi' \\ 0&0& -\xi_n \\ 0& \xi_n^{-1} & 0 
     \end{pmatrix}  
     ,\qquad \xi=\begin{pmatrix} \xi' \\ \xi_n \end{pmatrix},
\end{equation}
if $\xi_n \neq 0$. In the following $\lambda$ is assumed to be absolutely continuous with respect to Lebesgue measure on $\R^n$. The above choice therefore gives a complete parametrisation on a set of full measure.

\medskip

The following equidistribution theorem is the main dynamical result of the present study. It serves as the key input for all limit theorems stated in the introduction.

\begin{thm}
    \label{thm: main dynamical theorem}
     Let $\alpha \in \R^n$ and $\lambda$ a probability measure on $\R^n$ absolutely continuous with respect to Lebesgue measure, and $\sigma\in (0,\infty]$. Fix a sequence of $(Q,\Delta,\epsilon)$ satisfying \eqref{scalim} and $\e \Delta^{1/n}\to \sigma$. Then for $x= u(-\alpha)\widetilde{\Gamma}_{n+1}$, we have:

\begin{enumerate}
\renewcommand{\labelenumi}{\rm (\roman{enumi})}
\item 
If  $\alpha \notin \Q^n$ or $\sigma=\infty$, then 
\begin{equation}\label{infact}
 \tnu_{Q,x} \tovague \mu_{\tXn}.
 \end{equation}

\item
If $\alpha \in \Q^n$ and $\sigma<\infty$, then $\tHn \cdot x$ is closed and it satisfies
\begin{equation}\label{infact 2}
 \tnu_{Q,x} \tovague \mu^{(\sigma)}_{\lambda,x} .
\end{equation} 
\end{enumerate}
 Moreover for $\alpha = p/q \in \Q^n$ with $\gcd(p,q)=1$, 
 \begin{equation}
 \label{eq:rational case}
     \mu^{(\sigma)}_{\lambda,x} = \mu^{(q\sigma)}_{\lambda,\widetilde{\Gamma}_{n+1}},
 \end{equation}
 and
 \begin{equation}
 \label{eq; convergence of rational case}
   \mu^{(\sigma)}_{\lambda,x} \tovague \mu_{\tXn},
   \end{equation}
   as $\sigma \to \infty$.
\end{thm}

Convergence is understood here in the vague topology, i.e., for test functions $f\in C_c(\tXn)$. Since the limits are probability measures, the above convergence also holds in the weak topology, i.e., for all bounded continuous test functions. 

\begin{rem}
The convergence of the pushforward of $\mu^{(\sigma)}_{\lambda,x}$ under the natural projection $\pi:\tXn \to \X_{n+1}$ has been studied in several settings. In the case $n=1$ and $\epsilon = \Delta^{1/2+\delta}$, Hejhal \cites{Hejhal} and Strömbergsson \cites{Str04} studied the resulting distribution and proved convergence to $\mu_{\X_{n+1}}$. More recently, Shah and Yang \cite{ShahYang} considered a more general setting in which $\lambda$ is supported on a non-degenerate submanifold, and studied this construction on general homogeneous spaces in the finite-$\sigma$ regime. They prove an abstract result which, as a consequence, yields convergence of $\pi_*(\mu^{(\sigma)}_{\lambda,x})$. Theorem~\ref{thm: main dynamical theorem} refines this picture by giving an explicit description of the limit measures, and by extending the analysis to $\tXn$.
\end{rem}

\begin{rem}
The probability measure $\mu^{(\sigma)}_{\lambda,x}$ is invariant under the action of the horospherical subgroup $U=\{ u(s) : s\in\R^n\}$. This follows from the fact that $U$ stabilises $(\begin{smallmatrix} \xi \\ 0 \end{smallmatrix})$, and hence
\begin{equation}
u(s) M_\xi = M_\xi 
\begin{pmatrix}
    A_0 & 0 \\ {}^{\top}\! w_0 & 1
\end{pmatrix}
\end{equation}
for some $A_0,w_0$.
Note, however, that the $U$-action is {\em not} ergodic unless the closure of $\tHn\cdot x$ is $\tXn$. The ergodic decomposition is apparent from definition \eqref{eq: thm : main dynamical 2}. There is no contradiction with the Mozes-Shah Theorem \cite{MozesShah}, since the measures $\tnu_{Q,x}$ (with limit point $\mu^{(\sigma)}_{\lambda,x}$) are not $U$-invariant.
\end{rem}

\begin{rem}
    For $n=1$ and $\alpha=0$ (i.e., for $x=\widetilde{\Gamma}_2$), the measure
$\mu^{(\sigma)}_{\lambda,x}$ is given explicitly by
    \begin{align}
    \label{eq:def mu sigma}
        \mu^{(\sigma)}_{\lambda,x}(f) &= \int_{\R} \int_{\T} \int_{\T^2}f\left(  \begin{pmatrix}
       0 & - \sigma\xi\\ (\sigma\xi)^{-1} &0
    \end{pmatrix} \begin{pmatrix}
        1 &0 \\ v &1
    \end{pmatrix} \left[\id_2, \zeta \right] \widetilde{\Gamma}_2 \right) \,\, d\zeta \,  dv \, d\lambda(\xi) \notag \\
    &= \int_{\R} \int_{\T} \int_{\T^2}f\left(  \begin{pmatrix}
        \sigma \xi &0\\  0&(\sigma \xi)^{-1}
    \end{pmatrix}  \begin{pmatrix}
        1 & v  \\ 0&1
    \end{pmatrix} \left[\id_2, \zeta \right]\widetilde{\Gamma}_2  \right) \, d\zeta \,  dv \, d\lambda(\xi),
    \end{align}
    where $\T=\R/\Z$. The last equality follows from 
    \begin{equation} \widetilde{\Gamma}_2=  \begin{pmatrix}
        0 &1\\ -1 &0
    \end{pmatrix}\widetilde{\Gamma}_2.\end{equation}
    and the fact that $dv$ and $d\zeta$ are invariant under 
    \begin{equation}
v\mapsto -v, \qquad \zeta \mapsto \begin{pmatrix}
        0 &1\\ -1 &0
    \end{pmatrix} \zeta .
\end{equation}
\end{rem}

\medskip

The proof of Theorem~\ref{thm: main dynamical theorem} is divided into several steps, which we outline below.

\subsection*{Step A: Reduction to $\SL_{n+1}(\R)/\SL_{n+1}(\Z)$}
\label{subsec: Reduction to SL nR}
Let
\begin{align}
    H_{n+1} = \left\{\begin{pmatrix}
    A & 0\\ {}^{\top}\! w & 1
\end{pmatrix}: A \in \SL_n(\R), w \in \R^n \right\}.
\end{align}
As for $\tHn$, the group $H_{n+1}$ is generated by unipotent elements in $\tHn$, and therefore by Ratner's theorem~\cites{Ratner91A}, for every $x \in \X_{n+1}$, we have that the closure of the set $H_{n+1} \cdot x$ is a homogeneous space and carries a unique $H_{n+1}$-invariant probability measure. Throughout this paper, we will denote this measure by
$\mu_{H_{n+1} \cdot x}$.

The first step in the proof of Theorem~\ref{thm: main dynamical theorem} is to show that it is enough to study the sequence of measures $\nu_{Q,x}$ on $\X_{n+1}$ defined by 
 \begin{align}  \label{eq: def nu small}
    \nu_{Q,x}(f) = \int_{\R^n} f\left( a(\Delta)   u(-\e \xi) x \right) \, d\lambda(\xi),
        \end{align}
        for $f \in C_c(\X_{n+1})$.
This is achieved by the following proposition, which is proved in Section~\ref{sec: Reduction to X}. 

\begin{prop}
    \label{lem: reduction to X}
     Let $\alpha \in \R^n$ and $\lambda$ a probability measure on $\R^n$ absolutely continuous with respect to Lebesgue measure, and $\sigma\in (0,\infty]$. Fix a sequence of $(Q,\Delta,\epsilon)$ satisfying \eqref{scalim} and $\e \Delta^{1/n}\to \sigma$. Assume that for every $x= u(-\alpha)\Gamma_{n+1}$, we have
     \begin{align}
     \label{eq:lem:reduction to X}
        \nu_{Q,x}(f) \to \begin{cases}
             \displaystyle
          \int_{\R^n} \int_{\X_{n+1}} f(  M_{\sigma \xi} \, y)\, d\mu_{H_{n+1} \cdot x}(y) \ d{\lambda}(\xi), \quad &\text{ if } \sigma < \infty \text{ and } \alpha \in \Q^n, \\[10pt]
           \mu_{\X_{n+1}}(f), \quad& \text{ otherwise, }
         \end{cases}
     \end{align}
   for all $f \in C_c(\X_{n+1})$. Then the conclusions \eqref{infact} and \eqref{infact 2} of Theorem~\ref{thm: main dynamical theorem} hold.
    \end{prop}

\subsection*{Step B: Joint equidistribution}
The goal of Step B is to prove Proposition~\ref{lem: reduction to X} for $\sigma< \infty$. The key observation in proving the latter is the following identity:
\begin{align}
   \nonumber 
     a(\Delta) u( - \e \xi ) &=  \begin{pmatrix}
        u_n(\xi' /\xi_n)&0  \\ 0 & 1
    \end{pmatrix}  \begin{pmatrix}
        \id_{n-1} &0&0 \\0 &0   &  -\e \Delta^{1/n}\xi_n   \\ 0& (\e \Delta^{1/n} \xi_n)^{-1} &0 
    \end{pmatrix} \begin{pmatrix}
        \id_{n-1} &0&0 \\ 0& 1      &  \e \Delta^{-(n-1)/n} \xi_n \\0 &0 &1
    \end{pmatrix}  \nonumber\\
    & \times \begin{pmatrix}
        \id_{n-1} &0&0 \\ 0&   1&0     \\ 0& -\e^{-1} \Delta^{(n-1)/n} \xi_n^{-1}& 1 
    \end{pmatrix}  \begin{pmatrix}
        a_n(\Delta^{(n-1)/n}) &0 \\ 0 & 1
    \end{pmatrix} \begin{pmatrix}
         u_n(-\xi' /\xi_n)&0  \\0 & 1
    \end{pmatrix}      .      \label{eq: lem reduction 1}
\end{align}
Let $\widetilde{\lambda}$ denote the measure on $\R^n$ obtained via pushforward of $\lambda$ under the map
\begin{equation}
\xi= (\xi', \xi_n)\mapsto (\xi'/\xi_n, \xi_n).
\end{equation}
It is clear that $\widetilde{\lambda}$ is absolutely continuous 
with respect to Lebesgue measure.

Using equality~\eqref{eq: lem reduction 1}, we see that 
equation~\eqref{eq:lem:reduction to X} holds for $\sigma<\infty$ once we show that
\begin{align}
\label{eq:lem: simplification via joint 2}
\int_{\R^n}
F\!\left(\zeta,\e\Delta^{1/n}\beta,
\frac{\e\beta}{\Delta^{(n-1)/n}},
\begin{pmatrix}
        a_n(\Delta^{(n-1)/n}) &0 \\ 0 & 1
    \end{pmatrix} \begin{pmatrix}
         u_n(-\zeta)&0  \\0 & 1
    \end{pmatrix} x
\right)
\,d\widetilde{\lambda}(\zeta,\beta)
\end{align}
converges to
\begin{align}
\label{eq:lem: simplification via joint 3}
\int_{\R^n}
\int_{\X_{n+1}}
F\!\left(\zeta,\sigma\beta,0,y\right)
\,d\nu_{H_{n+1}\cdot x}(y)\,d\widetilde{\lambda}(\zeta,\beta)
\end{align}
for all $F\in C_c(\R^{n-1}\times\R\times\R\times\X_{n+1})$.

Using the techniques of \cite{MarStro}, the above convergence follows once 
we prove convergence in the $\X_{n+1}$ factor alone, for all 
absolutely continuous measures $\widetilde{\lambda}$ on $\R^n$. In particular, we will 
prove the following proposition.
\begin{prop}
\label{lem: simplification via joint}
   Fix $x \in \X_{n+1}$ and a sequence $(\e,\Delta)$ satisfying $\e \Delta^{1/n} \to \sigma < \infty$. Assume that for every probability measure $\lambda_0$ on $\R^{n-1}$ absolutely continuous with respect to Lebesgue measure, and any $0<b<c<\infty$, we have
    \begin{align}
        \label{eq:lem: simplification via joint}
        \frac{1}{c-b}\int_{\R^{n-1}} \int_b^c f\left( \begin{pmatrix} 
        \id_{n-1} &0&0\\ 0&   1&0     \\ 0& -\e^{-1} \Delta^{(n-1)/n} \beta^{-1}& 1 
    \end{pmatrix}  \begin{pmatrix}
        a_n(\Delta^{(n-1)/n}) &0 \\ 0 & 1
    \end{pmatrix} \begin{pmatrix}
         u_n(-\zeta)&0  \\0 & 1
    \end{pmatrix} x
\right)
 \, d\lambda_0(\zeta) d\beta ,
    \end{align}
    converges to $\mu_{H_{n+1}\cdot x}(f)$ for every $f \in C_c(\X_{n+1})$. 
    Then \eqref{eq:lem:reduction to X} holds for $x$ and for the sequence $(\e,\Delta)$.
\end{prop}

\subsection*{Step C: Proof of the case $\sigma<\infty$}
In this step, we prove the following proposition.

 \begin{prop}
 \label{lem: final proof}
     Fix $\alpha \in \R^n$, and a sequence $(\e,\Delta)$ satisfying $\e \Delta^{1/n} \to \sigma < \infty$. Let $x= u(-\alpha) \Gamma_{n+1} $. Then for every probability measure $\lambda_0$ on $\R^n$ absolutely continuous with respect to Lebesgue measure, and any $0<b<c<\infty$, we have that the sequence of measure in~\eqref{eq:lem: simplification via joint} converges vaguely to $\mu_{H_{n+1}\cdot x}$. Moreover for $\alpha \notin \Q^n$, we have
    \begin{align}
    \label{eq:lem: final proof 2}
         \mu_{H_{n+1}\cdot x}= \mu_{\X_{n+1}}.
    \end{align}
 \end{prop}

The proof of the proposition is divided into two parts, depending on whether $\alpha \in \Q^n$ or not. The key ingredients of the proof are the use of the well-known Margulis thickening technique~\cites{KM1,Margulis}, together with a generalised version of the Dani–Margulis theorem proved in Section~\ref{sec: Gen Dan Mar}.

 \subsection*{Step D: Proof of the case $\sigma=\infty$}

The final step in the proof of Theorem~\ref{thm: main dynamical theorem} is to show that Proposition~\ref{lem: reduction to X} holds for $\sigma=\infty$. More precisely, we prove the following proposition.

\begin{prop}
\label{lem: enough for finite}
Fix  $x \in \X_{n+1}$. Suppose that the equation~\eqref{eq:lem:reduction to X} holds for every sequence $(\e,\Delta)$ with $\e \Delta^{1/n} \to \sigma < \infty$, and for every probability measure $\lambda$ absolutely continuous with respect to Lebesgue measure. Then the same conclusion holds whenever $\e \Delta^{1/n} \to \infty$.
\end{prop}

The proof of the proposition is divided into two parts. 
In the first part, we apply the generalised Dani--Margulis theorem (see Theorem~\ref{thm: gen Dani Margulis}). This result shows that if equation~\eqref{eq:lem:reduction to X} holds for a sequence $(\varepsilon,\Delta)$ with $\sigma=\infty$, then it also holds for any sequence $(\varepsilon',\Delta)$ with $\varepsilon' \geq \varepsilon$.
In the second part, we use the fact that for every $f \in C_c(\X_{n+1})$,
\begin{align}
\lim_{\sigma \to \infty} 
\int_{G_{n+1}} 
\int_{\X_{n+1}} 
f(M_{\sigma \xi} \, y)
\, d\mu_{H_{n+1}\cdot x}(y)\,
d\hat{\lambda}(g) 
&=\mu_{\X_{n+1}}(f), \label{eq:lem: enough for finite}
\end{align}
which will follow from Margulis's thickening technique~\cites{KM1,Margulis}.

Using \eqref{eq:lem: enough for finite}, we show that equation~\eqref{eq:lem:reduction to X} holds for any sequence $(\varepsilon,\Delta)$ for which $\varepsilon \Delta^{1/n} \to \infty$ sufficiently slowly. In particular, we construct a function $\varepsilon'(\cdot)$ such that for any sequence $(\varepsilon,\Delta)$ satisfying 
\begin{equation}
\varepsilon < \varepsilon'(\Delta)
\quad\text{and}\quad
\varepsilon \Delta^{1/n} \to \infty,
\end{equation}
equation~\eqref{eq:lem:reduction to X} holds.

Combining these two facts, we conclude that equation~\eqref{eq:lem:reduction to X} holds for any sequence $(\varepsilon,\Delta)$ with $\sigma=\infty$.

\medskip

Theorem~\ref{thm: main dynamical theorem} then follows directly from Propositions~\ref{lem: reduction to X},~\ref{lem: simplification via joint},~\ref{lem: final proof} and~\ref{lem: enough for finite}.

\medskip
\section{Reduction to $\SL_{n+1}(\R)/\SL_{n+1}(\Z)$}
\label{sec: Reduction to X}
The main aim of this section is to prove Proposition~\ref{lem:  reduction to X}. Throughout this section, we assume the hypotheses of Proposition~\ref{lem:  reduction to X}. To proceed, we require the following auxiliary lemmas.

\begin{lem}
\label{lem: reduction 1 0}
Suppose $(F_k)_{k \in \N}$ is a uniformly bounded sequence of measurable functions $\R^n \to \R$, such that the limit
\begin{align}
\label{eq: lem: reduction 1 0: 1}
    \lim_{k \to \infty} \int_{\R^n} F_k(\xi) \, d\lambda(\xi)
\end{align}
exists. Then, for any sequence $(s_k)_{k \in \N} $ in $\R^n$ converging to $0$, we have
\begin{align}
\label{eq: lem: reduction 1 0: 2}
    \lim_{k \to \infty} \int_{\R^n} F_k(\xi + s_k) \, d\lambda(\xi)
    = \lim_{k \to \infty} \int_{\R^n} F_k(\xi) \, d\lambda(\xi).
\end{align}
\end{lem}
\begin{proof}
We first prove the lemma when $\lambda$ equals the normalised Lebesgue measure restricted to some cuboid $O$ of the form $[x_1, x_1'] \times \cdots \times [x_n, x_n']$, that is, when
\begin{equation}
\lambda= \frac{1}{m_{\R^n}(O)} m_{\R^n}|_{O}.
\end{equation}
In this case, 
\begin{align}
& \left| \int_{\R^n} F_k(\xi + s_k) \, d\lambda(\xi)
      - \int_{\R^n} F_k(\xi) \, d\lambda(\xi) \right| \notag \\ &= \frac{1}{m_{\R^n}(O)} 
   \left| \int_O F_k(\xi + s_k) \, d\xi - \int_O F_k(\xi) \, d\xi \right| \nonumber \\
&\le \frac{M}{m_{\R^n}(O)} 
   \big( m_{\R^n}((O+s_k)\setminus O)
       + m_{\R^n}(O\setminus(O+s_k)) \big), \label{eq: lem: aax 1}
\end{align} 
where $M$ denotes a fixed uniform bound on the functions $F_k$. 
Note that since $s_k \to 0$, the Lebesgue measure of the symmetric difference
\begin{equation}
\big( (O+s_k)\setminus O \big) \  \cup  \  \big(O\setminus(O+s_k) \big)
\end{equation}
converges to $0$ as $k\to\infty$. This, combined with \eqref{eq: lem: aax 1}, proves \eqref{eq: lem: reduction 1 0: 2} in the special case, when $\lambda$ equals the normalised Lebesgue measure restricted on some cuboid $O$. The general case now follows by approximating $\lambda$ by finite linear combinations of indicator functions of cuboids, along with uniform boundedness of the $F_k$. Hence the lemma follows.
\end{proof}

\begin{lem}
\label{lem:reduction 1 1}
Fix $x \in \tXn$. Then any subsequential limit of the sequence $\tnu_{Q,x}$ is invariant under the right action of the unipotent group $\widetilde{U}_{n+1}$ defined by
\begin{align}
\label{eq:def U n}
\widetilde{U}_{n+1} = \left\{ \left[\id_{n+1}, 
\begin{pmatrix}
y \\ 0
\end{pmatrix}
\right] : y \in \R^{n} \right\}.
\end{align}
\end{lem}

\begin{proof}
Let $\nu_{\infty,x}$ be a subsequential limit of $\nu_{Q,x}$ along a sequence $(Q_k)_{k \in \N}$. Replacing the sequence $Q$ by the subsequence $Q_k$, we may assume that the limit of $\nu_{Q,x}$ exists as $Q \to \infty$ and equals $\nu_{\infty,x}$.

To prove the lemma, fix a continuous, compactly supported function $f$ on $\widetilde{\X}_{n+1}$ and $\zeta \in \R^n$. Apply Lemma~\ref{lem: reduction 1 0} to the sequence of functions $F_Q$ defined by
\begin{equation}
F_Q(\xi) = f \left(\left[ a(\Delta) u(-\e \xi), 
\begin{pmatrix}
\zer \\ -Q\Delta^{-1}
\end{pmatrix} \right] \right),
\end{equation}
with $s_Q = \zeta (\Delta^{1/n}Q)^{-1}$. This yields
\begin{align}
\label{eq:aax 5}
\nu_{\infty,x}(f) 
= \lim_{Q \rightarrow \infty} 
\int_{\R^n} 
f\left(  
\left[ a(\Delta)  u\left( -\e \left(\xi+ \frac{\zeta}{ \Delta^{1/n} \e Q} \right) \right) , 
\begin{pmatrix}
\zer \\ -Q \Delta^{-1}
\end{pmatrix} \right] x  
\right) 
\, d\lambda(\xi).
\end{align}

A direct computation shows that
\begin{align}
\label{eq:aax 7}
\left[ a(\Delta)  u\left( -\e \left(\xi+ \frac{\zeta}{ \Delta^{1/n} \e Q} \right) \right) , 
\begin{pmatrix}
\zer \\ -Q \Delta^{-1}
\end{pmatrix} \right] 
= g_Q  
\left[ a(\Delta)  u\left( -\e \xi \right) , 
\begin{pmatrix}
\zer \\ -Q \Delta^{-1}
\end{pmatrix} \right],
\end{align}
where
\begin{align}
\label{eq:aax 8}
g_Q:= \left[ u\left( \frac{\zeta\Delta}{Q} \right), 
\begin{pmatrix}
\zeta \\ 0
\end{pmatrix}\right] 
\longrightarrow 
g_\infty := \left[ \id_{n+1}, 
\begin{pmatrix}
\zeta \\ 0
\end{pmatrix}\right].
\end{align}
Combining \eqref{eq:aax 7} with \eqref{eq:aax 5}, we obtain
\begin{align}
\label{eq:aax 10}
\nu_{\infty,x}(f) 
= \lim_{Q \rightarrow \infty} 
\nu_{Q,x} \left( f \circ g_Q \right).
\end{align}
Using the uniform continuity of $f$, we conclude that
\begin{align}
\label{eq:aax 3}
\lim_{Q \rightarrow \infty} 
\nu_{Q,x} \left( f \circ g_Q \right)
= \nu_{\infty,x}(f \circ g_\infty).
\end{align}
Combining \eqref{eq:aax 10}, \eqref{eq:aax 3}, and \eqref{eq:aax 8}, we deduce that
\begin{align}
\label{eq:aax 11}
\nu_{\infty,x}(f)
= \nu_{\infty,x}\left(f \circ 
\left[ \id_{n+1}, 
\begin{pmatrix}
\zeta \\ 0
\end{pmatrix}\right] 
\right).
\end{align}
Since $f$ and $\zeta \in \R^n$ are arbitrary, the lemma follows from \eqref{eq:aax 11}.
\end{proof}

\medskip

We now prove Proposition~\ref{lem: reduction to X}.
\begin{proof}[Proof of Proposition~\ref{lem: reduction to X}]
Fix $\widetilde{x} \in \tXn$ and let $x= \pi(\widetilde{x})$, where 
\begin{equation}
\pi: \tXn \rightarrow \X_{n+1}, \qquad [A,b]\widetilde{\Gamma}_{n+1} \mapsto A \Gamma_{n+1}, \quad A \in \SL_{n+1}(\R), \ b \in \R^{n+1},
\end{equation}
denotes the natural projection map. Let $\nu$ denote the measure given by the right-hand side of \eqref{eq:lem:reduction to X}.

By the Banach--Alaoglu theorem, the sequence $\tnu_{Q,\widetilde{x}}$ is relatively compact in the vague topology. Hence, for every sequence $Q_j \to \infty$, there exists a subsequence along which $\tnu_{Q_j,\widetilde{x}}$ converges vaguely to a measure $\tnu$, which a priori need not be a probability measure. To prove convergence, it therefore suffices to show that every such limit measure coincides with
\begin{align}
\label{eq:abcde 1}
\int_{\X_{n+1}}\int_{[0,1]^{n+1}} 
\delta_{[g,g\cdot w]\widetilde{\Gamma}_{n+1}}\,dw\, d\nu(g\Gamma_{n+1}).
\end{align}
Accordingly, let $\tnu$ be a subsequential limit of $\tnu_{Q,\widetilde{x}}$ as $Q\to\infty$.

Using \eqref{eq:lem:reduction to X}, we have
\begin{equation}
\pi_*(\tnu)=\lim_{Q\to\infty}\pi_*(\tnu_{Q,\widetilde{x}})
=\lim_{Q\to\infty}\nu_{Q,x}=\nu,
\end{equation}
and therefore on disintegrating $\tnu$ along $\pi$, we get:
\begin{align}
\label{eq: a 2}
\tnu = \int_{\X_{n+1}} \tnu^{(y)}\, d\nu(y),
\end{align}
where for $\nu$-almost every $y\in\X_{n+1}$,
\begin{equation}
\supp(\tnu^{(y)})\subset \pi^{-1}(y).
\end{equation}

By Lemma~\ref{lem:reduction 1 1}, the measure $\tnu$ is invariant under the right action of $\widetilde{U}_{n+1}$ defined in \eqref{eq:def U n}.  
Hence for $\nu$-almost every $y$, the conditional measure $\tnu^{(y)}$ is also $\widetilde{U}_{n+1}$-invariant.

Note that for each $y\in\X_{n+1}$, the fiber $\pi^{-1}(y)$ is naturally identified with $\T^{n+1}$, and $\widetilde{U}_{n+1}$ acts by translations on this torus. While the action may not be uniquely ergodic for a fixed $y$, perturbing $y$ by $H_{n+1}$ shows that for $m_{H_{n+1}}$-almost every $h\in H_{n+1}$, the action of $\widetilde{U}_{n+1}$ on $\pi^{-1}(hy)$ is uniquely ergodic.  
Since $\nu$ is $H_{n+1}$-invariant by its definition (equation \eqref{eq:lem:reduction to X}), it follows that for $\nu$-almost every $y$, the $\widetilde{U}_{n+1}$-action on $\pi^{-1}(y)$ is uniquely ergodic. Therefore, for $\nu$-almost every $y$, the unique $\widetilde{U}_{n+1}$-invariant probability measure on the fiber is the Lebesgue measure
\begin{align}
\label{eq: axa 1}
\int_{[0,1]^{n+1}} f([g,g\cdot w]\widetilde{\Gamma}_{n+1})\,dw,
\end{align}
where $y=g\Gamma_{n+1}$.

Substituting this description of $\tnu^{(y)}$ into \eqref{eq: a 2} gives \eqref{eq:abcde 1}, completing the proof.
\end{proof}

\medskip

\section{Joint equidistribution}
\label{sec:Simplification via joint}
The main aim of this section is to prove Lemma~\ref{lem: simplification via joint}. We will need the following results.

\begin{prop}
\label{prop: single to joint equi}
    Suppose $X$ and $Y$ are locally compact second countable metric spaces equipped with probability measures $\nu_X$ and $\nu_Y$ respectively. Let $\phi_l: X \rightarrow Y$ denote a family of measurable functions. Assume that for any probability measure $\widetilde{\nu}$ on $X$ which is absolutely continuous with respect to $\nu_X$, and any $f\in C_c(Y)$ we have 
    \begin{align}
    \label{eq: lem: single to joint equi 1}
        \lim_{l \rightarrow \infty} \int_X f(\phi_l(x)) \, d\widetilde{\nu}(x)= \int_Y f(y)\, d\nu_Y(y).
    \end{align}
    Then the following holds.  Let $F: X \times Y \rightarrow \R$ be bounded continuous, and $F_l: X \times Y \rightarrow \R $ a family of uniformly bounded, continuous functions indexed by $l >0$ such that $F_l  \rightarrow F$ as $l \rightarrow \infty$, uniformly on compacta. Then
    \begin{align}
    \label{eq: lem: single to joint equi 2}
        \lim_{l \rightarrow \infty} \int_X F_l(x,\phi_l(x)) \, d\nu_X(x)= \int_X \int_Y F(x,y)\, d\nu_X(x) d\nu_Y(y).
    \end{align}
\end{prop}
\begin{proof}
    The proof of the proposition is essentially the same as the proof of \cite[Thm.~5.3]{MarStro}. 
    
    Let $d_X(\cdot, \cdot)$ denote the metric on $X$. Let us first assume that the support of the functions $(F_l)_l$ and $F$ are contained in a fixed compact subset $K_{X \times Y}$ of $X \times Y$. In this case, the convergence $F_l \rightarrow F$ is uniform, and also the functions $(F_l)_l$ and $F$ are uniformly continuous. Let $\delta>0$ be given. Using uniform continuity and uniform convergence, there exist $\e>0$, $l_0>0$ such that 
    \begin{align}
        F(x_0,y)-\delta \leq F(x,y) \leq F(x_0,y) + \delta, \label{eq: xax 1} \\
        F(x_0,y)-\delta \leq F_l(x,y) \leq F(x_0,y) + \delta \label{eq: xax 2}
    \end{align}
    for all $l > l_0$, $(x_0,y) \in X \times Y$ and $x$ satisfying $d_X(x,x_0) < \e$. 
    
    Let $K$ be a compact subset of $X$ such that
    \begin{equation}
    K_{X \times Y} \subset K \times Y.
    \end{equation}
    Using compactness of $K$, fix a partition $K_1, \ldots, K_p$ of $K$, each of positive $\nu_X$-measure, with diameter less than $\epsilon$. Also, fix $x_i \in K_i$ for each $i$. Then for $l \geq l_0$, using \eqref{eq: xax 2}, we have 
    \begin{align}
        \int_X F_l(x,\phi_l(x)) \, d\nu_X(x) &= \sum_{i=1}^p \int_{X} F_l(x,\phi_l(x)) \, \chi_{K_i}(x) d\nu_X(x) \nonumber \\
        &\leq \sum_{i=1}^p \int_{X} \left( F(x_i,\phi_l(x))+ \delta \right) \, \chi_{K_i}(x) d\nu_X(x) . \label{eq: xax 3}
    \end{align}
    Using \eqref{eq: lem: single to joint equi 1} for measures $(\nu_X(K_i))^{-1}\chi_{K_i}(x) d\nu_X(x)$ along with equation \eqref{eq: xax 1}, we see that
    \begin{align}
        \lim_{l \rightarrow \infty} \int_{X} F(x_i,\phi_l(x)) \, \chi_{K_i}(x) \, d\nu_X(x) &= \nu_X(K_i) \int_X F(x_i, y) \, d\nu_Y(y) \nonumber \\
        &\leq \int_{X}\int_Y \left(F(x,y)+ \delta \right) \chi_{K_i}(x) \, d\nu_X(x)d\nu_Y(y) . \label{eq: xax 4}
    \end{align}
    Combining \eqref{eq: xax 3} and \eqref{eq: xax 4}, we have
    \begin{align}
        \limsup_{l \rightarrow \infty}  \int_X F_l(x,\phi_l(x)) \, d\nu_X(x) \leq  \int_{X}\int_Y F(x,y) \, d\nu_X(x)d\nu_Y(y) + 2\delta.
    \end{align}
    An analogous argument shows that
    \begin{align}
        \liminf_{l \rightarrow \infty}  \int_X F_l(x,\phi_l(x)) \, d\nu_X(x) \geq  \int_{X}\int_Y F(x,y) \, d\nu_X(x)d\nu_Y(y) - 2\delta.
    \end{align}
    Since $\delta>0$ is arbitrary, we see that the limit exists and we have
    \begin{align}
        \lim_{l \rightarrow \infty}  \int_X F_l(x,\phi_l(x)) \, d\nu_X(x) =  \int_{X}\int_Y F(x,y) \, d\nu_X(x)d\nu_Y(y) .
    \end{align}

    We now extend the result to the bounded continuous function $F$, bounded by $|F| \leq M$. Given $\delta>0$ we choose a compact set $K_1 \subset X$ and $K_2 \subset Y$, so large that 
    \begin{equation}
    (1- \nu_X(K_1))+ (1-{\nu}_Y(K_2)) \leq \frac{\delta}{M}.
    \end{equation}
    Let $c_1: X \rightarrow [0,1]$ and $c_2 : Y \rightarrow [0,1]$ be continuous functions which have compact support and satisfy $\chi_{K_1} \leq c_1$ and $\chi_{K_2} \leq c_2$ respectively. Write 
    \begin{equation}
    F= F_1 + F_2 \quad \text{with} \quad F_1(x,y)= F(x,y) c_1(x) c_2(y), \quad F_2= F-F_1.
    \end{equation}
    Then, since $F_1$ is compactly supported, using the proved case, we see that
    \begin{align}
        &\lim_{l \rightarrow \infty} \int_X F_1(x, \phi_l(x)) \, d\nu_X(x) =  \int_{X}\int_Y F_1(x,y) \, d\nu_X(x)d\nu_Y(y) \nonumber \\
        &= \int_{X}\int_Y F(x,y) \, d\nu_X(x)d\nu_Y(y) +  \int_X \int_Y F(x,y)(c_1(x) c_2(y)-1) \, d\nu_X(x) d\nu_Y(y), \label{eq: xax 6}
    \end{align}
    and we have
    \begin{align}
         \left|\int_X \int_Y F(x,y)(c_1(x) c_2(y)-1) \, d\nu_X(x) d\nu_Y(y) \right| \leq M(1- \nu_X(K_1)\nu_Y(K_2)) \leq {\delta} \label{eq: xax 7}
    \end{align}
    For $F_2$, we use \eqref{eq: lem: single to joint equi 1} to see that
    \begin{align}
        &\limsup_{l \rightarrow \infty} \int_{X} |F_2(x, \phi_l(x))| \, d\nu_X(x) \nonumber \\
        &\leq M \int_{X} (1-c_1(x))\, d\nu_X(x) +  \limsup_{l \rightarrow \infty} \int_{K_2} M(1- c_2(\phi_l(x))) \, d\nu_X(x) \nonumber \\
        &\leq  M(1- {\nu}_X(K_1)) +  \int_{Y} M(1- c_2(y)) \, d\nu_X(x) \nonumber \\
        &\leq M(1- {\nu}_X(K_1)) + M(1- \nu_Y(K_2)) \leq \delta. \label{eq: xax 5}
    \end{align}
    Since $\delta>0$ is arbitrary, the lemma in the current case follows from \eqref{eq: xax 6}, \eqref{eq: xax 7} and \eqref{eq: xax 5}. This completes the proof.
\end{proof}

\begin{proof}[Proof of Proposition~\ref{lem:  simplification via joint}]
Fix $x \in \X_{n+1}$ and a sequence $(\e,\Delta)$ satisfying
$\e \Delta^{1/n} \to \sigma < \infty$. Let $\lambda$ be a measure on
$\R^n$ absolutely continuous with respect to Lebesgue measure, and let
$\widetilde{\lambda}$ be the measure on $\R^n$ obtained via pushforward of $\lambda$ under the map
\begin{equation}
\xi= (\xi', \xi_n)\mapsto (\xi'/\xi_n, \xi_n).
\end{equation}
It is clear that $\widetilde{\lambda}$ is absolutely continuous 
with respect to Lebesgue measure. 

Note that for any probability measure $\lambda_1$ on $\R^n$ which is absolutely
continuous with respect to
Lebesgue measure, we have
\begin{align}
\int_{\R^n}
f\!\left(
\begin{pmatrix}
\id_{n-1} &0&0\\ 0& 1&0 \\ 0& -\e^{-1}\Delta^{(n-1)/n}\beta^{-1} & 1
\end{pmatrix}
\begin{pmatrix}
\Delta^{1/n} \id_{n-1}&0&0 \\0 & \Delta^{-(n-1)/n}&0 \\ 0&0 & 1
\end{pmatrix}
\begin{pmatrix}
n(-\zeta)&0 \\0 & 1
\end{pmatrix}
x
\right)
\, d\lambda_1(\zeta,\beta)
\end{align}
converges to $\mu_{H_{n+1} \cdot x}(f)$ for every $f\in C_c(\X_{n+1})$. To see this, note that by the assumptions of the lemma, the claim holds for any
probability measure of the form
\begin{align}
\label{eq:q1}
\lambda_1=\lambda_0\otimes\left(\frac{1}{c-b}m_{\R}|_{[b,c]}\right),
\end{align}
where $\lambda_0$ is absolutely continuous with respect to Lebesgue measure
and $0<b<c<\infty$. By linearity, the same also holds for any probability
measure that can be written as a finite linear combination of measures of
the form~\eqref{eq:q1}. The general case then follows by a standard density
argument, since such measures are dense (in the weak-$*$ sense) in the space
of probability measures on $\R^n$ that are absolutely
continuous with respect to Lebesgue measure.

Since $\widetilde{\lambda}$ is also absolutely continuous with respect to
Lebesgue measure, the same convergence holds for
every probability measure $\lambda_1$ which is absolutely continuous with respect to
$\widetilde{\lambda}$.

To prove the lemma, fix $f\in C_c(\X_{n+1})$ and apply
Proposition~\ref{prop: single to joint equi} with
$X = \R^n$, $\nu_X = \widetilde{\lambda}$, $Y = \X_{n+1}$, $\nu_Y = \mu_{H_{n+1}\cdot x},$
\begin{align}
\phi_Q(\zeta,\beta) =
\begin{pmatrix}
\id_{n-1}&0&0 \\ 0& 1&0 \\0 & -\e^{-1}\Delta^{(n-1)/n}\beta^{-1} & 1
\end{pmatrix}
\begin{pmatrix}
\Delta^{1/n} \id_{n-1}&0&0 \\ 0& \Delta^{-(n-1)/n}&0 \\ 0&0 & 1
\end{pmatrix}
\begin{pmatrix}
n(-\zeta)&0 \\0 & 1
\end{pmatrix}
x, 
\end{align}
\begin{align}
F(\zeta,\beta,y) =
f\!\left(
\begin{pmatrix}
n(\zeta)&0 \\0 & 1
\end{pmatrix}
\begin{pmatrix}
\id_{n-1}&0&0 \\0 &0 & -\sigma\beta \\ 0& (\sigma\beta)^{-1} &0
\end{pmatrix}
y
\right), 
\end{align}
and
\begin{align}
F_Q(\zeta,\beta,y) =
F\!\left(
\zeta,\,
\sigma^{-1}\e\Delta^{1/n}\beta,\,
\begin{pmatrix}
\id_{n-1}&0&0 \\ 0& 1 & \e\Delta^{-(n-1)/n}\beta \\ 0&0 & 1
\end{pmatrix}
y
\right),
\end{align}
along with \eqref{eq: lem reduction 1} to see that $\nu_{Q,x}(f)$ converges to
\begin{align}
    \int_{\R^{n}} \int_{\X} f \left( \begin{pmatrix}
        \id_{n-1}& 0 & -\sigma \zeta \beta \\ 0&0& -\sigma \beta \\ 0&(\sigma \beta)^{-1} &0 
    \end{pmatrix}  y \right)\, d\mu_{H_{n+1}\cdot x}(y) d\widetilde{\lambda}(\zeta,\beta).
\end{align}
This proves the lemma.
\end{proof}

\section{Proof of the Case $\sigma<\infty$}
\label{sec:Finite case}
In this section, we prove Proposition~\ref{lem:  final proof} and conclude the proof of Theorem~\ref{thm: main dynamical theorem} for $\sigma < \infty$.

\begin{proof}[Proof of Proposition~\ref{lem:  final proof} for $\alpha \notin \Q^n$.]
Fix a sequence $(\e,\Delta)$ satisfying $\e \Delta^{1/n} \to \sigma < \infty$. Let $x= u(-\alpha)\Gamma_{n+1}$. 
Let $\widetilde{G}_n= \SL_n(\R) \ltimes \R^n$, considered as subgroup of $G_{n+1}$ via the map 
    \begin{equation}
    [A,b] \mapsto \begin{pmatrix}
        A & b \\ 0& 1
    \end{pmatrix}.
    \end{equation}
    Note that the preimage of $\Gamma_{n+1}$ under the above map equals $\widetilde{\Gamma}_n= \SL_n(\Z) \ltimes \Z^n$, and hence the map induces an injective map from the homogeneous space $\widetilde{\X}_n= \widetilde{G}_n/\widetilde{\Gamma}_n$ into $\X_{n+1}$. Using this injection, we will consider $\widetilde{\X}_n$ as subset of $\X_{n+1}$.

   With this identification, it is a well-known result that for any measure $\lambda_0$ on $\R^{n-1}$ absolutely continuous with respect to Haar measure, the following holds:
    \begin{align}
    \label{eq: abcd 1}
       \int_{\R^n} f\left(\begin{pmatrix}
        \Delta^{1/n} \id_{n-1} &0&0 \\0 & \Delta^{-(n-1)/n}&0 \\ 0&0 & 1
    \end{pmatrix} \begin{pmatrix}
        n(-\zeta)&0  \\0 & 1
    \end{pmatrix} x\right) \, d\lambda_0(\zeta) \rightarrow \mu_{\widetilde{\X}_n} (f),
    \end{align}
    where $\mu_{\widetilde{\X}_n}$ equals the unique $\widetilde{G}_n$-invariant probability measure on $\widetilde{\X}_n$.  This follows, for example, from 
\cite[Thm.~1.4]{Shah96}; see also \cite[Thm.~5.2]{MarStro}.

    Now let us consider the group
\begin{equation}
U=\left\{
\begin{pmatrix}
\id_{n-1}&0&0\\
0&1&0\\
0&\beta&1
\end{pmatrix}
:\beta\in\R
\right\}
< G_{n+1},
\end{equation}
and note that
\begin{align}
\label{eq:q2}
    \mu_{\widetilde{\X}_n}(\cS(U))=0,
\end{align}
where $\cS(U)$ is defined as in Section~\ref{sec: Gen Dan Mar}.
To see this, let $m_{G_{n+1}}$ and $m_{\Tilde{G}_n}$ denote Haar measures on
$G_{n+1}$ and $\Tilde{G}_n$ respectively. Consider the map
\begin{equation}
\R^n \times \R_{>0} \times \Tilde{G}_n
\to
G_{n+1},
\quad
(w,t,h)
\mapsto
\begin{pmatrix}
\id_n &0\\ {}^{\top}\! w &1
\end{pmatrix}
a_t h .
\end{equation}

The image of this map is an open subset of $G_{n+1}$ and the pullback of
$m_{G_{n+1}}$ under this map is absolutely continuous with respect to
$m_{\R^n}\otimes m_{\R}|_{\R_{>0}}\otimes m_{\Tilde{G}_n}$.

By ergodicity of the $U$-action with respect to $\mu_{\X_{n+1}}$, we know that
for any $x\in\X_{n+1}$ and for $m_{G_{n+1}}$-almost every $g$, we have
$gx\notin \cS(U)$. Hence for $m_{\R^n}$-almost every $w\in\R^n$,
$m_{\R}$-almost every $t\in(0,\infty)$, and $m_{\Tilde{G}_n}$-almost every
$h\in\Tilde{G}_n$, we have
\begin{equation}
\begin{pmatrix}
\id_n &0\\ {}^{\top}\! w &1
\end{pmatrix}
a_t h x
\notin \cS(U),
\end{equation}
that is,
\begin{equation}
\overline{U\cdot
\begin{pmatrix}
\id_n &0\\ {}^{\top}\! w &1
\end{pmatrix}
a_t h x}
=
\X_{n+1}.
\end{equation}

Since
\(
\begin{pmatrix}
\id_n &0\\ {}^{\top}\! w &1
\end{pmatrix}
\)
commutes with $U$ and $a_t$ normalises $U$, we obtain
\begin{equation}
\overline{U\cdot
\begin{pmatrix}
\id_n &0\\ {}^{\top}\! w &1
\end{pmatrix}
a_t h x}
=
\begin{pmatrix}
\id_n &0\\ {}^{\top}\! w &1
\end{pmatrix}
a_t
\ \cdot \overline{U\cdot h x}.
\end{equation}

Since the left-hand side equals $\X_{n+1}$, it follows that
\begin{equation}
\overline{U\cdot h x}=\X_{n+1}.
\end{equation}

Thus for $m_{\Tilde{G}_n}$-almost every $h\in\Tilde{G}_n$ we have
$hx\notin \cS(U)$. The claim now follows from the invariance of
$\mu_{\Tilde{\X}_n}$ under the action of $\Tilde{G}_n$.

 Thus, using Theorem~\ref{thm: gen Dani Margulis} together with
\eqref{eq: abcd 1} and~\eqref{eq:q2}, we obtain that for any $0<b<c<\infty$ and every
$f\in C_c(\X_{n+1})$,
\begin{align}
\frac{1}{c-b}\int_b^c \int_{\R^{n-1}}
f\!\left(
\begin{pmatrix}
\id_{n-1} &0&0 \\ 0& 1&0 \\ 0& \e^{-1}\Delta^{(n-1)/n}\beta & 1
\end{pmatrix}
\begin{pmatrix}
\Delta^{1/n} \id_{n-1}&0&0 \\ 0& \Delta^{-(n-1)/n} &0 \\ 0&0 & 1
\end{pmatrix}
\begin{pmatrix}
n(-\zeta) &0 \\0 & 1
\end{pmatrix}
y_\alpha
\right)
\, d\lambda_0(\zeta)\, d\beta
\end{align}
converges to $\mu_{\X_{n+1}}(f)$ as $Q\to\infty$.

In particular, this implies
\begin{equation}
\X_{n+1}
= \supp(\mu_{\X_{n+1}})
\subset
\overline{U\cdot \overline{G_n\cdot x}}
\subset
\overline{H_{n+1}\cdot x}
\subset
\X_{n+1}.
\end{equation}
Hence all inclusions are equalities, proving~\eqref{eq:lem: final proof 2}.
\end{proof}

\begin{proof}[Proof of Proposition~\ref{lem:  final proof} for $\alpha \in \Q^n$.]
Fix $\alpha \in \R^n$ and a sequence $(\e,\Delta)$ such that
$\e \Delta^{1/n} \to \sigma < \infty$. Let $\lambda_0$ be a measure on
$\R^{n-1}$ that is absolutely continuous with respect to Lebesgue measure, and
fix $0<b<c<\infty$. Set $x=u(-\alpha)\Gamma_{n+1}$.

First note that the orbit $H_{n+1}\cdot x$ is closed. Indeed, there is a natural
surjection from $\widetilde{\X}_{n}(N)$ onto $H_{n+1}\cdot x$, where
\begin{equation}
\widetilde{\X}_{n}(N)
=
\SL_n(\R) \ltimes \R^n \Big/
\{g\in \SL_n(\Z)\ltimes \Z^n : g \equiv [\id,\mathbf 0] \!\!\!\pmod N\},
\end{equation}
and $N=M^2$, with $M$ the smallest positive integer satisfying
$M\alpha\in\Z^n$.
The map is given by
\begin{equation}
[A,b]\widetilde{\Gamma}_n(N)
\;\mapsto\;
\begin{pmatrix}
A^{-1} &0 \\
{}^{\top}\!(-Ab) & 1
\end{pmatrix} x .
\end{equation}
Therefore, it suffices to show that the sequence of measures
$\gamma_Q$ defined by
\begin{align}
\gamma_Q(f)
=
\frac{1}{c-b}
\int_b^c
\int_{\R^{n-1}}
f\!\left(
\left[
\begin{pmatrix}
\Delta^{1/n} \id_{n-1} &0 \\0
& \Delta^{-(n-1)/n}
\end{pmatrix}
\begin{pmatrix}
n(-\zeta) &0 \\0
& 1
\end{pmatrix},
\begin{pmatrix}
\zer \\
\e^{-1}\Delta^{(n-1)/n}\beta
\end{pmatrix}
\right]
\widetilde{\Gamma}_{n+1}(N)
\right)
\, d\lambda_0(\zeta)\, d\beta
\end{align}
converges to $\mu_{\widetilde{\X}_{n+1}(N)}$, the unique
$\SL_n(\R)\ltimes \R^n$-invariant probability measure on
$\widetilde{\X}_{n+1}(N)$.

To this end, define
\begin{equation}
\X_n(N)
=
\SL_n(\R)\big/
\{\gamma\in\SL_n(\Z):\gamma\equiv\id_n \pmod N\},
\end{equation}
and let $\mu_{\X_n(N)}$ denote the unique $\SL_n(\R)$-invariant
probability measure on $\X_n(N)$. We denote by $\pi$ the natural projection map from $\widetilde{\X}_{n+1}(N)$ to $\X_n(N)$.

\smallskip
By the Banach--Alaoglu theorem, the sequence $(\gamma_Q)$ is relatively compact in the vague topology. Hence, every sequence $Q_j \to \infty$ admits a subsequence along which $\gamma_{Q_j}$ converges vaguely to a measure $\gamma_\infty$, which a priori need not be a probability measure. To prove convergence, it therefore suffices to show that every such subsequential limit coincides with $\mu_{\widetilde{\X}_{n+1}(N)}$. Let $\gamma_\infty$ be a subsequential limit of $(\gamma_Q)$. Using the same argument as in Step~2 of the proof of
Theorem~\ref{thm: gen Dani Margulis}, one shows that $\gamma_\infty$
is invariant under the subgroup
\begin{align}
\label{eq: cc 2}
\left\{
\begin{pmatrix}
\id_{n-1} &0 &0 \\
0& 1 &0 \\
0& \beta & 1
\end{pmatrix} : \beta\in\R
\right\}.
\end{align}

Disintegrating $\gamma_\infty$ along the fibers of $\pi$ gives
\begin{align}
\label{eq: cc 3}
\gamma_\infty
=
\int_{\X_n(N)} \gamma_\infty^{(y)}\, d\pi_*(\gamma_\infty)(y),
\end{align}
where for $\pi_*(\gamma_\infty)$-almost every $y$, the conditional measure
$\gamma_\infty^{(y)}$ is supported on the fiber $\pi^{-1}(y)$ and is
invariant under the group \eqref{eq: cc 2}.

Note that by Margulis’s thickening technique (see, e.g., \cites{KM1,Margulis}),
we know that $\pi_*(\gamma_Q)$ converges to $\mu_{\X_n(N)}$. Hence, we have
\begin{equation}\pi_*(\gamma_\infty)=\mu_{\X_n(N)}.\end{equation}

Also, note that the fiber $\pi^{-1}(y)$ can be
naturally identified with $(\R/N\Z)^n$ for $y \in \X_n(N)$. Under this identification,
the action of \eqref{eq: cc 2} becomes a translation action on
$(\R/N\Z)^n$, which for $\mu_{\X_n(N)}$-almost every $y$, is uniquely ergodic. Hence, for $\mu_{\X_n(N)}$-almost every $y$, the measure
$\gamma_\infty^{(y)}$ coincides with the Haar probability
measure on $(\R/N\Z)^n$.

Substituting this into \eqref{eq: cc 3} yields $\gamma_\infty = \mu_{\widetilde{\X}_{n+1}(N)}$. 
Since this holds for every subsequential limit $\gamma_\infty$ of the sequence $\gamma_Q$, the proof is complete.
\end{proof}

\section{Proof of the Case $\sigma=\infty$}
\label{sec: sigma finite is enough}

The main aim of this section is to prove Proposition~\ref{lem:  enough for finite}. We will need the following results to prove the lemma. The first result is a corollary of Theorem~\ref{thm: gen Dani Margulis}.

\begin{cor}
    \label{cor: Gen Dan Mar}
Let $\lambda$ be the normalised restriction of Lebesgue measure to a set
$O$ of the form $[\alpha_1, \beta_1] \times \cdots \times [\alpha_k, \beta_k]$.
 
 Let $\e_i,\Delta_i $ be sequences in $(0, \infty)$ such that
  \begin{equation}
  \lim_{i \rightarrow \infty} \e_i \Delta_i^{1+ 1/n}= \infty,
  \end{equation}
  and for every continuous function $f$ on $\X_{n+1}$, we have
  \begin{align}
  \label{eq: 1 : cor: Gen Dan Mar}
       \lim_{i \rightarrow \infty} \int_{\R^n} f(  a(\Delta_i)   u(-\e_i \xi) x ) \, d\lambda(\xi) = m_{\X_{n+1}}(f).
    \end{align}
    Then for any sequence $(\e_i')_i$ satisfying $\e_i'\geq \e_i$ and any continuous function $f$ on $\X_{n+1}$, we have
    \begin{align}
    \label{eq: 2 : cor: Gen Dan Mar}
       \lim_{i \rightarrow \infty} \int_{\R^n} f(  a(\Delta_i)   u(-\e_i' \xi) x ) \, d\lambda(\xi) = m_{\X_{n+1}}(f).
    \end{align}
\end{cor}
\begin{proof} 
  Let $U= \{u(\theta): \theta \in \R^n\}$, and fix $x_0 \in \cG(U)$ (defined as in \eqref{eq:def g U}). Using \eqref{eq: 1 : cor: Gen Dan Mar}, we know that the set 
    \begin{equation}
    \{ a(\Delta_i)  u(-\e_i  \xi) x: \xi \in O \}
    \end{equation}
    equidistribute in $\X_{n+1}$. In particular, we get a sequence $\xi_i \in O$ such that
    \begin{align}
        \label{eq: 3 : cor: Gen Dan Mar}
        a(\Delta_i)  u(-\e_i  \xi_i) x \rightarrow x_0,
    \end{align}
    as $i \rightarrow \infty$. Let us define
    \begin{align}
        \mu_i = \delta_{ a(\Delta_i) u(-\e_i \xi_i)  x}, \qquad
        \mu = \delta_{x_0}, \qquad
        I_i = \{ -\e_i' \Delta_i^{1+1/n} \xi+ \e_i \Delta_i^{1+1/n} \xi_i: \xi \in O\}.
    \end{align}
    Then one gets that 
    \begin{align}
\nonumber
      \int_{\R^n} f(  a(\Delta_i)   u(-\e_i' \xi) x ) \, d\lambda(\xi) =  \frac{1}{m_{\R^k}(I_i)} \int_{I_i} \int_{G/\Gamma} f(u(t) x) \, d\mu_i(x) \, dt .
    \end{align}
    The corollary thus follows from Theorem~\ref{thm: gen Dani Margulis}, along with the fact that $x_0 \in \cG(U)$ and equation \eqref{eq: 3 : cor: Gen Dan Mar}. This completes the proof.
\end{proof}

The second result we need is the following proposition.
\begin{prop}
\label{prop: double}
 Let $\{\gamma_{\sigma, \Delta}: \sigma, \Delta \in (1, \infty)\}$ be a doubly-parametrised sequence of probability measures on a locally compact second countable metric space $X$, such that the following holds.
 \begin{enumerate}
 \renewcommand{\labelenumi}{\rm (\roman{enumi})}
     \item For every $\sigma>0$, there exists a probability measure $\gamma_{\sigma, \infty}$ on $X$ satisfying
 \begin{align}
     \label{eq: prop double 1}
      \lim_{\Delta \rightarrow \infty} \gamma_{\sigma,\Delta} = \gamma_{\sigma,\infty}.
 \end{align}
 \item There exists a probability measure $\gamma_{\infty, \infty}$ on $X$ such that 
 \begin{align}
     \label{eq: prop double 2}
     \lim_{\sigma \rightarrow \infty} \gamma_{\sigma , \infty} = \gamma_{\infty, \infty}.
 \end{align}
 \item For every sequences $\sigma_i, \sigma_i', \Delta_i$ in $(1, \infty)$, if $\Delta_i \rightarrow \infty$ as $i \rightarrow \infty$, $\sigma_i \leq \sigma_i'$ and
\begin{align}
    \label{eq: prop double 3}
    \lim_{i \rightarrow \infty} \gamma_{\sigma_i, \Delta_i} = \gamma_{\infty, \infty},
\end{align}
then, we have
\begin{align}
    \label{eq: prop double 4}
    \lim_{i \rightarrow \infty} \gamma_{\sigma_i', \Delta_i}= \gamma_{\infty, \infty}
\end{align}
 \end{enumerate}

 Then for every sequence $(\sigma_i, \Delta_i) \in (1, \infty)^2$ such that both $\sigma_i$ and $\Delta_i$ diverges to $\infty$ as $i \rightarrow \infty$, we have
 \begin{align}
     \label{eq: prop double 5}
     \lim_{i \rightarrow \infty} \gamma_{\sigma_i, \Delta_i} = \gamma_{\infty, \infty}.
 \end{align}
\end{prop} 
\begin{proof}
    Let $\cP(X)$ denote the space of all probability measures on $X$. Let $d(\cdot, \cdot)$ denote the Prokhorov metric on $\cP(X)$, given by
    \begin{equation}
    d(\mu, \nu) = \inf\{\e>0: \mu(A) < \nu(A^\e)+\e \text{ and } \nu(A)< \mu(A^\e)+\e \text{ for all } A \in \cB(X)\},
    \end{equation}
    where $\cB(X)$ denotes the Borel sigma algebra of $X$ and for $A \in \cB(X)$, we define
    \begin{equation}
    A^\e= \{y \in X: \text{ there exists $z \in A$ such that } d_X(y, z)< \e\},
    \end{equation}
    where further $d_X$ denotes the metric on $X$.

    Let $\psi: (1, \infty) \rightarrow [1,\infty)$ be defined as
    \begin{equation}
    \psi(\sigma)= \inf \left\{\Delta_0: \text{ for all } \Delta \geq \Delta_0, \text{ we have } d(\gamma_{\sigma, \Delta}, \gamma_{\sigma, \infty}) \leq \frac{1}{2^\sigma}\right\},
    \end{equation}
    which is finite by condition (i).

    Fix a sequence $\sigma_i, \Delta_i$ in $(1, \infty)$ both diverging to infinity. Without loss of generality, we may assume that both $\sigma_i,\Delta_i$ are non-decreasing sequences.

   For each $i$ with $\Delta_i> \psi(\sigma_1)$, let
\begin{equation}
\eta_i = \max\{\sigma_j: 1 \leq j \leq i,\ \psi(\sigma_j)\leq \Delta_i\}.
\end{equation}
Note that $\eta_i$ is defined for all sufficiently large $i$. Moreover, by definition,
$(\eta_i)_i$ is non-decreasing and satisfies
$\eta_i \leq \sigma_i$ and $\psi(\eta_i)\leq \Delta_i$.

Furthermore, since $\sigma_i\to\infty$ and $\Delta_i\to\infty$, we claim that $\eta_i\to\infty$. Indeed, fix $M>0$. Choose $j$ such that $\sigma_j\ge M$. Since $\Delta_i\to\infty$, we have $\Delta_i \ge \psi(\sigma_j)$ for all sufficiently large $i$. Hence for all such $i$, the index $j$ is admissible in the definition of $\eta_i$, so $\eta_i \ge \sigma_j \ge M$. This proves $\eta_i\to\infty$.

   By condition~(iii), it is therefore enough to prove that   
   $$
   \gamma_{\eta_i, \Delta_i} \rightarrow \gamma_{\infty, \infty}.
   $$
   To prove this, note that
    \begin{align}
    \label{eq: prop double proof 1}
        d(\gamma_{\eta_i, \Delta_i}, \gamma_{\infty, \infty}) &\leq d(\gamma_{\eta_i, \Delta_i}, \gamma_{\eta_i, \infty}) + d(\gamma_{\eta_i, \infty}, \gamma_{\infty, \infty})
    \end{align}
    Since $\eta_i \rightarrow \infty$ as $i \rightarrow \infty$, we have using condition (ii) that
    \begin{align}
    \label{eq: prop double proof 2}
        \lim_{i \rightarrow \infty} d(\gamma_{\eta_i, \infty}, \gamma_{\infty, \infty}) = 0.
    \end{align}
    Also, since $\psi(\eta_i)\leq \Delta_i $, we have
    \begin{align}
        \label{eq: prop double proof 3}
        d(\gamma_{\eta_i, \Delta_i}, \gamma_{\eta_i, \infty}) \leq \frac{1}{2^{\eta_i}} \rightarrow 0.
    \end{align}
    Combining equations \eqref{eq: prop double proof 1}, \eqref{eq: prop double proof 2} and \eqref{eq: prop double proof 3}, we see that $\gamma_{\eta_i, \Delta_i}$ converges to $\gamma_{\infty, \infty}$ as $i \rightarrow \infty$. Hence the proposition follows.
\end{proof}

We now prove Proposition~\ref{lem:  enough for finite}.
\begin{proof}[Proof of Proposition~\ref{lem:  enough for finite}]
Fix $x \in \X_{n+1}$ and a sequence $(\e_i,\Delta_i)$ such that $\e_i \Delta_i \rightarrow \infty$. We first prove the lemma when $\lambda$ equals the normalised Lebesgue measure restricted on some cuboid $O$ of the form $[\alpha_1, \beta_1] \times \cdots \times [\alpha_n, \beta_n]$ , that is, when
\begin{equation}
\lambda= \frac{1}{m_{\R^n}(O)} m_{\R^n}|_{O}.
\end{equation} 

To prove this, we will use Proposition~\ref{prop: double} for $X= \X_{n+1}$, $\sigma_i= \e_i \Delta_i^{1/n}$ and $\Delta_i= \Delta_i$ with 
 \begin{align}
     \gamma_{\sigma,\Delta}(f) &=  \int_{\R^n} f\left( a(\Delta) u(-\sigma \Delta^{-1/n} \xi) x \right)\, d\lambda(\xi), 
\end{align}
\begin{align}
     \gamma_{\sigma, \infty}(f) &=  \int_{\R^n} \int_{\X_{n+1}} f\left(\begin{pmatrix}
     \id_{n-1}  &0 & -\sigma \xi' \\ 0&0& -\sigma \xi_n \\ 0&(\sigma \xi_n)^{-1} & 0 \end{pmatrix}  y\right)\, d\mu_{H_{n+1} \cdot x}(y) \ d\hat{\lambda}(g) ,
\end{align}
and $\gamma_{\infty, \infty} = \mu_{\X_{n+1}}$.
Condition (i) of the proposition follows from the assumption that
\eqref{eq:lem:reduction to X} holds for $\sigma<\infty$, and condition (iii)
follows from Corollary~\ref{cor: Gen Dan Mar}. For condition (ii), note that
by the $H_{n+1}$-invariance of $\mu_{H_{n+1}\cdot x}$ we can write
$\gamma_{\sigma,\infty}$ as
\begin{equation}
\gamma_{\sigma,\infty}(f)=
\int_{\R^n}
\left(
\int_{\X_{n+1}}
f\!\left(
\begin{pmatrix}
\id_{n-1} &0 & -\xi' \\
0&0&-\xi_n \\
0&(\xi_n)^{-1}&0
\end{pmatrix}
\begin{pmatrix}
\sigma^{-1/n}\id_n &0\\
0&\sigma
\end{pmatrix}
y
\right)
\,d\mu_{H_{n+1}\cdot x}(y)
\right)
d\lambda(\xi',\xi_n).
\end{equation}

Using Margulis's thickening technique~\cites{KM1,Margulis} together with the
$G_{n+1}$-invariance of $\mu_{\X_{n+1}}$, it follows that the inner integral in
the above expression converges to $\mu_{\X_{n+1}}(f)$ for each $(\xi',\xi_n)$.
Consequently, $\gamma_{\sigma,\infty}$ converges to
$\mu_{\X_{n+1}}=\gamma_{\infty,\infty}$. This proves condition~(ii).

Therefore, the lemma follows in this special case by applying Proposition~\ref{prop: double}. The general case follows from the special case by
approximating $\lambda$ by finite linear combinations of indicator
functions of cuboids.
\end{proof}

\begin{proof}[Proof of Theorem~\ref{thm: main dynamical theorem}]
    The conclusions \eqref{infact} and \eqref{infact 2} of Theorem~\ref{thm: main dynamical theorem} follow directly from Propositions~\ref{lem: reduction to X}, \ref{lem: simplification via joint}, \ref{lem: final proof}, and \ref{lem: enough for finite}. The equation \eqref{eq; convergence of rational case} follows from the proof of Proposition~\ref{lem: enough for finite}.  
    
    For~\eqref{eq:rational case}, let $\alpha= p/q \in \Q^n$ with $\gcd(p,q)=1$.    In view of Proposition~\ref{lem: reduction to X} and the independence of the choice of $M_\xi$ in the expression~\eqref{eq: thm : main dynamical 2}, it is sufficient to show that
    \begin{align}
        \label{eq to show}
        H_{n+1} \cdot u(-\alpha) {\Gamma}_{n+1}
        =
        \begin{pmatrix}
           \id_{n-1} & 0 &0 \\
           0 & q^{-1}& 0 \\
           0 &0 & q
       \end{pmatrix}
       H_{n+1} \cdot {\Gamma}_{n+1}.
    \end{align}
To prove~\eqref{eq to show}, let $p_0= \gcd(p)$, and fix $\gamma_0 \in \SL_n(\Z)$ such that
    \begin{equation}
    -\gamma_0 \cdot \alpha = \frac{p_0}{q}\,\bfe_n.
    \end{equation}
    Then
    \begin{align}
        H_{n+1} \cdot u(-\alpha) {\Gamma}_{n+1}
        &=
        H_{n+1} 
        \begin{pmatrix}
            \gamma_0 &0 \\
            0& 1
        \end{pmatrix} \cdot
        u(-\alpha)
        \begin{pmatrix}
            \gamma_0^{-1} &0 \\
            0& 1
        \end{pmatrix}
        {\Gamma}_{n+1}
        \nonumber\\
        &=
        H_{n+1} \cdot u(-p_0/q \bfe_n) {\Gamma}_{n+1}.
        \label{eq to show 1}
    \end{align}
    
    Fix $\gamma_1 \in \SL_2(\Z)$ such that $(q,p_0) \gamma_1 = (1,0)$.
    Then
    \begin{align*}
         \begin{pmatrix}
            1 & p_0/q \\
            0& 1 
        \end{pmatrix}
        \gamma_1 
        =
        \begin{pmatrix}
            q^{-1} & 0 \\
            \star & q
        \end{pmatrix},
    \end{align*}
    for some $\star \in \R$. Thus, 
    \begin{align}
       & H_{n+1} \cdot u(-p_0/q \bfe_n) {\Gamma}_{n+1}
       =
       H_{n+1} 
       \begin{pmatrix}
           \id_{n-1} & 0 &0 \\
           0 & 1& p_0/q \\
           0 &0 & 1
       \end{pmatrix}
       \begin{pmatrix}
           \id_{n-1} & 0  \\
           0 & \gamma_1
       \end{pmatrix}
       {\Gamma}_{n+1}
       \nonumber \\
       & =
       H_{n+1} \cdot
       \begin{pmatrix}
           \id_{n-1} & 0 &0 \\
           0 & q^{-1}& 0 \\
           0 & \star & q
       \end{pmatrix}
       {\Gamma}_{n+1} 
       =
       H_{n+1} \cdot
       \begin{pmatrix}
           \id_{n-1} & 0 &0 \\
           0 & q^{-1}& 0 \\
           0 &0 & q
       \end{pmatrix}
       {\Gamma}_{n+1}.
       \label{eq to show 2}
    \end{align}
    Now, since
    \begin{align}
    \begin{pmatrix}
           \id_{n-1} & 0 &0 \\
           0 & q^{-1}& 0 \\
           0 &0 & q
    \end{pmatrix}^{-1}
    H_{n+1}
    \begin{pmatrix}
           \id_{n-1} & 0 &0 \\
           0 & q^{-1}& 0 \\
           0 &0 & q
    \end{pmatrix}
    =
    H_{n+1},
    \end{align}
    equation~\eqref{eq to show} follows from~\eqref{eq to show 1} and~\eqref{eq to show 2}. This proves~\eqref{eq:rational case}. Hence the theorem follows. 
\end{proof}

\section{Background on convergence of measures}\label{AppB}

The following lemmas will be key in translating the equidistribution results to the convergence of fine-scale statistics. This section follows the approach developed in \cite[\S~5]{MarStro}.

For a bounded subset $B \subset \R^{n+1}$ and a non-negative integer $r$, define
\begin{align}
\label{eq: def B}
\cE(B,r)
:= \left\{ [A,b]\widetilde{\Gamma} \in \widetilde{X}_{n+1} :
\big|B \cap (A\Z^{n+1}+b)\big| \ge r \right\}.
\end{align}
Let $\{\cE_t\}_{t\ge t_0}$ be a family of sets, where $t_0$ is a fixed real constant.
Define
\begin{align}
\liminf_{t\to\infty} \cE_t
:= \bigcup_{t\ge t_0}\ \bigcap_{s\ge t}\cE_s, \qquad \limsup_{t\to\infty} \cE_t
:= \bigcap_{t\ge t_0}\ \bigcup_{s\ge t}\cE_s .
\end{align}
We will also use the notation
\begin{align}
\lim (\inf \cE_t)^\circ
:= \bigcup_{t\ge t_0}\left(\bigcap_{s\ge t}\cE_s\right)^\circ, \qquad
\lim \overline{\sup \cE_t}
:= \bigcap_{t\ge t_0}\overline{\bigcup_{s\ge t}\cE_s}.
\end{align}
Note that $\lim (\inf \cE_t)^\circ$ is open and
$\lim \overline{\sup \cE_t}$ is closed.

If $\{\cE_t\}_{t\ge t_0}$ is a decreasing family and
$\cE = \bigcap_{t\ge t_0}\cE_t$, we write $\cE_t \downarrow \cE$.
If $\{\cE_t\}_{t\ge t_0}$ is an increasing family and
$\cE = \bigcup_{t\ge t_0}\cE_t$, we write $\cE_t \uparrow \cE$.

\begin{lem}
\label{lem: characteristic}
    Suppose $X$ is a locally compact second-countable metric space. Suppose $(\nu_i)_{i}$ is a sequence of probability measures on $X$ converging to the probability measure $\nu_\infty$. Let $\cE_l$ be a family of subsets of $X$, then 
    \begin{align}
        \label{eq: lem characteristic 2}
        \liminf_{l \rightarrow \infty} \nu_l(\cE_l) &\geq \nu_\infty(\lim (\inf \cE_t)^\circ), \\ \label{eq: lem characteristic 3}
        \limsup_{l \rightarrow \infty} \nu_l(\cE_l)  &\leq  \nu_\infty(\lim \overline{\sup \cE_t}).
    \end{align}
    If, furthermore, the set $\lim \overline{\sup \cE_t} \setminus \lim (\inf \cE_t)^\circ$ has measure zero, then 
    \begin{align}
        \label{eq: lem characteristic 4} \lim_{l \rightarrow \infty}\nu_l(\cE_l) &= \nu_\infty(\lim \overline{\sup \cE_t}).
    \end{align}
\end{lem}
\begin{proof}
    The proof is the same as the proof of \cite[Thm.~5.6]{MarStro}. We begin with the proof of equation \eqref{eq: lem characteristic 3}. Define the closed set
    \begin{equation}
    \hat{\cE}_l:= \overline{\bigcup_{s \geq l} \cE_s}.
    \end{equation}
    Clearly $\cE_l \subset \hat{\cE}_l \subset \hat{\cE}_{t}$ for $t \geq l$. So
    \begin{align}
        \limsup_{l \rightarrow \infty} \nu_l(\cE_l) \leq \limsup_{t \rightarrow \infty} \limsup_{l \rightarrow \infty} \int_X \nu_l(\hat{\cE}_t).
    \end{align}
  Since $\hat{\cE}_t$ is closed, it follows from the definition of weak convergence that
   \begin{align}
       \limsup_{l \rightarrow \infty} \nu_l(\hat{\cE}_t) \leq \nu_\infty(\hat{\cE}_t).
   \end{align}
    Since $\hat{\cE}_t \downarrow \lim \overline{\sup \cE_t}$,
    \begin{align}
         \limsup_{t \rightarrow \infty}  \nu_\infty(\hat{\cE}_t) =  \nu_\infty(\lim \overline{\sup \cE_t}),
    \end{align}
     and equation \eqref{eq: lem characteristic 3} follows. Relation \eqref{eq: lem characteristic 2} is established by taking complements, and \eqref{eq: lem characteristic 4} then follows from \eqref{eq: lem characteristic 2} and \eqref{eq: lem characteristic 3}.
\end{proof}

\begin{lem}
\label{MS10 Lem 6.2} 
    Fix $r \in \Z_{> 0}$. Then the following hold: 
    \begin{enumerate}
    \renewcommand{\labelenumi}{\rm (\roman{enumi})}
        \item If $U \subset B$, then $\cE(U, r) \subset \cE(B,r)$.
        \item If $B_t$ is a decreasing family of bounded sets, then $\bigcap_t \cE(B_t, r) = \cE(\bigcap_t B_t,r)$.
        \item If $B_t$ is an increasing family of bounded sets, then $\bigcup_t \cE(B_t, r) = \cE(\bigcup_t B_t,r)$.
        \item If $B$ is open, then $\cE(B,r)$ is open.
        \item If $B$ is closed and bounded, then $\cE(B,r)$ is closed.
    \end{enumerate}
       
\end{lem}
\begin{proof}
    The proof of the lemma is analogous to that of~\cite[Lem.~6.2]{MarStro}, hence is skipped.
\end{proof}

\begin{lem}
\label{lem: liminf calculations}
Let $k \in \N$. For each $1 \le i \le k$, let $(B_l^{(i)})_{l \in \N}$ be a family of bounded subsets of $\R^{n+1}$, and let $r_i \in \N$. For $l \in \N$, let $\cE_l = \bigcap_{i=1}^k \cE(B_l^{(i)}, r_i)$. Then
    \begin{align}
    \lim \left(\inf \cE_l \right)^\circ &\supset \bigcap_{i=1}^k \cE \left( \lim \left(\inf B_j^{(i)} \right)^\circ, r_i   \right), \label{eq: zz 1}\\
    \lim \overline{\sup \cE_l} &\subset    \bigcap_{i=1}^k \cE \left( \lim \overline{\sup B_j^{(i)}}, r_i \right). \label{eq: zz 2}
    \end{align}
\end{lem}
\begin{proof}
    For $l \in \N$, define 
    \begin{equation}
   \hat{\cE}_l := \bigcap_{i=1}^k\cE\left( \left( \bigcap_{j \geq l} B_j^{(i)} \right)^\circ, r_i \right).
   \end{equation}
   These sets are clearly an increasing family of open sets (cf. Lemma \ref{MS10 Lem 6.2}(iv)). Moreover, since they are contained in $\cE_l$, we see that
   \begin{align}
       \lim \left( {\inf \cE_l} \right)^\circ &\supset \lim \left( {\inf \hat{\cE}_l} \right)^\circ = \bigcup_{l\geq 1} \left( {\bigcap_{j \geq l} \hat{\cE}_j } \right)^\circ =  \bigcup_{l\geq 1}\hat{\cE}_l  =  \bigcup_{l \geq 1} \left(\bigcap_{i=1}^k\cE\left( \left( \bigcap_{j \geq l} B_j^{(i)} \right)^\circ, r_i \right)  \right).
    \end{align}
    Note that since $\cE\left( \left( \bigcap_{j \geq l} B_j^{(i)} \right)^\circ, r_i \right)$ is an increasing family of sets, we see that
    \begin{align}
        \bigcup_{l \geq 1} \left(\bigcap_{i=1}^k\cE\left( \left( \bigcap_{j \geq l} B_j^{(i)} \right)^\circ, r_i \right)  \right) = \bigcap_{i=1}^k \left( \bigcup_{l \geq 1} \cE\left( \left( \bigcap_{j \geq l} B_j^{(i)} \right)^\circ, r_i \right)  \right).
    \end{align}
    This combined with Lemma \ref{MS10 Lem 6.2}(iii) gives that
    \begin{align}
         \lim \left( {\inf \cE_l} \right)^\circ \supset \bigcap_{i=1}^k \cE\left( \bigcup_{l \geq 1} \left( \bigcap_{j \geq l} B_j^{(i)} \right)^\circ , r_i  \right) = \bigcap_{i=1}^k \cE \left( \lim \left(\inf B_j^{(i)} \right)^\circ, r_i   \right) .
   \end{align}
   This proves \eqref{eq: zz 1}. The proof of \eqref{eq: zz 2} is analogous.
\end{proof}

\begin{lem}
\label{lem: zero measure of thin set}
Let $\widetilde{\nu}$ be a measure on $\widetilde{\X}_{n+1}$ of the form
\begin{align}
\widetilde{\nu}(f)
= \int_{\X_{n+1}} \int_{[0,1]^{n+1}}
f([g,\zer][\id_{n+1},\xi] \Tilde{\Gamma}_{n+1}) \, d\xi \, d\nu(g\Gamma_{n+1}),
\end{align}
for some measure $\nu$ on $\X_{n+1}$, where $\X_{n+1}$ is defined as in Section~\ref{subsec: Reduction to SL nR}. Let $k\in\N$. For $1\le i\le k$, let $B_i,C_i\subset\R^{n+1}$ satisfy $B_i\subset C_i$ and
\begin{equation}
m_{\R^{n+1}}(C_i\setminus B_i)=0.
\end{equation}
Then for all $s_1,\ldots,s_k\in\N$,
\begin{equation}
\widetilde{\nu}\!\left(\bigcap_{i=1}^k \cE(B_i,s_i)\right)
=
\widetilde{\nu}\!\left(\bigcap_{i=1}^k \cE(C_i,s_i)\right).
\end{equation}
\end{lem}
\begin{proof}
   By Lemma~\ref{MS10 Lem 6.2}(i), we have $\bigcap_{i=1}^k \cE(C_i,s_i) \supset \bigcap_{i=1}^k \cE(B_i,s_i) $. Therefore, it is enough to show that the set
    \begin{align}
        \bigcap_{i=1}^k \cE(C_i,s_i) &\setminus \bigcap_{i=1}^k \cE(B_i,s_i) \subset \bigcup_{i=1}^k \left( \cE(C_i, s_i) \setminus \cE(B_i, s_i) \right) \notag \\
        &= \bigcup_{i=1}^k \left\{ [A,b]\widetilde{\Gamma}_{n+1} \in \tXn: (A\Z^n+ b) \cap C_i \geq s_i,  (A\Z^n+ b) \cap B_i < s_i \right\} \notag \\
        & \subset \bigcup_{i=1}^k \left\{ [A,b]\widetilde{\Gamma}_{n+1} \in \tXn: (A\Z^n+ b) \cap (C_i\setminus B_i) \neq \emptyset \right\}
    \end{align}
    has zero measure with respect to $\tnu$.

    To prove this, let $G_{n+1}$ and $\Gamma_{n+1}$ be as in Section~\ref{subsec: Reduction to SL nR}, and fix $\cF$ a fundamental domain for right action of $\Gamma_{n+1}$ on $G_{n+1}$. Then, observe that for each $1 \leq i \leq k$, we have
    \begin{align}
        & \widetilde{\nu} \left( \left\{ [A,b]\widetilde{\Gamma}_{n+1} \in \tXn: (A\Z^n+ b) \cap (C_i\setminus B_i) \neq \emptyset \right\} \right) \notag \\
        &\leq \nu \otimes m_{\R^{n+1}} \left( \left\{ (A \Gamma_{n+1}, \xi): A \in \cF, \xi \in [0,1]^{n+1},    (A\Z^{n+1} + A\xi) \cap (C_i \setminus B_i)  \neq \emptyset \right\} \right) \notag \\
        &\leq \sum_{w \in \Z^{n+1}} \nu \otimes m_{\R^{n+1}} \left( \left\{ (A \Gamma_{n+1}, \xi): \  A \in \cF,\  \xi \in [0,1]^{n+1}, \     A\cdot(w+\xi) \in  C_i \setminus B_i \right\} \right) \notag \\
        &= \nu \otimes m_{\R^{n+1}} \left( \left\{ (A \Gamma_{n+1}, \xi): \  A \in \cF,\  \xi \in \R^{n+1}, \     A\cdot\xi \in  C_i \setminus B_i  \right\} \right) =0,
    \end{align}
    where the last equality follows using Fubini's theorem and the fact that for each $A \in \cF$, we have
    \begin{equation}
    m_{\R^{n+1}}\left( \left\{\xi \in \R^{n+1}: A\cdot\xi \in  C_i \setminus B_i   \right\} \right)= m_{\R^{n+1}}(C_i \setminus B_i)=0.
    \end{equation}
    Hence, the lemma follows.
\end{proof}

\begin{lem}
\label{lem: main appendix}
   Let $k \in \N$. For each $1 \le i \le k$, let $\cA_i\subset \R^{n}$ be a bounded set with boundary of Lebesgue measure zero, and let $r_i \in \N$. Let $(\widetilde{\nu}_l)_l$ be a sequence of probability measures on $\tXn$ converging to a probability measure $\tnu_\infty$, satisfying
    \begin{align}
        \tnu_\infty(f)= \int_{\X_{n+1}} \int_{[0,1]^{n+1}} f([g,\zer] [\id_{n+1},\xi] \Tilde{\Gamma}_{n+1}) \, d\xi d\nu_\infty(g\Gamma_{n+1}),
    \end{align}
    for some measure $\nu_\infty$ on $\X_{n+1}$ and for all $f \in C_c(\tXn)$. Then the sequence
    \begin{align}
        \tnu_l\left(\left\{[A,b]\widetilde{\Gamma} \in \widetilde{\X}_{n+1}: A\Z^{n+1} \cap \cZ(\cA_i, l)= r_i \text{ for all } 1\le i \leq k \right\} \right) \label{eq: yy 1}
    \end{align}
    converges to
    \begin{align}
        \tnu_\infty \left(\left\{[A,b]\widetilde{\Gamma} \in \widetilde{\X}_{n+1}: A\Z^{n+1} \cap \cZ(\cA_i)= r_i \text{ for all } 1\le i \leq k \right\} \right), \label{eq: yy 2}
    \end{align}
    where $\cZ(\cA)$ and $\cZ(\cA, l)$ is defined as in \eqref{eq: def C A} and \eqref{eq: def C A Q}.
\end{lem}
\begin{proof}
    Fix $s=(s_1, \ldots, s_k) \in \N^k$, and define
    \begin{equation}
    \cE_l^s= \bigcap_{i=1}^k \cE\left( \cZ(\cA_i,l),  s_i\right).
    \end{equation}
    Using Lemma~\ref{lem: characteristic} and~\ref{lem: liminf calculations}, we see that 
    \begin{align}
        \liminf_{l \rightarrow \infty} \tnu_l\left( \cE_l^s \right) &\geq \tnu_\infty\left( \bigcap_{i=1}^k \cE \left( \lim \left( \inf \cZ(\cA_i, l)\right)^\circ, s_i \right) \right)= \tnu_\infty\left( \bigcap_{i=1}^k \cE\left(  (\cA_i)^\circ \times (-1,0), s_i  \right) \right) , \\
        \limsup_{l \rightarrow \infty} \tnu_l\left( \cE_l^s \right) &\leq \tnu_\infty\left( \bigcap_{i=1}^k \cE \left( \lim \overline{\sup \cZ(\cA_i, l) }, s_i \right) \right)= \tnu_\infty\left( \bigcap_{i=1}^k \cE\left(  \overline{\cA}_i \times [-1,0], s_i  \right) \right) .
    \end{align}
    Note that for each $i$, the set $(\cA_i)^\circ \times (-1,0) \subset \cZ(\cA_i) \subset \overline{\cA}_i \times [-1,0]$ and
    \begin{equation}
    m_{\R^{n+1}} \left(\overline{\cA}_i \times [0,1] \setminus (\cA_i)^\circ \times (-1,0) \right) =0.
    \end{equation}
    Therefore, using Lemma~\ref{lem: zero measure of thin set}, we see that
    \begin{align}
        \nonumber
        \lim_{l \rightarrow \infty} \tnu_l(\cE_l^s)= \tnu_\infty\left( \bigcap_{i=1}^k \cE(\cZ(\cA_i), s_i) \right) .
    \end{align}
    
    The lemma now follows by writing the expressions~\eqref{eq: yy 1} and~\eqref{eq: yy 2} in terms of $\tnu_l\left( \cE_l^s \right)$ and $\tnu_\infty\left( \cap_{i=1}^k \cE(\cZ(\cA_i), s_i) \right)$ respectively. 
\end{proof}

\section{Proof of Theorems~\ref{thm:main0},~\ref{thm:main} and~\ref{thm:main2}}\label{sec:introproofs}

In this section, we prove Theorems~\ref{thm:main0},~\ref{thm:main} and~\ref{thm:main2}. To this end, we will need the following notation.
For $\cA \subset \R^n$, define
\begin{align}
\label{eq: def C A}    \zA &= \{(x,y) \in \R^n \times \R:  {x} \in    \cA, \  y \in  [-1,0] \}, \\ \label{eq: def C A Q}
    \cZ(\cA,Q) &= \left\{(x,y) \in \R^n \times \R: -1\leq y \leq 0, {x} \in   \left(1 + \frac{y\Delta}{Q}\right)  \cA \right\}.
\end{align}

\begin{proof}[Proof of Theorem~\ref{thm:main}]
Fix $\alpha \in \R^n$, and let $\cA,\cD\subset\R^n$ be bounded sets with boundaries of Lebesgue measure zero. Let $\sigma\in(0,\infty]$, and fix $(Q,\Delta,\varepsilon)$ be any sequence satisfying \eqref{scalim} and $\varepsilon \Delta^{1/n}\to \sigma$.

 Note that for any $(p,q) \in \R^{n+1}$ and $\xi \in \cD_{\e, \alpha}$, we have
\begin{equation}
\frac{p}{q} \in  \xi + \eta_{Q,\Delta} \cA, \quad Q-\Delta \leq q \leq Q,
\end{equation}
if and only if 
\begin{equation}
a(\Delta) u(-\alpha- \e \zeta) \begin{pmatrix}
    p \\ q
\end{pmatrix} + \begin{pmatrix}
    \zer \\ -Q\Delta^{-1}
\end{pmatrix} \in \cZ(\cA,Q),
\end{equation}
where $\zeta = (\xi-\alpha)/\e \in \cD$.
Therefore for every $k \in \N$,
\begin{align}
    E_Q(k, \cA) &= \frac{1}{\vol(\cD)} \vol\left( \left\{ \zeta \in \cD: \left| \left(a(\Delta) u(-\alpha- \e \zeta) \Z^{n+1}  + \begin{pmatrix}
    \zer \\ -Q\Delta^{-1}
\end{pmatrix}\right) \cap \cZ(\cA,Q) \right|= k   \right\}\right) \nonumber\\
&= \widetilde{\nu}_{Q,x} \left( \{ [A,b]\widetilde{\Gamma}_{n+1}: | (A\Z^{n+1}+b) \cap \cZ(\cA,Q)|=k \} \right), \label{eq: ww 2}
\end{align}
where $\widetilde{\nu}_{Q,x}$ is defined as in \eqref{eq: def nu} for $\lambda$ equal to normalised Lebesgue measure restricted to $\cD$ and $x= u(-\alpha)\widetilde{\Gamma}_{n+1}$. The theorem now follows directly from Theorem~\ref{thm: main dynamical theorem} and Lemma~\ref{lem: main appendix} with
\begin{align}
\label{eq:def E sigma pointwise}
    E_{0,\sigma \cD}(k,\cA)= \mu^{(\sigma)}_{\lambda,\Tilde{\Gamma}_{n+1}}\left( \{ [A,b]\widetilde{\Gamma}_{n+1}: |(A\Z^{n+1}+b) \cap \cZ(\cA)|=k \} \right),
\end{align}
\begin{align}
\label{eq:def E sigma pointwise infinity}
    E(k,\cA)= \mu_{\X_{n+1}}\left( \{ [A,b]\widetilde{\Gamma}_{n+1}: |(A\Z^{n+1}+b) \cap \cZ(\cA)|=k \} \right),
\end{align}
where $\mu^{(\sigma)}_{\lambda,x}$ is defined as in~\eqref{eq: thm : main dynamical 2}.
\end{proof}

\begin{proof}[Proof of Theorem~\ref{thm:main2}]

Fix $\alpha \in \R^n$, and let $\cA,\cD\subset\R^n$ be bounded sets with boundaries of Lebesgue measure zero. Let $\sigma\in(0,\infty]$, and fix $(Q,\Delta,\varepsilon)$ be any sequence satisfying \eqref{scalim} and $\varepsilon \Delta^{1/n}\to \sigma$.

    Analogous to the proof of Theorem~\ref{thm:main}, one can check that for any $r \in \N$, for any $ \cA_1,\ldots,\cA_r$ bounded subsets of $\R^n$ with boundary of Lebesgue measure zero and any $k_1, \ldots, k_r \in \N$, we have
    \begin{align}
   & \frac{1}{\vol(\cD_{\e,\alpha})}
\vol\!
\left\{\xi\in\cD_{\e,\alpha}:
|\{(p,q)\in\Z^n\times\N:
q\in[Q-\Delta,Q],\;
\eta_{Q,\Delta}^{-1}(q^{-1}p-\xi)\in\cA_i\}|
= k_i\;\forall\,1\le i\le r
\right\}
 \nonumber\\
    &\qquad= \widetilde{\nu}_{Q,x} \{ [A,b]\widetilde{\Gamma}_{n+1}: |(A\Z^{n+1}+b) \cap \cZ(\cA_i,Q)|=k_i \text{ for all } 1\leq i \leq r  \} .  \label{eq: ww 3}
    \end{align}
    where $\widetilde{\nu}_{Q,x}$ is defined as in \eqref{eq: def nu} for $\lambda$ equal to normalised Lebesgue measure restricted to $\cD$ and $x= u(-\alpha)\widetilde{\Gamma}_{n+1}$. 

    The theorem now follows directly from Theorem~\ref{thm: main dynamical theorem} and Lemma~\ref{lem: main appendix} with the point process $X$ and $X_{0, \sigma\cD}$ given by $\Lambda$ varying in $\tXn$ according to $\mu_{\tXn}$ and $\Lambda_{0, \sigma\cD}$ varying in $\tXn$ according to $\mu^{(\sigma)}_{\lambda,\widetilde{\Gamma}_{n+1}}$, defined as in~\eqref{eq: thm : main dynamical 2} for $x= u(-\alpha)\widetilde{\Gamma}_{n+1}$.
\end{proof}

\begin{proof}[Proof of Theorem~\ref{thm:main0}]
  Theorem~\ref{thm:main0} follows directly from Theorem~\ref{thm:main} using~\eqref{eq:void to gap 1}, \eqref{eq:void to gap 2} as discussed in Section~\ref{subsec:Local statistics in higher dimension}.
\end{proof}

\begin{proof}[Proof of the explicit formula \eqref{p_alf} for the gap density]

The key is an explicit description of the process $X_{0,\cD}$.
The limit measure $\Lambda_{0,\cD}$ defining the process is described in \eqref{LambdaAD}. So
\begin{equation}
X_{0,\cD}   = \sum_{\substack{(m,n)\in\Z^2 \\ -t\leq y+n \leq 0}} \delta_{t(\omega n+x+m)} ,
\end{equation}
where $t$ is a random variable uniformly distributed in $\sigma\cD$, $\omega$, $x$, $y$ are uniformly distributed in the unit interval $[0,1]$ or, equivalently, in $\R/\Z$. The process is stationary under $x\mapsto x+s$ for any $s\in\R$. By this translation invariance, we see that the distribution of the process $X_{0,\cD}$ remains the same if we replace $n$ by $n+n_0$ in the summation, for any fixed integer $n_0$. In other words, the distribution of the process only depends on the (random) number $N_t(y)=|\Z \cap [-t-y,-y]|$, i.e.,
\begin{equation}
X_{0,\cD}   \eqdist \sum_{n=1}^{N_t(y)} \sum_{m\in\Z} \delta_{t(\omega n+x+m)} .
\end{equation}
Conditioning on $t$, we have 
\begin{equation}
\mathbb{P}( N_t(y) = k \,|\, t) = \max\big( 1 - |k-t|, 0\big) .
\end{equation}
if we condition the process $X_{0,\cD}$ on $t$ and $N_t(y) = k$, we obtain the homogeneous process
\begin{equation}
X_{0,\cD}^{t,k} =\sum_{n=1}^{k} \sum_{m\in\Z} \delta_{t(\omega n+x+m)} .
\end{equation}
with intensity $\frac{k}{t}$. This means the average spacing between points is $\frac{t}{k}$.
Its void probability is
\begin{equation}
E_{0,\cD}^{t,k}(0,\cA) =  \mathbb{P}( X_{0,\cD}^{t,k}(\cA)=0 ),
\end{equation}
and we obtain the gap distribution via the standard formula 
\begin{equation}
P_{0,\cD}^{t,k}(s) = \frac{t}{k} \; \frac{d}{ds} E_{0,\cD}^{t,k}(0,[0,s]) .
\end{equation}
The void distribution for $X_{0,\cD}$ is
\begin{equation}
E_{0,\cD}(0,\cA) = \mathbb{P}( X_{0,\cD}(\cA)=0 ) = \frac{1}{|\cD|} \sum_{k=1}^\infty \int_{\cD}   \max\big( 1 - |k-t|, 0\big) E_{0,\cD}^{t,k}(0,\cA) dt .
\end{equation}
We conclude that the gap distribution of $X_{0,\cD}$ (which has intensity one)
\begin{equation}
P_{0,\cD}(s) = \frac{d}{ds} E_{0,\cD}(0,[0,s]) =  \frac{1}{|\cD|} \sum_{k=1}^\infty \int_{\cD}  \max\big( 1 - |k-t|, 0\big) k P_{0,\cD}^{t,k}(s) \frac{dt}{t} .
\end{equation}
Now 
\begin{equation}
P_{0,\cD}^{t,k}(s) = \int_s^\infty \int_0^1 \frac{1}{k} \sum_{j=1}^k  \delta_{t s_{k,j}(\omega)}(s') \, d\omega \, ds',
\end{equation}
which yields the following formula for the gap density (with $s_{k,j}(\omega)$ as defined in Section \ref{sec1.1}),
\begin{equation}
p_{0,\cD}(s) =  \frac{1}{|\cD|} \sum_{k=1}^\infty \sum_{j=1}^k \int_{\cD}  \int_0^1 \max\big( 1 - |k-t|, 0\big) \delta_{t s_{k,j}(\omega)}(s) \, d\omega\, \frac{dt}{t} .
\end{equation}
Formula \eqref{p_alf} follows after carrying out the integration over $t$.
\end{proof}

\begin{proof}[Proof of convergence \eqref{zerosigma} of the gap distribution for $\sigma=0$, $\alpha=0$]
We assume without loss of generality that $\cD=[0,b]$ for some $b>0$. Consider the $0=\xi_1\leq \xi_2\leq \ldots\leq \xi_N$ as in \eqref{xis}, where $\xi_j=p_j/q_j$ with $p_j,q_j$ are the integers defining $\xi_j$ in the multiset \eqref{FQ}. This requires that $q_j\in [Q-\Delta,Q]$ and $p_j\leq q_j \epsilon b$. 

Let us assume $j$ is such that $q_j-1\in [Q-\Delta,Q]$ and $p_j\leq (q_j-1) \epsilon b$. Then there is some $\xi_\ell>\xi_j$ which  equals $p_j/(q_j-1)$, and hence we have the following upper bound for the $j$-th gap,
\begin{equation}
\xi_{j+1}-\xi_j \leq \xi_\ell-\xi_j = \frac{p_j}{q_j (q_j-1)} \leq \frac{\epsilon b}{Q-\Delta} . 
\end{equation}
and hence, in units of the average gap size,
\begin{equation}
Q\Delta\, (\xi_{j+1}-\xi_j) = O( \epsilon \Delta) . 
\end{equation}
By assumption, $\epsilon \Delta\to 0$, and therefore the $j$-gap will not contribute to $P_N(s)$ for any $s>0$.

To show that this is true for almost all $\xi_j$ (i.e., a set of full density), we need to estimate the number of  $\xi_j$ for which the above hypotheses are violated, i.e., count those integers $(p,q)$ for which 
\begin{equation}
q\in [Q-\Delta,Q]\smallsetminus [Q+1-\Delta,Q+1], \qquad p\leq q \epsilon b,
\end{equation}
plus the number of $(p,q)$ for which
\begin{equation}
q\in [Q-\Delta,Q], \qquad (q-1) \epsilon b < p\leq q \epsilon b.
\end{equation}
In the first case, there is at most one $q$ and hence the total number bounded above by $\epsilon Q b$; in the second case, there is at most one $p$ (for $\epsilon b<1$) and hence the total number is bounded by $\Delta+1$. In view of \eqref{Ntot}, these are negligible when compared to $N\sim \epsilon Q\Delta b$; recall our assumption that $\epsilon Q\to\infty$.
\end{proof}

\section{Directions of lattice points observed from a distance}\label{sec:ABCZ}

We generalise the setting of Anderson et al.\ (as described in Section \ref{intro:ABCZ}) by considering more general rectangles of the form $R=[a,b]\times [c,d]$ with $a<b$ and $c<d$, and the observer position
\begin{equation}
P_{\vartheta,J} = (-\vartheta(J),0) ,
\end{equation}
where we assume the function $\vartheta : \Z_{\geq 1} \to \R_{>0}$ satisfies
\begin{equation}\label{vartheta}
\frac{\vartheta(J)}{J} \to \infty,
\qquad
\frac{J^2}{\vartheta(J)} \to \sigma \in [0,\infty] \qquad (J\to\infty).
\end{equation}

For $\xi\geq 0$, we denote the gap distribution function by
\begin{equation}
G_{\vartheta,J}(\xi) = \frac{\left|\left\{ j\leq N-1 : \alpha_{J,j+1} - \alpha_{J,j} > \xi \Delta_{J,\mathrm{av}} \right\}\right|}{N-1} .
\end{equation}
The original setting in \cite{ABCZ} corresponds to the choice $\vartheta(J)=t J^\alpha$ (with $\alpha=2$) and $R=[-1,1]\times [0,1]$. Theorem \ref{thm:ABCZ} is therefore a special case of the following.

\begin{thm}\label{thm:ABCZ22}
Fix $\sigma\in[0,\infty]$ and $\vartheta$ as above. 
Then, for every $\xi > 0$, 
\begin{equation}
\lim_{J\to\infty} G_{\vartheta,J}(\xi) = G_\vartheta(\xi),
\end{equation}
where
\begin{equation}
\label{eq:thm:ABCZ22}
G_\vartheta(\xi) =
\begin{cases}
0 & \text{if } \sigma=0 \\[6pt]
\displaystyle \int_{\xi}^\infty p_{0,\,(b-a) \sigma[c,d]}(s)\,ds
& \text{if } 0<\sigma<\infty \\[10pt]
\displaystyle \int_{\xi}^\infty p(s)\,ds
& \text{if } \sigma=\infty,
\end{cases}
\end{equation}
and the densities $p_{0,\cD}$ and $p$ are as in Theorem~\ref{thm:main0}.
\end{thm}

\begin{proof}
%

\medskip
\noindent
\textbf{Step 1. Reduction to linear gaps.}
Consider
\begin{equation}
\label{eq:first set}
\left\{ \frac{p}{q+\vartheta(J)} : (q,p) \in \Z^2 \cap R_J \right\}, \qquad R_J=[aJ,bJ]\times [cJ,dJ]
\end{equation}
ordered as $\beta_{J,1} \le \cdots \le \beta_{J,N}$, where $N=N_J$ is the total number of points, and then set
\begin{equation}
\label{eq:red 1}
G_{\vartheta,J}^{(1)}(\xi) = \frac{\left|\left\{ j\leq N-1 : (b-a)J\vartheta(J)\,(\beta_{J,j+1} - \beta_{J,j}) > \xi \right\}\right|}{N-1} .
\end{equation}
We claim that for every $\xi>0$, every sufficiently small $\epsilon>0$, we have
\begin{equation}
\liminf_{J\to\infty} G_{\vartheta,J}^{(1)}(\xi+\epsilon) \leq \liminf_{J\to\infty} G_{\vartheta,J}(\xi)
\leq \limsup_{J\to\infty} G_{\vartheta,J}(\xi) \leq \limsup_{J\to\infty} G_{\vartheta,J}^{(1)}(\xi-\epsilon) .
\end{equation}
Indeed,
\begin{align*}
\alpha_{J,j+1} - \alpha_{J,j}
&= \tan^{-1}(\beta_{J,j+1}) - \tan^{-1}(\beta_{J,j}) \\
&= \tan^{-1}\!\left( \frac{\beta_{J,j+1} - \beta_{J,j}}{1 + \beta_{J,j+1}\beta_{J,j}} \right).
\end{align*}
Since $\beta_{J,j} = O\!\bigl(J/\vartheta(J)\bigr)$, we have
\begin{equation}
\beta_{J,j+1}\beta_{J,j} = O\!\left(\frac{J^2}{\vartheta(J)^2}\right),
\end{equation}
and hence
\begin{equation}
\alpha_{J,j+1} - \alpha_{J,j}
= (\beta_{J,j+1} - \beta_{J,j})
\left(1 + O\!\left(\frac{J^2}{\vartheta(J)^2}\right)\right)
+ O\!\bigl((\beta_{J,j+1} - \beta_{J,j})^3\bigr).
\end{equation}
Moreover,
\begin{equation}
\alpha_{1,N} = \frac{cJ}{\vartheta(J)} + O\!\left(\frac{J^2}{\vartheta(J)^2}\right),
\qquad
\alpha_{J,N} = \frac{dJ}{\vartheta(J)} + O\!\left(\frac{J^2}{\vartheta(J)^2}\right),
\end{equation}
\begin{equation}
N = (b-a)(d-c) J^2 + O(J),
\end{equation}
so that
\begin{equation}\label{meangap}
\Delta_{J,\mathrm{av}} 
= \frac{1}{(b-a)J\vartheta(J)} 
+ O\!\left(\frac{1}{J^2\vartheta(J)} + \frac{1}{\vartheta(J)^2}\right).
\end{equation}
Combining these estimates, we obtain uniformly in $j$,
\begin{equation}
\Delta_{J,\mathrm{av}}^{-1}(\alpha_{J,j+1} - \alpha_{J,j})
= (b-a)J\vartheta(J)\,(\beta_{J,j+1} - \beta_{J,j}) (1 + o(1)),
\end{equation}
which proves the claim.

The proof of the claim for $\sigma=0$ follows from the proof of convergence \eqref{zerosigma} at the end of Section~\ref{sec:introproofs}. Namely, for a full-density set of $j$, we have
$$
\left|\frac{p_j}{q_j + \vartheta(J)}- \frac{p_j}{q_j-1 + \vartheta(J)} \right|= \frac{p_j}{(q_j+ \vartheta(J)) (q_j-1+ \vartheta(J))} = O(J \vartheta(J)^{-2})= o((J \vartheta(J))^{-1}),
$$
which converges to zero after scaling by $\Delta_{J,\mathrm{av}}(J)$, cf.~\eqref{meangap}. We can therefore assume $\sigma>0$ throughout the remaining proof.

\medskip
\noindent
\textbf{Step 2. Geometric replacement.}
Consider the elements in the multiset
\begin{equation}
\label{eq:reduction 2}
S_J:= \left\{ \frac{y}{x} : (y,x) \in \Z^2+ \begin{pmatrix}
    0 \\ \vartheta(J)
\end{pmatrix}, \ \vartheta(J)+aJ \le x \le \vartheta(J)+bJ,\ \frac{c J}{\vartheta(J)} \le \frac{y}{x} \le \frac{d J}{\vartheta(J)} \right\},
\end{equation}
ordered as $\gamma_{J,1} \le \cdots \le \gamma_{J,M}$ (with $M=M_J$ the total number of elements), and set
\begin{equation}
G_{\vartheta,J}^{(2)}(\xi) = \frac{\left|\left\{ j\leq M-1 : (b-a)J\vartheta(J)\,(\gamma_{J,j+1} - \gamma_{J,j}) > \xi \right\}\right|}{M-1} .
\end{equation}

Let
\begin{equation}
C_J = \big\{(y,x) : \vartheta(J)+aJ \le x \le  \vartheta(J)+bJ,\ cJ \le y \le dJ\big\},
\end{equation}
\begin{equation}
\label{eq:def d J}
D_J = \bigg\{(y,x) :\ \vartheta(J)+aJ \le x \le \vartheta(J)+bJ,\ \frac{c J}{\vartheta(J)} \le \frac{y}{x} \le \frac{d J}{\vartheta(J)}  \bigg\}.
\end{equation}
Then $\beta_{J,j}$ corresponds to points in $\left(\Z^2+ \begin{pmatrix}
    0 \\ \vartheta(J)
\end{pmatrix} \right) \cap C_J$, while $\gamma_{J,j}$ corresponds to $\left(\Z^2+ \begin{pmatrix}
    0 \\ \vartheta(J)
\end{pmatrix} \right) \cap D_J$.
A geometric comparison shows
\begin{equation}
\left|\left(\Z^2+ \begin{pmatrix}
    0 \\ \vartheta(J)
\end{pmatrix} \right) \cap (C_J \triangle D_J)\right| = O(J^3 \vartheta(J)^{-1}) = o(J^2),
\end{equation}
while $M \sim (b-a)(d-c)J^2$. Hence the sequences differ in only $o(M)$ points, which implies 
\begin{equation}
\label{eq:red 2}
\liminf_{J\to\infty} G_{\vartheta,J}^{(2)}(\xi+\e) = \liminf_{J\to\infty} G_{\vartheta,J}^{(1)}(\xi)
\leq \limsup_{J\to\infty} G_{\vartheta,J}^{(1)}(\xi) = \limsup_{J\to\infty} G_{\vartheta,J}^{(2)}(\xi-\e).
\end{equation}

\medskip
\noindent
\textbf{Step 3. Translation to void statistics.}
By \eqref{eq:void to gap 1} and \eqref{eq:void to gap 2}, it suffices to determine the limiting void probabilities associated with the set $S_J$. More precisely, for $s\ge 0$, we consider
\begin{align}
\label{eq:red 3 1}
\frac{1}{d-c}
\left|
\left\{
\xi\in[c,d]:
\#
\left(
\left(
\frac{J\xi}{\vartheta(J)}
+
\frac{1}{(b-a)\vartheta(J)J}[0,s]
\right)
\cap S_J
\right)
=0
\right\}
\right|,
\end{align}
and study its limit as $J\to\infty$.

Arguing exactly as in the proof of Theorem~\ref{thm:main}, this quantity~\eqref{eq:red 3 1} can be rewritten as
\begin{align}
\label{eq:red 3 1 0}
\mu_J
\left\{
(A,w)\widetilde{\Gamma}_{n+1}:
\left(
\left[
\begin{pmatrix}
b-a & 0\\
0 & (b-a)^{-1}
\end{pmatrix},
\begin{pmatrix}
0\\
-a(b-a)^{-1}
\end{pmatrix}
\right]
(A\mathbb Z^2+w)
\right)
\cap
\mathcal Z([0,s],Q)
=\emptyset
\right\},
\end{align}
where $\mathcal Z([0,s],Q)$ is defined as in \eqref{eq: def C A Q} with
\begin{equation}
Q=\vartheta(J),
\qquad
\Delta=J,
\end{equation}
and $\mu_J$ is the probability measure defined by
\begin{align}
\label{eq:red 3 3}
\mu_J(f)
=
\frac1{d-c}
\int_c^d
f\!\left(
\left[
\begin{pmatrix}
J&0\\
0&J^{-1}
\end{pmatrix}
\begin{pmatrix}
1&-J \vartheta(J)^{-1}\xi\\
0&1
\end{pmatrix},
\begin{pmatrix}
0\\
-\vartheta(J)J^{-1}
\end{pmatrix}
\right]
z_J
\right)
\,d\xi,
\end{align}
where
\begin{equation}
z_J=
\begin{pmatrix}
I_2,
\begin{pmatrix}
0\\
-\vartheta(J)
\end{pmatrix}
\end{pmatrix}
\widetilde{\Gamma}_{n+1}.
\end{equation}

Repeating the proof of Lemma~\ref{lem:reduction 1 1}, with the substitutions
\begin{equation}
(Q,\Delta,\varepsilon)
=
(\vartheta(J),J,J\vartheta(J)^{-1})
\end{equation}
and with $\lambda$ equal to the normalised Lebesgue measure on $[c,d]$, shows that every subsequential limit of $\mu_J$ is invariant under the subgroup
\begin{equation}
\widetilde U_2
=
\left\{
\left[
I_2,
\begin{pmatrix}
\zeta\\
0
\end{pmatrix}
\right]:
\zeta\in\mathbb R
\right\}.
\end{equation}

Combined with proof of Proposition~\ref{lem: reduction to X} and the convergence of $\pi_*(\mu_J)$ provided by Theorem~\ref{thm: main dynamical theorem}, this invariance implies that $\mu_J$ converges to
\begin{equation}
\mu_\infty
=
\begin{cases}
\mu_{\widetilde X_2},
&
\sigma=\infty,
\\[6pt]
\mu_{\lambda,\widetilde\Gamma_2}^{(\sigma)},
&
0<\sigma<\infty,
\end{cases}
\end{equation}
where $\lambda$ denotes the normalised Lebesgue measure on $[c,d]$ and $\mu_{\lambda,\widetilde\Gamma_2}^{(\sigma)}$ is defined in \eqref{eq:def mu sigma}.

Lemma~\ref{lem: main appendix} therefore implies that \eqref{eq:red 3 1 0} converges to
\begin{align}
\label{eq:red 3 4}
\mu_\infty
\left\{
(A,w)\widetilde{\Gamma}_{n+1}:
\left(
\left[
\begin{pmatrix}
b-a & 0\\
0 & (b-a)^{-1}
\end{pmatrix},
\begin{pmatrix}
0\\
-a(b-a)^{-1}
\end{pmatrix}
\right]
(A\mathbb Z^2+w)
\right)
\cap
\mathcal Z([0,s])
=\emptyset
\right\},
\end{align}
where $\mathcal Z([0,s])$ is defined in \eqref{eq: def C A}.

Finally, pushing forward $\mu_\infty$ under the map
\begin{equation}
z
\mapsto
\left[
\begin{pmatrix}
b-a&0\\
0&(b-a)^{-1}
\end{pmatrix},
\begin{pmatrix}
0\\
-a(b-a)^{-1}
\end{pmatrix}
\right]
z,
\end{equation}
shows that the limiting void probability equals
\begin{equation}
\begin{cases}
E(0,[0,s]),
&
\sigma=\infty,
\\[6pt]
E_{0,(b-a)\sigma[c,d]}(0,[0,s]),
&
0<\sigma<\infty,
\end{cases}
\end{equation}
where $E(0, \cA)$ and $E_{0,\cD}(0, \cA)$ are defined as in Theorem~\ref{thm:main}.

Using the correspondence between void probabilities and gap distributions given by \eqref{eq:void to gap 1} and \eqref{eq:void to gap 2}, together with the explicit computation carried out in the proof of Theorem~\ref{thm:main0}, we conclude that $G_{\vartheta,J}^{(2)}(\xi)$ converges for every $\xi>0$ to the right-hand side of \eqref{eq:thm:ABCZ22}. The theorem now follows from the reductions established in Steps~1 and~2.
\end{proof}

We remark that (i) if $\vartheta(J)\in\mathbb N$ for all sufficiently large $J$, or (ii) if
\begin{equation}
\frac{J^3}{\vartheta(J)} \to 0
\qquad (J\to\infty),
\end{equation}
then the gap statistics of the set $S_J$ defined in \eqref{eq:reduction 2} can be approximated directly by the gap statistics of $\mathcal F_{Q,\Delta}$, namely $P_M(\xi)$ defined in \eqref{eq:def PN}, with the choice
\begin{equation}
Q=\vartheta(J)+bJ,
\qquad
\Delta=(b-a)J,
\qquad
\varepsilon=\frac{J}{\vartheta(J)},
\qquad
\alpha=0,
\qquad
\mathcal D=[c,d].
\end{equation}
Consequently, in this regime Theorem~\ref{thm:ABCZ22} follows directly from Theorem~\ref{thm:main0}, without the need for the dynamical argument developed in Step~3.

\bibliography{Biblio}

@article {Str04,
    AUTHOR = {Str\"ombergsson, Andreas},
     TITLE = {On the uniform equidistribution of long closed horocycles},
   JOURNAL = {Duke Math. J.},
  FJOURNAL = {Duke Mathematical Journal},
    VOLUME = {123},
      YEAR = {2004},
    NUMBER = {3},
     PAGES = {507--547},
      ISSN = {0012-7094,1547-7398},
   MRCLASS = {11F72 (30F35 37A45 37D40)},
  MRNUMBER = {2068968},
MRREVIEWER = {Boris\ Hasselblatt},
       DOI = {10.1215/S0012-7094-04-12334-6},
       URL = {https://doi-org.bris.idm.oclc.org/10.1215/S0012-7094-04-12334-6},
}

@article {MozesShah,
    AUTHOR = {Mozes, Shahar and Shah, Nimish},
     TITLE = {On the space of ergodic invariant measures of unipotent flows},
   JOURNAL = {Ergodic Theory Dynam. Systems},
  FJOURNAL = {Ergodic Theory and Dynamical Systems},
    VOLUME = {15},
      YEAR = {1995},
    NUMBER = {1},
     PAGES = {149--159},
      ISSN = {0143-3857,1469-4417},
   MRCLASS = {58F11 (22E40 28D10)},
  MRNUMBER = {1314973},
MRREVIEWER = {Garrett\ Stuck},
       DOI = {10.1017/S0143385700008282},
       URL = {https://doi-org.bris.idm.oclc.org/10.1017/S0143385700008282},
}

@article {MarStroGAFA,
    AUTHOR = {Marklof, Jens and Str\"ombergsson, Andreas},
     TITLE = {The periodic {L}orentz gas in the {B}oltzmann-{G}rad limit:
              asymptotic estimates},
   JOURNAL = {Geom. Funct. Anal.},
  FJOURNAL = {Geometric and Functional Analysis},
    VOLUME = {21},
      YEAR = {2011},
    NUMBER = {3},
     PAGES = {560--647},
      ISSN = {1016-443X,1420-8970},
   MRCLASS = {37A60 (37A50 37D50 60J35)},
  MRNUMBER = {2810859},
MRREVIEWER = {Andr\'as\ Kr\'amli},
       DOI = {10.1007/s00039-011-0116-9},
       URL = {https://doi-org.bris.idm.oclc.org/10.1007/s00039-011-0116-9},
}

@article {MV2017,
    AUTHOR = {Marklof, Jens and Vinogradov, Ilya},
     TITLE = {Spherical averages in the space of marked lattices},
   JOURNAL = {Geom. Dedicata},
  FJOURNAL = {Geometriae Dedicata},
    VOLUME = {186},
      YEAR = {2017},
     PAGES = {75--102},
      ISSN = {0046-5755,1572-9168},
   MRCLASS = {37A17 (60B10 82D05)},
  MRNUMBER = {3602886},
MRREVIEWER = {Massimo\ Campanino},
       DOI = {10.1007/s10711-016-0180-2},
       URL = {https://doi.org/10.1007/s10711-016-0180-2},
}

@incollection {Marklof2007,
    AUTHOR = {Marklof, Jens},
     TITLE = {Distribution modulo one and {R}atner's theorem},
 BOOKTITLE = {Equidistribution in number theory, an introduction},
    SERIES = {NATO Sci. Ser. II Math. Phys. Chem.},
    VOLUME = {237},
     PAGES = {217--244},
 PUBLISHER = {Springer, Dordrecht},
      YEAR = {2007},
      ISBN = {978-1-4020-5403-7; 1-4020-5403-3},
   MRCLASS = {11K31 (11J71 11K06 28D15 37A45)},
  MRNUMBER = {2290501},
MRREVIEWER = {Thomas\ Ward},
       DOI = {10.1007/978-1-4020-5404-4\_11},
       URL = {https://doi.org/10.1007/978-1-4020-5404-4_11},
}

@article {MS2017,
    AUTHOR = {Marklof, Jens and Str\"ombergsson, Andreas},
     TITLE = {The three gap theorem and the space of lattices},
   JOURNAL = {Amer. Math. Monthly},
  FJOURNAL = {American Mathematical Monthly},
    VOLUME = {124},
      YEAR = {2017},
    NUMBER = {8},
     PAGES = {741--745},
      ISSN = {0002-9890,1930-0972},
   MRCLASS = {11K06 (11H06 52C05)},
  MRNUMBER = {3706822},
       DOI = {10.4169/amer.math.monthly.124.8.741},
       URL = {https://doi.org/10.4169/amer.math.monthly.124.8.741},
}

@article {Marklof2021,
    AUTHOR = {Marklof, Jens},
     TITLE = {Random lattices in the wild: from {P}\'olya's orchard to
              quantum oscillators},
   JOURNAL = {Lond. Math. Soc. Newsl.},
  FJOURNAL = {London Mathematical Society. Newsletter},
    NUMBER = {493},
      YEAR = {2021},
     PAGES = {42--49},
      ISSN = {2516-3841,2516-385X},
   MRCLASS = {60G55},
  MRNUMBER = {4542360},
}

@article {Dahlqvist1997,
    AUTHOR = {Dahlqvist, Per},
     TITLE = {The {L}yapunov exponent in the {S}inai billiard in the small
              scatterer limit},
   JOURNAL = {Nonlinearity},
  FJOURNAL = {Nonlinearity},
    VOLUME = {10},
      YEAR = {1997},
    NUMBER = {1},
     PAGES = {159--173},
      ISSN = {0951-7715,1361-6544},
   MRCLASS = {58F11 (82C05)},
  MRNUMBER = {1430746},
MRREVIEWER = {Nikolai\ Chernov},
       DOI = {10.1088/0951-7715/10/1/011},
       URL = {https://doi.org/10.1088/0951-7715/10/1/011},
}

@article {BZ2007,
    AUTHOR = {Boca, Florin P. and Zaharescu, Alexandru},
     TITLE = {The distribution of the free path lengths in the periodic
              two-dimensional {L}orentz gas in the small-scatterer limit},
   JOURNAL = {Comm. Math. Phys.},
  FJOURNAL = {Communications in Mathematical Physics},
    VOLUME = {269},
      YEAR = {2007},
    NUMBER = {2},
     PAGES = {425--471},
      ISSN = {0010-3616,1432-0916},
   MRCLASS = {37A25 (11B57 37A60 37D50 82C05 82C40)},
  MRNUMBER = {2274553},
MRREVIEWER = {Nikolai\ Chernov},
       DOI = {10.1007/s00220-006-0137-7},
       URL = {https://doi.org/10.1007/s00220-006-0137-7},
}

@article {BCZ2001,
    AUTHOR = {Boca, Florin P. and Cobeli, Cristian and Zaharescu, Alexandru},
     TITLE = {A conjecture of {R}. {R}.\ {H}all on {F}arey points},
   JOURNAL = {J. Reine Angew. Math.},
  FJOURNAL = {Journal f\"ur die Reine und Angewandte Mathematik. [Crelle's
              Journal]},
    VOLUME = {535},
      YEAR = {2001},
     PAGES = {207--236},
      ISSN = {0075-4102,1435-5345},
   MRCLASS = {11N37 (11B57)},
  MRNUMBER = {1837099},
MRREVIEWER = {Dmitry\ Y.\ Kleinbock},
       DOI = {10.1515/crll.2001.049},
       URL = {https://doi.org/10.1515/crll.2001.049},
}

@article {PSZ2016,
    AUTHOR = {Polanco, Gerem\'ias and Schultz, Daniel and Zaharescu,
              Alexandru},
     TITLE = {Continuous distributions arising from the three gap theorem},
   JOURNAL = {Int. J. Number Theory},
  FJOURNAL = {International Journal of Number Theory},
    VOLUME = {12},
      YEAR = {2016},
    NUMBER = {7},
     PAGES = {1743--1764},
      ISSN = {1793-0421,1793-7310},
   MRCLASS = {11B57 (11K36)},
  MRNUMBER = {3544409},
MRREVIEWER = {Robert\ F.\ Tichy},
       DOI = {10.1142/S1793042116501074},
       URL = {https://doi.org/10.1142/S1793042116501074},
}

@article {Greenman1996,
    AUTHOR = {Greenman, C. D.},
     TITLE = {The generic spacing distribution of the two-dimensional
              harmonic oscillator},
   JOURNAL = {J. Phys. A},
  FJOURNAL = {Journal of Physics. A. Mathematical and General},
    VOLUME = {29},
      YEAR = {1996},
    NUMBER = {14},
     PAGES = {4065--4081},
      ISSN = {0305-4470,1751-8121},
   MRCLASS = {81Q50 (11Z05)},
  MRNUMBER = {1406922},
MRREVIEWER = {P\'eter\ P\'al\ L\'evay},
       DOI = {10.1088/0305-4470/29/14/028},
       URL = {https://doi.org/10.1088/0305-4470/29/14/028},
}

@article {KZ1997,
    AUTHOR = {Kargaev, Pavel and Zhigljavsky, Anatoly},
     TITLE = {Asymptotic distribution of the distance function to the
              {F}arey points},
   JOURNAL = {J. Number Theory},
  FJOURNAL = {Journal of Number Theory},
    VOLUME = {65},
      YEAR = {1997},
    NUMBER = {1},
     PAGES = {130--149},
      ISSN = {0022-314X,1096-1658},
   MRCLASS = {11J83 (11K60)},
  MRNUMBER = {1458209},
MRREVIEWER = {G\'erald\ Tenenbaum},
       DOI = {10.1006/jnth.1997.2110},
       URL = {https://doi.org/10.1006/jnth.1997.2110},
}

@article {Hall1970,
    AUTHOR = {Hall, R. R.},
     TITLE = {A note on {F}arey series},
   JOURNAL = {J. London Math. Soc. (2)},
  FJOURNAL = {Journal of the London Mathematical Society. Second Series},
    VOLUME = {2},
      YEAR = {1970},
     PAGES = {139--148},
      ISSN = {0024-6107,1469-7750},
   MRCLASS = {10.07},
  MRNUMBER = {253978},
MRREVIEWER = {H.\ London},
       DOI = {10.1112/jlms/s2-2.1.139},
       URL = {https://doi.org/10.1112/jlms/s2-2.1.139},
}

@article {BZ2005,
    AUTHOR = {Boca, Florin P. and Zaharescu, Alexandru},
     TITLE = {The correlations of {F}arey fractions},
   JOURNAL = {J. London Math. Soc. (2)},
  FJOURNAL = {Journal of the London Mathematical Society. Second Series},
    VOLUME = {72},
      YEAR = {2005},
    NUMBER = {1},
     PAGES = {25--39},
      ISSN = {0024-6107,1469-7750},
   MRCLASS = {11J71 (11B57 11K36)},
  MRNUMBER = {2145726},
MRREVIEWER = {Dmitry\ Y.\ Kleinbock},
       DOI = {10.1112/S0024610705006629},
       URL = {https://doi.org/10.1112/S0024610705006629},
}

@article {ABCZ,
    AUTHOR = {Anderson, Jack and Boca, Florin P. and Cobeli, Cristian and
              Zaharescu, Alexandru},
     TITLE = {Distribution of angles to lattice points seen from a fast
              moving observer},
   JOURNAL = {Res. Number Theory},
  FJOURNAL = {Research in Number Theory},
    VOLUME = {10},
      YEAR = {2024},
    NUMBER = {3},
     PAGES = {Paper No. 62, 31},
      ISSN = {2522-0160,2363-9555},
   MRCLASS = {11P21 (11B05 11K06)},
  MRNUMBER = {4761452},
MRREVIEWER = {Donald\ Jason\ Gibson},
       DOI = {10.1007/s40993-024-00548-z},
       URL = {https://doi.org/10.1007/s40993-024-00548-z},
}

@incollection {Marklof2013,
    AUTHOR = {Marklof, Jens},
     TITLE = {Fine-scale statistics for the multidimensional {F}arey
              sequence},
 BOOKTITLE = {Limit theorems in probability, statistics and number theory},
    SERIES = {Springer Proc. Math. Stat.},
    VOLUME = {42},
     PAGES = {49--57},
 PUBLISHER = {Springer, Heidelberg},
      YEAR = {2013},
      ISBN = {978-3-642-36067-1; 978-3-642-36068-8},
   MRCLASS = {11B57 (11H06 37D40)},
  MRNUMBER = {3079137},
MRREVIEWER = {Dieter\ H.\ Mayer},
       DOI = {10.1007/978-3-642-36068-8\_3},
       URL = {https://doi.org/10.1007/978-3-642-36068-8_3},
}

@article {MarStro,
    AUTHOR = {Marklof, Jens and Str\"{o}mbergsson, Andreas},
     TITLE = {The distribution of free path lengths in the periodic
              {L}orentz gas and related lattice point problems},
   JOURNAL = {Ann. of Math. (2)},
  FJOURNAL = {Annals of Mathematics. Second Series},
    VOLUME = {172},
      YEAR = {2010},
    NUMBER = {3},
     PAGES = {1949--2033},
      ISSN = {0003-486X},
   MRCLASS = {37D50 (37A17 37A60 60F17 82C41)},
  MRNUMBER = {2726104},
MRREVIEWER = {Nikolai Chernov},
       DOI = {10.4007/annals.2010.172.1949},
       URL = {https://doi.org/10.4007/annals.2010.172.1949},
}

@incollection {DM93,
    AUTHOR = {Dani, S. G. and Margulis, G. A.},
     TITLE = {Limit distributions of orbits of unipotent flows and values of
              quadratic forms},
 BOOKTITLE = {I. {M}. {G}el\cprime fand {S}eminar},
    SERIES = {Adv. Soviet Math.},
    VOLUME = {16, Part 1},
     PAGES = {91--137},
 PUBLISHER = {Amer. Math. Soc., Providence, RI},
      YEAR = {1993},
      ISBN = {0-8218-4118-1},
   MRCLASS = {22E40 (11H55 58F11)},
  MRNUMBER = {1237827},
MRREVIEWER = {Nimish\ A.\ Shah},
}

@article {Ratner91A,
    AUTHOR = {Ratner, Marina},
     TITLE = {Raghunathan's topological conjecture and distributions of
              unipotent flows},
   JOURNAL = {Duke Math. J.},
  FJOURNAL = {Duke Mathematical Journal},
    VOLUME = {63},
      YEAR = {1991},
    NUMBER = {1},
     PAGES = {235--280},
      ISSN = {0012-7094,1547-7398},
   MRCLASS = {22E40 (22D40 28D10)},
  MRNUMBER = {1106945},
MRREVIEWER = {Gopal\ Prasad},
       DOI = {10.1215/S0012-7094-91-06311-8},
       URL = {https://doi.org/10.1215/S0012-7094-91-06311-8},
}

@incollection {KM1,
    AUTHOR = {Kleinbock, D. Y. and Margulis, G. A.},
     TITLE = {On effective equidistribution of expanding translates of
              certain orbits in the space of lattices},
 BOOKTITLE = {Number theory, analysis and geometry},
     PAGES = {385--396},
 PUBLISHER = {Springer, New York},
      YEAR = {2012},
   MRCLASS = {37A17 (11J83 37A25)},
  MRNUMBER = {2867926},
MRREVIEWER = {Thomas Ward},
       DOI = {10.1007/978-1-4614-1260-1\_18},
       URL = {https://doi.org/10.1007/978-1-4614-1260-1_18},
}

@book {Margulis,
    AUTHOR = {Margulis, Grigoriy A.},
     TITLE = {On some aspects of the theory of {A}nosov systems},
    SERIES = {Springer Monographs in Mathematics},
      NOTE = {With a survey by Richard Sharp: Periodic orbits of hyperbolic
              flows,
              Translated from the Russian by Valentina Vladimirovna
              Szulikowska},
 PUBLISHER = {Springer-Verlag, Berlin},
      YEAR = {2004},
     PAGES = {vi+139},
      ISBN = {3-540-40121-0},
   MRCLASS = {37D20 (37C27 37C30 37C40)},
  MRNUMBER = {2035655},
MRREVIEWER = {Boris\ Hasselblatt},
       DOI = {10.1007/978-3-662-09070-1},
       URL = {https://doi.org/10.1007/978-3-662-09070-1},
}

@incollection {Hooley4,
    AUTHOR = {Hooley, Christopher},
     TITLE = {On the intervals between consecutive terms of sequences},
 BOOKTITLE = {Analytic number theory ({P}roc. {S}ympos. {P}ure {M}ath.,
              {V}ol. {XXIV}, {S}t. {L}ouis {U}niv., {S}t. {L}ouis, {M}o.,
              1972)},
    SERIES = {Proc. Sympos. Pure Math.},
    VOLUME = {Vol. XXIV},
     PAGES = {129--140},
 PUBLISHER = {Amer. Math. Soc., Providence, RI},
      YEAR = {1973},
   MRCLASS = {10L99},
  MRNUMBER = {384742},
MRREVIEWER = {S.\ L. G. Choi},
}

@article {Hooley2,
    AUTHOR = {Hooley, Christopher},
     TITLE = {On the difference between consecutive numbers prime to {$n$}.
              {II}},
   JOURNAL = {Publ. Math. Debrecen},
  FJOURNAL = {Publicationes Mathematicae Debrecen},
    VOLUME = {12},
      YEAR = {1965},
     PAGES = {39--49},
      ISSN = {0033-3883,2064-2849},
   MRCLASS = {10.42},
  MRNUMBER = {186641},
MRREVIEWER = {P.\ Erd\H os},
       DOI = {10.5486/pmd.1965.12.1-4.06},
       URL = {https://doi.org/10.5486/pmd.1965.12.1-4.06},
}

@article {Hooley3,
    AUTHOR = {Hooley, Christopher},
     TITLE = {On the difference between consecutive numbers prime to {$n$}.
              {III}},
   JOURNAL = {Math. Z.},
  FJOURNAL = {Mathematische Zeitschrift},
    VOLUME = {90},
      YEAR = {1965},
     PAGES = {355--364},
      ISSN = {0025-5874,1432-1823},
   MRCLASS = {10.42},
  MRNUMBER = {183702},
MRREVIEWER = {P.\ Erd\H os},
       DOI = {10.1007/BF01112354},
       URL = {https://doi.org/10.1007/BF01112354},
}

@article {Hooley1,
    AUTHOR = {Hooley, Christopher},
     TITLE = {On the difference of consecutive numbers prime to {$n$}},
   JOURNAL = {Acta Arith.},
  FJOURNAL = {Polska Akademia Nauk. Instytut Matematyczny. Acta Arithmetica},
    VOLUME = {8},
      YEAR = {1962/63},
     PAGES = {343--347},
      ISSN = {0065-1036},
   MRCLASS = {10.40},
  MRNUMBER = {155807},
MRREVIEWER = {P.\ Erd\H os},
       DOI = {10.4064/aa-8-3-343-347},
       URL = {https://doi.org/10.4064/aa-8-3-343-347},
}

@article {Shah96,
    AUTHOR = {Shah, Nimish A.},
     TITLE = {Limit distributions of expanding translates of certain orbits
              on homogeneous spaces},
   JOURNAL = {Proc. Indian Acad. Sci. Math. Sci.},
  FJOURNAL = {Indian Academy of Sciences. Proceedings. Mathematical
              Sciences},
    VOLUME = {106},
      YEAR = {1996},
    NUMBER = {2},
     PAGES = {105--125},
      ISSN = {0253-4142,0973-7685},
   MRCLASS = {22E40},
  MRNUMBER = {1403756},
MRREVIEWER = {Dave\ Witte\ Morris},
       DOI = {10.1007/BF02837164},
       URL = {https://doi.org/10.1007/BF02837164},
}

@article{ShahYang,
title = {Limit distributions of expanding translates of shrinking submanifolds and non-improvability of Dirichlet's approximation theorem},
journal = {Journal of Modern Dynamics},
volume = {19},
number = {0},
pages = {947-965},
year = {2023},
issn = {1930-5311},
doi = {10.3934/jmd.2023027},
url = {https://www.aimsciences.org/article/id/655dc43ac8d8053960399ebb},
author = {Nimish A. Shah and Pengyu Yang}
}

@article{Hejhal,
  author  = {Hejhal, Dennis A.},
  title   = {On the uniform equidistribution of long closed horocycles},
  journal = {Asian Journal of Mathematics},
  volume  = {4},
  number  = {4},
  pages   = {839--853},
  year    = {2000}
}

\end{document}